\documentclass{article}
\usepackage{amssymb}
\usepackage{amsmath}
\usepackage{arxiv}
\usepackage[utf8]{inputenc}
\usepackage[T1]{fontenc}
\usepackage{url}
\usepackage{booktabs}
\usepackage{amsfonts}
\usepackage{nicefrac}
\usepackage{microtype}
\usepackage{color}
\usepackage{lipsum}
\usepackage{amsthm}
\usepackage{indentfirst}
\usepackage{graphicx}
\usepackage{epstopdf}
\usepackage{fourier}
\usepackage{bm}
\usepackage{mathtools}
\usepackage{latexsym,enumerate}
\usepackage{multicol}
\usepackage[skip=2pt]{caption}
\usepackage[font=small,skip=0pt]{subcaption}
\usepackage{array}

\usepackage{mathtools}
\usepackage{arydshln}

\usepackage[pdfstartview=FitH, CJKbookmarks=true,
bookmarksnumbered=true, bookmarksopen=true,
colorlinks,
linkcolor=red,
anchorcolor=blue, citecolor=blue
]{hyperref}

\newcommand{\comment}[1]{}

\newcommand{\BEA}{\begin{eqnarray}}
\newcommand{\EEA}{\end{eqnarray}}

\newtheorem{thm}{Theorem}[section]

\newtheorem{lem}[thm]{Lemma}

\newtheorem{defn}[thm]{Definition}

\newtheorem{algm}[thm]{Algorithm}

\newcommand{\PreserveBackslash}[1]{\let\temp=\\#1\let\\=\temp}
\newcolumntype{C}[1]{>{\PreserveBackslash\centering}p{#1}}
\newcolumntype{R}[1]{>{\PreserveBackslash\raggedleft}p{#1}}
\newcolumntype{L}[1]{>{\PreserveBackslash\raggedright}p{#1}}

\begin{document}

\title{Kernel-based methods for Solving Time-Dependent Advection-Diffusion Equations on Manifolds}
\author{ Qile Yan \\
Department of Mathematics, The Pennsylvania State University, University
Park, PA 16802, USA\\
\texttt{qzy42@psu.edu}\\
\And Shixiao Willing Jiang \\
Institute of Mathematical Sciences, ShanghaiTech University, Shanghai
201210, China\\
\texttt{jiangshx@shanghaitech.edu.cn} \\
\And John Harlim \\
Department of Mathematics, Department of Meteorology and Atmospheric
Science, \\
Institute for Computational and Data Sciences \\
The Pennsylvania State University, University Park, PA 16802, USA\\
\texttt{jharlim@psu.edu} }
\date{\today}
\maketitle

\begin{abstract}
In this paper, we extend the class of kernel methods, the so-called diffusion maps (DM) and ghost point diffusion maps (GPDM), to solve the time-dependent advection-diffusion PDE on unknown smooth manifolds without and with boundaries. The core idea is to directly approximate the spatial components of the differential operator on the manifold with a local integral operator and combine it with the standard implicit time difference scheme. When the manifold has a boundary, a simplified version of the GPDM approach is used to overcome the bias of the integral approximation near the boundary. The Monte-Carlo discretization of the integral operator over the point cloud data gives rise to a mesh-free formulation that is natural for randomly distributed  points, even when the manifold is embedded in high-dimensional ambient space. Here, we establish the convergence of the proposed solver on appropriate topologies, depending on the distribution of point cloud data and boundary type. We provide numerical results to validate the convergence results on various examples that involve simple geometry and an unknown manifold. Additionally, we also found positive results in solving the one-dimensional viscous Burger's equation where GPDM is adopted with a pseudo-spectral Galerkin framework to approximate nonlinear advection term.
\end{abstract}



\lhead{Local Kernels for Parabolic PDE's}

\newpage

\section{Introduction}

Solving Partial Differential Equations (PDEs) on manifolds arises naturally in the modeling of physical phenomena, including granular flow \cite{rauter2018finite}, liquid crystal \cite{virga2018variational}, biomembranes \cite{elliott2010modeling}, and computer graphics \cite{bertalmio2001variational}.  In computer graphics applications, the problem of solving PDEs on surfaces has been formulated to restore damaged patterns \cite{macdonald2010implicit}, brain imaging \cite{memoli2004implicit}, among other applications. Techniques for solving PDEs, especially on two-dimensional surfaces, have been proposed. Most of these methods, however, require a parameterization of the surface, which is subsequently used to approximate the tangential derivatives along the surface.

For example, the finite element method (FEM) parameterizes the surface \cite{dziuk2013finite,camacho,bonito2016high} using triangular meshes and solve the Galerkin projected approximation on the finite-element space of functions defined on the triangular meshes. Thus its accuracy relies on the quality of the generated meshes which can be a complicated task if the given point clouds data is randomly distributed. Another class of approaches is to solve the embedded PDE on the ambient space, where the embedding function can be estimated using various techniques, such as the level set representation \cite{bertalmio2001variational} or closest point representation \cite{ruuth2008simple}. Even when the challenges in estimating the embedding function can be overcome, the key issue with this class of approaches is that since the embedded PDE is at least one dimension higher than the dimension of the two-dimensional surface (i.e., co-dimension higher than one), the computational cost may not be feasible if the manifold is embedded in high-dimensional ambient space. Another class of approaches is the mesh-free radial basis function (RBF) method. While several versions of RBF solvers have been proposed \cite{piret2012orthogonal,fuselier2013high} for solving PDE on surfaces, the shape parameter of the RBF can be difficult to tune for high co-dimensional problems as pointed out in \cite{fuselier2013high} and the convergence near the boundary can be problematic.

In this paper, we consider solving a time-dependent advection-diffusion PDE defined on manifolds by approximating the intrinsic tangential (spatial) derivatives using a class of Gaussian integral operators, diffusion maps (DM) algorithm \cite{coifman2006diffusion}, that was designed and introduced as a nonlinear manifold learning. Computationally, this mesh-free algorithm does not require a parameterization and can naturally handle embedded submanifolds of arbitrary co-dimensions and randomly sampled point cloud data. When the manifolds have either no boundary or when the operator only admits homogeneous Neumann boundary condition, the DM approximation (applied with non-symmetric Gaussian kernels \cite{berry2016local}) was shown to be an accurate scheme for solving elliptic PDEs associated with the backward Kolmogorov differential operators with convergence guarantees in the uniform sense \cite{gh2019}. Recently, the Ghost Point Diffusion Maps
(GPDM) \cite{jiang2020ghost} has been introduced to handle other types of boundary conditions, including non-homogeneous Dirichlet, Neumann, and Robin types. The basic idea of GPDM is to remove biases of the integral approximation near the boundary by supplementing the unknown manifold (which is only identified by point clouds) with a set of ghost points on the exterior collar along the boundary. This nontrivial extension was shown to be effective for solving elliptic PDEs with convergence guarantees in a uniform sense.

Building upon this encouraging result, we now consider solving for $u:M\times (0,T] \to \mathbb{R}$ that satisfies a time-dependent advection-diffusion equation,
\begin{eqnarray}
u_{t} &=& a\cdot \nabla _{g}+\Delta _{g}u+f,\quad \quad  (x,t)\in M\times (0,T],
\notag \\
u(\cdot,0) &=& u_{0},\quad\quad \quad \quad\quad\quad\quad \quad\,\,\,\, x\in M,
\label{Eqn:utbf} \\
\beta \partial _{\mathbf{\nu }}u+\alpha u&=& g,\quad\quad \quad\quad\quad\quad\quad \quad\quad (x,t)\in \partial M\times \lbrack 0,T],  \notag
\end{eqnarray}%
In \eqref{Eqn:utbf}, $\Delta _{g}$ is the negative-definite, Laplace-Beltrami operator defined on $M$, the gradient operator $\nabla _{g}$ and the dot product $\cdot $ are defined with respect to the Riemannian metric inherited by $M$ from $\mathbb{R}^{m}$, the vector field $a:M\rightarrow \mathbb{R}^{d}$ is a space-dependent drift along tangent space of $M$. In \eqref{Eqn:utbf}, we use the notation $x\in M \subseteq\mathbb{R}^m$ to denote points on the manifold with ambient coordinate representation (or sometimes called the point cloud) and $t\in \lbrack 0,T]$ to denote the finite time interval. The forcing term $f:M\times \lbrack 0,T]\rightarrow\mathbb{R}$ is a function of $x\in M$ and time $t$ that is independent of the quantity $u$. The second equation in \eqref{Eqn:utbf} is the initial condition where  $u_{0}:M\rightarrow \mathbb{R}$ is the given function. The third equation in (\ref{Eqn:utbf}) is the boundary condition if manifold $M$ has a boundary $\partial M$ (if $M$ has no boundary, then this equation can be ignored). On the boundary $\partial M$, we consider either $\left( \alpha ,\beta \right) =(1,0)$ for Dirichlet boundary condition or $\left( \alpha ,\beta \right) =(0,1)$ for Neumann boundary condition, where $\boldsymbol{\nu }$ is the exterior normal direction and $\partial _{\boldsymbol{\nu }}$ is the normal derivative, and time-independent $g:M\to\mathbb{R}$ denotes possibly non-homogeneous boundary conditions.

The main contribution in this paper is to develop a solver for the parabolic-type problem in \eqref{Eqn:utbf} that combines GPDM (or DM for closed manifolds) as a spatial discretization of the tangential derivatives on the manifolds and the backward difference temporal scheme. In particular, we will provide a convergence study and numerical verifications on various test examples with known and unknown manifolds, including randomly distributed point clouds. In addition, we also provide a numerical method to solve nonlinear viscous Burger's equation on a one-dimensional manifold, where GPDM is adopted with a pseudo-spectral Galerkin formulation to approximate nonlinear advection term.

The paper will be organized as follows. In Section~\ref{section2}, we provide a short review of the approximation of tangential derivatives using local kernel methods. We will discuss the classical diffusion maps for manifolds without boundaries and the ghost point diffusion maps for manifolds with Dirichlet and Neumann boundaries. In Section~\ref{section3}, we present the proposed solvers, combining DM and GPDM with the backward time difference, and provide the convergence analysis. In Section~\ref{section4}, we provide numerical examples. In Section~\ref{section5}, we discuss a GPDM-based pseudo-spectral approach for solving viscous nonlinear Burger's equation. We close the paper with a summary and a list of open problems in Section~\ref{section6}. To improve the readability, we review the detailed algorithm for estimating normal vectors in Appendix~\ref{App:A} and include the longer proofs in several Appendices.

\section{Local kernel approximation of spatial derivatives}\label{section2}

In this section, we first review the relevant results for the local kernel
method from \cite{berry2016local,gh2019,jiang2020ghost} that approximates the
differential operator $\mathcal{L}:=a\cdot \nabla_g+\Delta_g$\ in (\ref{Eqn:utbf}) with a local
integral operator. In a nutshell, the local kernel method is a generalization of
the diffusion maps algorithm \cite{coifman2006diffusion,bh:16vb} to approximate the second-order elliptic operator with a variable-coefficient
diffusion and a non-symmetric drift \cite{berry2016local,gh2019}. Next, we
review the Ghost Point Diffusion Maps method \cite{jiang2020ghost} for approximating the operator on manifolds with boundaries. Here, we will consider a simpler version of GPDM (a version that was used to solve eigenvalue problem in \cite{jiang2020ghost}) and present the consistency error bound.
\subsection{Review of local kernel theory}

\label{basicDMtheory}

Based on local kernel theory \cite{berry2016local,gh2019}, one can
approximate the Kolmogorov Backward (KB) operator,
\begin{equation}
\mathcal{L}_{\text{KB}}u:=a\cdot \nabla _{g}u+\frac{1}{2}c^{ij}\nabla
_{i}\nabla _{j}u,  \label{Eqn:L3}
\end{equation}%
where $\nabla _{i}$ is the covariant derivative in the $i$th direction, $%
\nabla _{i}\nabla _{j}$ are the components of the Hessian operator, $a:%
M\to\mathbb{R}^d$ is the same advection term as in %
\eqref{Eqn:utbf}, and the symmetric positive definite diffusion tensor $%
c:M\rightarrow \mathbb{R}^{d}\times \mathbb{R}^{d}$ is a $d\times d$\
diffusion matrix. One can define the prototypical local kernel \cite%
{berry2016local}, $K_{\epsilon }:\mathbb{R}^{m}\times \mathbb{R}%
^{m}\rightarrow \mathbb{R}$, as,
\begin{equation}
K_{\epsilon }(x,y):=\exp \left( -\frac{(x+\epsilon A(x)-y)^{\top
}C(x)^{-1}(x+\epsilon A(x)-y)}{2\epsilon }\right),
\label{prototypicalkernel}
\end{equation}%
where $\epsilon >0$ characterizes the kernel bandwidth. In %
\eqref{prototypicalkernel}, $A:M\rightarrow \mathbb{R}^{m}$ and $%
C:M\rightarrow \mathbb{R}^{m}\times \mathbb{R}^{m}$ are related to $a$ and $%
c $, respectively, through {\color{black}a local parameterization $%
\iota:U\subseteq \mathbb{R}^d\rightarrow M\subseteq\mathbb{R}^m$ of the
manifold $M$} as follows:
\begin{equation}
A\left( x\right) =D\iota \left( x\right) a\left( x\right) ,\text{ \ \ }%
C(x)^{-1}=\left( D\iota \left( x\right) c\left( x\right) D\iota \left(
x\right) ^{{\top }}\right) ^{\dag }.  \label{Eqn:BC}
\end{equation}%
{\color{black}Here, the set $U\subseteq \mathbb{R}^d$ denotes a domain that
contains $\iota^{-1}(x)$.} The notation $^{\dag }$\ denotes the
pseudo-inverse and the {\color{black}differential map $D\iota \left(
x\right): T_{\iota^{-1}(x)}M \subseteq \mathbb{R}^{d} \to T_x\mathbb{R}^{m}
\subseteq \mathbb{R}^{m}$ is an $m\times d$ matrix that is usually known as
the Jacobian matrix (or pushforward) corresponding to the map $\iota$.}

For all $x\in M$ whose distance from the boundary is larger than $\epsilon
^{r}$, where $0<r<1/2$, define the integral operator,
\begin{equation}
G_{\epsilon }u(x):=\epsilon ^{-d/2}\int_{M}K_{\epsilon }\left( x,y\right)
u(y)dV_{y}=\epsilon ^{-d/2}\int_{M_{\epsilon ,x}}K_{\epsilon }\left(
x,y\right) u(y)dV_{y}+O(\epsilon ^{2}),  \label{integralop1}
\end{equation}%
which is effectively a local integral operator over the $\epsilon ^{r}$-ball
around $x$, $M_{\epsilon ,x}:=\{y\in M,|y-x|\leq\epsilon ^{r}\}$. The main
result from \cite{berry2016local,gh2019,jiang2020ghost} is that for any $%
u\in C^{3}(M)$, the asymptotic expansion,
\begin{equation}
G_{\epsilon }u\left( x\right) =m\left( x\right) u\left( x\right) +\epsilon
\left( \omega \left( x\right) u\left( x\right) + m(x)\mathcal{L}_{\text{KB}%
}u\left( x\right) \right) +O\left( \epsilon ^{2}\right) ,  \label{Eqn:Geu}
\end{equation}%
holds for points $x$ whose distance from the boundary is at least of order-$%
\epsilon ^{r}$ with $0<r<1/2$ or for closed manifolds. Here, $\omega \left(
x\right) $ depends on the geometry of the manifold and $m\left( x\right)
=\left( 2\pi \right) ^{d/2}\det \left( C\left( x\right) \right) ^{1/2}$
can be approximated by $G_{\epsilon }1(x)=m(x)+O(\epsilon )$.

It is clear that the advection-diffusion operator,
\BEA
\mathcal{L}:=a\cdot \nabla_g+\Delta_g\label{Eqn:oper3}
\EEA
in \eqref{Eqn:utbf} is a
special case of \eqref{Eqn:L3} with $c_{ij} = 2g^{ij}$, where $g^{ij}$
denotes the $(i,j)$-component of the inverse of the metric tensor $g_{ij}$.
Based on this observation, all the results above are valid by simply
changing $C(x)^{-1} = \frac{1}{2}\mathbf{I}_{m\times m}$ in the local kernel %
\eqref{prototypicalkernel}. For the convenience of the discussion below, we
will abuse the notation and refer to,
\begin{equation}
K_{\epsilon }(x,y):=\exp \left( -\frac{\|x+\epsilon A(x)-y\|^{2}}{4\epsilon }%
\right),  \label{Eqn:Kexy}
\end{equation}
as the local kernel corresponding to advection-diffusion operator in %
\eqref{Eqn:oper3} instead of to Kolmogorov Backward operator in %
\eqref{Eqn:L3}. In \eqref{Eqn:Kexy} and the rest of this paper, the notation
$\|\cdot\|$ denotes the standard Euclidean norm.

In the advection-diffusion case, $m(x)$ is independent of $x$ since $\det
\left( C\left( x\right) \right) ^{1/2}$ is constant, and we will denote it as $%
m_{0}$ in the rest of this paper. While one should use the normalized graph
Laplacian formulation to remove the constant scaling factor $m_{0}$ in the
numerics, for the convenience of the convergence analysis in the rest of
this paper, we consider the un-normalized Graph Laplacian formulation (see
e.g.~\cite{von2008consistency,vaughn2019diffusion}) . Particularly, from the
integral operator in \eqref{Eqn:Geu} with the local kernel in %
\eqref{Eqn:Kexy}, we can access the operator $\mathcal{L}$ in %
\eqref{Eqn:oper3} through the following algebraic expression,
\begin{equation*}
L_{\epsilon }u(x):=\frac{1}{m_{0}\epsilon }\Big(G_{\epsilon
}u(x)-u(x)G_{\epsilon }1(x)\Big)=\mathcal{L}u(x)+O(\epsilon ).
\end{equation*}

Numerically, we discretize this integral equation as follows. Given point cloud data $\{x_{i}\}_{i=1,\ldots ,N}$, we define $\mathbf{K}\in \mathbb{R%
}^{N\times N}$ to be the matrix whose components are $\mathbf{K}%
_{ij}=K_{\epsilon }(x_{i},x_{j})$. Subsequently, let $\mathbf{D}$ be a
diagonal matrix with diagonal components consists of the components of $%
\mathbf{K}\mathbf{1}$. Then, the discrete estimator is denoted by,
\begin{equation}
\mathbf{L}:=\frac{\epsilon ^{-d/2-1}}{m_{0}N}\big(\mathbf{K}-\mathbf{D}\big).
\label{classicDMmatrix}
\end{equation}

We should point out that in the discussion above, we only present the
formulation for uniformly distributed data. With a slight modification, as
discussed in \cite{gh2019}, one can achieve a consistent estimator even if
the data is not sampled uniformly. Let $\mathbf{u}_{M}= (u(x_{1}),\ldots
,u(x_{N}))$ be a column vector whose components are the function values of $u$
on the given points on manifold $\{x_{i}\}_{i=1,\ldots ,N} \subseteq M$. Then, one has
the pointwise error estimate for this operator approximation as stated in
Ref. \cite{berry2016local,gh2019}:

\begin{lem}
\label{thm_LK} (Pointwise consistency) Let $x_{i}\in M\subset \mathbb{R}^{m}$
for $i=1,\ldots ,N$ be i.i.d. uniform samples with respect to the volume
form inherited by the d-dimensional manifold $M\subseteq\mathbb{R}^{m}$. For
any $u\in C^{3}\left( {M}\right) \cap L^{2}\left( {M}\right)$, and for all
points $x_{i}\in M$ when $M$ has no boundary or for points $x_{i}\in M$\
away from the boundary with distance of at least $O(\epsilon ^{r})$
with $0<r<1/2$ when M has boundary,
\begin{equation}
\left\vert (\mathbf{L}\mathbf{u}_{M})_{i}-\mathcal{L}u(x_{i})\right\vert =%
O\left( \epsilon,N^{-1/2}\epsilon ^{-(1/2+d/4)}\right) ,
\label{Eqn:cbc}
\end{equation}%
in high probability, as $\epsilon\to 0$ after $N\to\infty$. Here, $\mathbf{u}%
_M:=(u(x_1),\ldots,u(x_N)) \in \mathbb{R}^N$ is the column vector
corresponding to the given function $u$ at $\{x_i\}_{i=1}^N$.
\end{lem}

Throughout the paper, we use the notation $O(f,g)$ as a shorthand
for $O(f)+O(g)$ as $f,g \to 0$. The second error bound
in \eqref{Eqn:cbc} above corresponds to the discrete estimate of the
integral operator $L_\epsilon$ for fixed $\epsilon>0$ so it is defined as $%
N\to\infty$. Effectively, this means $\epsilon \to 0$ after $N\to \infty$ or
$\epsilon$ decays to zero as $N$ increases (which can be observed by
balancing these error terms).

\textbf{Practical Implementation:} To evaluate the kernel in \eqref{Eqn:Kexy}%
, we assume that $A:M\subseteq\mathbb{R}^m \to\mathbb{R}^m$ is available so
we can evaluate this function on any point cloud $x\in M\subseteq\mathbb{R}%
^m $. In Section~\ref{cownumericalsection}, we will discuss one way to
specify $A$ when it is not available. For accurate estimation, one also has
to specify the appropriate bandwidth parameter, $\epsilon$. In our
implementation, we also use $k-$nearest neighbor algorithm to avoid
computing the distances of pair of points that are sufficiently large.

Our choice of $\epsilon$ follows the method that was originally proposed in
\cite{coifman2008TuningEpsilon}. Basically, the idea relies on the following
observation,
\begin{eqnarray}
S(\epsilon):=\frac{1}{Vol(M)^2}\int_M \int_{T_xM} K_\epsilon(x,y) dy\,dV(x)
= \frac{1}{Vol(M)^2}\int_M (4\pi\epsilon)^{d/2} dV(x) = \frac{%
(4\pi\epsilon)^{d/2}}{Vol(M)}.  \label{scalingS}
\end{eqnarray}
Since $S$ can be approximated by a Monte-Carlo integral, for a fixed $k$, we
approximate,
\begin{eqnarray}
S(\epsilon) \approx \frac{1}{Nk}\sum_{i,j=1}^{N,k} \exp\Big(-\frac{%
\|x_i-x_j\|^2}{4\epsilon}\Big),  \notag
\end{eqnarray}
where $\{x_j\}_{j=1}^k$ is the $k$-nearest neighbors (knn) of each $x_i$. We
choose $\epsilon$ from a domain (e.g., $[2^{-14},10]$ in our numerical
implementation) such that $\frac{d\log(S)}{d\log\epsilon} \approx \frac{d}{2}
$. Numerically, we found that the maximum slope of $\log(S)$ often coincides
with $d/2$, which allows one to use the maximum value as an estimate for the
intrinsic dimension $d$ when it is not available, and then choose the
corresponding $\epsilon$. In our implementation, we will report the specific
choices of the knn parameter $k$ and the corresponding automated tuned
bandwidth parameter $\epsilon$. We should also point out that this bandwidth
tuning may not necessarily give the most accurate result (as noted in \cite%
{bh:16vb}), however, it gives a useful reference value for further tuning.
In some example, we will indeed use this automated tuned $\epsilon$ as a reference value for a particular sample size, $N$, and set the bandwidth parameters for different sample sizes, $N$, to respect the scaling
from the error estimate.


\textbf{The $\ell ^{2}$-topology:} For manifold with boundary,
unfortunately, we may not have point-wise consistency as in \eqref{Eqn:cbc}.
To facilitate the analysis, we will consider a weaker topology, that is, $%
\ell ^{2}:=(\mathbb{R}^{N},\Vert \cdot \Vert _{\ell ^{2}})$, defined with
the following norm,
\begin{equation}
\Vert \mathbf{v}\Vert _{\ell ^{2}}=\left( \frac{1}{N}\sum_{j=1}^{N}v_{j}^{2}%
\right) ^{1/2},  \label{Eqn:l2norm}
\end{equation}%
for any vector $\mathbf{v}=(v_{1},\ldots ,v_{N})^{\top }\in \mathbb{R}^{N}$.
We will also use the notation,
\begin{equation}
\langle \mathbf{u},\mathbf{v}\rangle _{\ell ^{2}}=\frac{1}{N}\mathbf{u}%
^{\top }\mathbf{v}=\frac{1}{N}\sum_{j=1}^{N}u_{j}v_{j},  \label{Eqn:inner}
\end{equation}%
to denote the corresponding inner-product for any $\mathbf{u},\mathbf{v}\in
\mathbb{R}^{N}$. Effectively, the $\ell ^{2}$ topology is the standard
Euclidean metric scaled by $1/N$, for any vector in $\mathbb{R}^{N}$. That
is, $\Vert \cdot \Vert _{\ell ^{2}}=\frac{1}{\sqrt{N}}\Vert \cdot \Vert $
and $\langle \cdot ,\cdot \rangle _{\ell ^{2}}=\frac{1}{N}\langle \cdot
,\cdot \rangle $.

\comment{
Under the same hypothesis of the Lemma~\ref{thm_LK}, it is also immediately
clear that
\begin{equation}
\Vert \mathbf{L}\mathbf{u}_{M}-\mathcal{L}u\Vert _{\ell ^{2}}=O\left(
\epsilon ,N^{-1/2}\epsilon ^{-(1/2+d/4)}\right) ,  \label{Eqn:l2_nob}
\end{equation}%
in high probability, as $\epsilon \rightarrow 0$ after $N\rightarrow \infty $%
.}

In the remainder of this paper, we will also use the notation $\|\cdot\|_{\infty}$ to denote the uniform (or maximum) norms for matrices and vectors.

\subsection{Review of the Ghost Point Diffusion Maps approach}

As mentioned in various Refs.
\cite{coifman2006diffusion,berry2016local,harlim2018,gh2019,jiang2020ghost},
the asymptotic expansion (\ref{Eqn:cbc}) is not valid near the boundary of
the manifold. One way to overcome this issue is with the ghost point approach introduced in \cite%
{jiang2020ghost}. This approach, called the Ghost Point Diffusion Maps (GPDM), extends the classical ghost point
method \cite{leveque2007finite} to solving elliptic PDEs on unknown manifolds with boundary.


The Ghost Point Diffusion Maps (GPDM), introduced in \cite{jiang2020ghost},
specifies the ghost points as data points that lie on the exterior normal
collar, $\Delta M$, along the boundary. For two-dimensional manifolds, it
was shown that under appropriate conditions, the extended set $M\cup \Delta
M $, can be isometrically embedded with an embedding function that is
consistent with the embedding $M\hookrightarrow \mathbb{R}^{m}$ when
restricted on $M$ (see Lemma~3.5 in \cite{jiang2020ghost}). For each
boundary point $x_{b}\in \partial M\subset \mathbb{R}^{m}$ with normal
vector $\boldsymbol{\nu }\in \mathbb{R}^{m}$, the embedding of $\Delta
M\hookrightarrow \mathbb{R}^{m}$ gives rise to the ghost points specified
as,
\begin{equation}
x_{b,k}:=x_{b}+kh\boldsymbol{\nu },\quad k=1,\ldots ,K,
\label{idealghostpoints}
\end{equation}%
for each $b=1,\ldots ,B$, where $B$ denotes the total number of given boundary
points, $K$ denotes the total number of ghost layers for each boundary point, and
the parameter $h$ denotes the distance between consecutive ghost points (see
Fig.~\ref{fig1dsketch}(a) for a geometric illustration). Practically, we
choose $h$ and ${K}$ such that $h{K}=O(\epsilon ^{r})$ and
all interior points whose distance are within $\epsilon ^{r}$ from the
boundary $\partial M$ are at least $\epsilon ^{r}$ away from the boundary of
the extended manifold $M\cup \Delta M$. In addition, we will also specify
interior ghost points,
\begin{equation*}
x_{b,0}:=x_{b}-h\boldsymbol{\nu },
\end{equation*}%
which will be used to estimate the normal derivatives of the boundary
conditions and estimate the function values on the exterior ghost points, $%
\{x_{b,k}\}_{b,k=1}^{{B},{K}}$.

\textbf{Estimation of normal vectors:} Since the normal vector $\boldsymbol{%
\nu}$ is unknown, it will be numerically estimated. For instance, when the
data is well-sampled, where by well-sampled we mean the data points are well-ordered along intrinsic coordinates, one can identify $\tilde{\boldsymbol{\nu}}$ as the
tangent line approximation to $\boldsymbol{\nu}$ (see Fig.~\ref{fig1dsketch}
for a geometric illustration). In this case, the parameter scales as $h=%
O(\epsilon^{1/d})$ and the error is $|\boldsymbol{\nu}-\boldsymbol{\tilde{\nu}}|=O(h)$ (see Proposition 3.1 in \cite{jiang2020ghost}). For randomly sampled data, one can also use
the kernel method to estimate $\boldsymbol{\tilde{\nu}}$ and the
error is $|\boldsymbol{\nu}-\boldsymbol{\tilde{\nu}}|=O(\sqrt{%
\epsilon})$ (see Appendix~\ref{App:A} for the detailed
discussion). The parameter $%
h $ can be estimated by the mean distance from $x_b$ to its $P$ (around 10 in simulations) nearest neighbors.
Given the estimated normal vector, $\boldsymbol{\tilde{\nu}}$, the
approximate ghost points are given by,
\begin{eqnarray}
\tilde{x}_{b,k} = x_b + kh\boldsymbol{\tilde{\nu}}, \quad k=1,\ldots,{K} .
\label{approxghostpoints}
\end{eqnarray}
We should point out that from \eqref{idealghostpoints} and \eqref{approxghostpoints}, we have $|%
\tilde{x}_{b,k}-x_{b,k}|=O(h^2)$ for well sampled data with the tangent line method and have $|%
\tilde{x}_{b,k}-x_{b,k}|=O(h\sqrt{\epsilon})$ for randomly sampled data with the kernel method,  for one and two-dimensional manifolds (see \cite{jiang2020ghost}).

As shown in Fig.~\ref{fig1dsketch}(b), when tangent line is used, the
interior estimated ghost point, $\tilde{x}_{b,0}=x_b - h\boldsymbol{\tilde{%
\nu}}$, coincides with one of interior points. However, for randomly sampled
data, the estimated interior ghost points will not necessarily coincide with an
interior point (see Fig.~\ref{fig1dsketch}(c)). In such case, if we
define $\gamma_b:[0,h]\to M$ as the arc-length geodesic parameterization
with based point $\gamma_b(0)=x_b$ and unit tangent vector $\gamma_b^{\prime
}(0) = -\boldsymbol{\nu}$, then $x_{b,0}:=x_b-h\boldsymbol{\nu}%
=\gamma_b(0)-h\gamma_b^{\prime }(0)$ and the distance between $%
\gamma_b(h):=\exp_{x_b}(-h\boldsymbol{\nu})$ and $x_{b,0}$ is of order-$h^2$, which is usually smaller
than $h\sqrt{\epsilon}$ in random sampled case. Then, it is clear that when the data is random, the distance between the
estimated ghost point $\tilde{x}_{b,0}$ and the manifold is given by,
\BEA
|\tilde{x}_{b,0} - \gamma_b(h)|\leq |\tilde{x}_{b,0}-x_{b,0}| +
|x_{b,0}-\gamma_b(h)| = O(h\sqrt{\epsilon}). \label{eqngamtid}
\EEA
In the discussion
below, we will show how this nonzero distance affects the consistency error in the case of randomly sampled data.


\begin{figure}[tbp]
\centering
\begin{tabular}{ccc}
(a) ideal construction & (b) secant line approximation & (c) kernel-based
approximation \\
\includegraphics[width=2in]{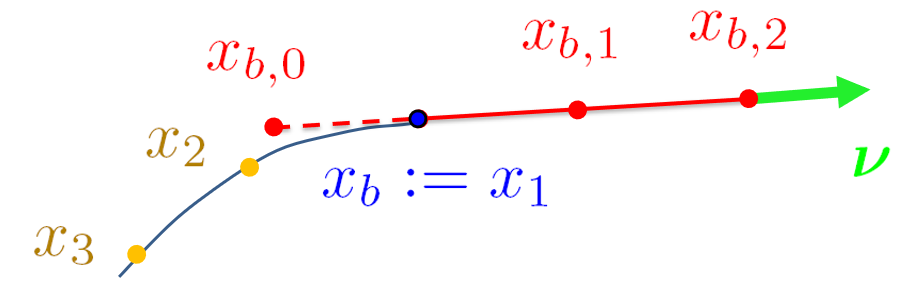} & %
\includegraphics[width=2in]{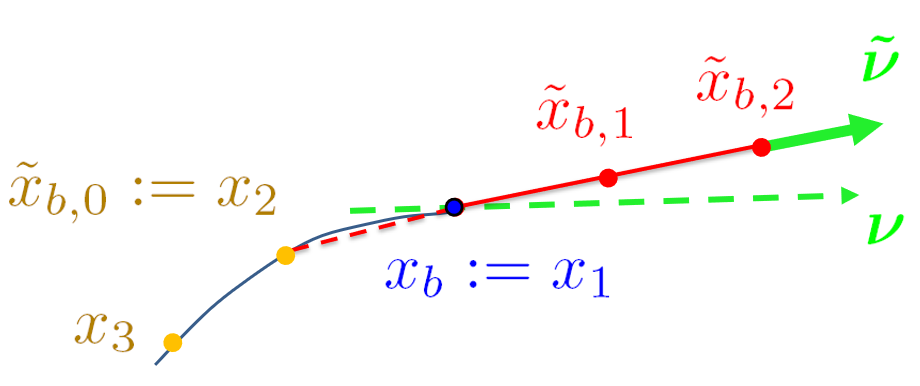} & %
\includegraphics[width=2in]{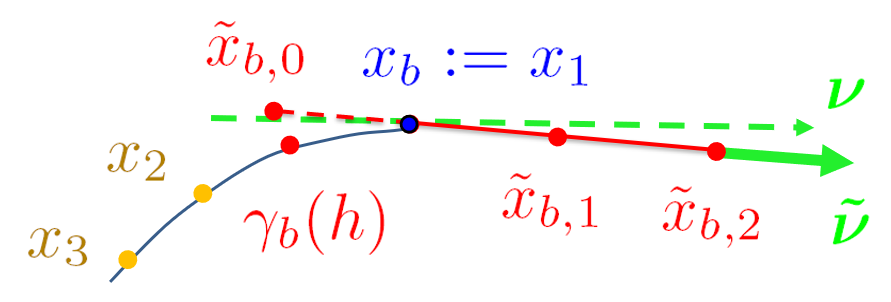}%
\end{tabular}%
\caption{Geometric illustration for ghost point extension: (a) Ideal
construction: ghost point extension for $x_{b,0},x_{b,1},x_{b,2}$ along true
$\boldsymbol{\protect\nu}$. Here, $x_2$ and $x_3$ are points on the manifold
$M$, $x_b:=x_1$ is a point on the boundary $\partial M$. (b) Secant line
extension for ghost points $\tilde{x}_{b,1}, \tilde{x}_{b,2}$ along the
estimated $\boldsymbol{\tilde{\protect\nu}}$. Here, $\boldsymbol{\tilde{%
\protect\nu}}$ is along the secant line connecting $x_2$ and $x_1$. (c)
Kernel-based approximation. Notice that $x_{b,0}$ does not coincide with $%
x_2 $. In fact, its distance from $\protect\gamma_b(h):=\exp_{x_b}(-h%
\boldsymbol{\protect\nu})$, a projected point on the manifold $M$, is of
order $h\protect\sqrt{\protect\epsilon}$. }
\label{fig1dsketch}
\end{figure}

\textbf{Estimation of function values on the ghost points:} The main goal
here is to estimate the function values $\{u(x_{b,k})\}_{b,k=1}^{{B},{K}}$,
given the
estimated ghost points $\{\tilde{x}_{b,k}\}_{b,k=1,0}^{{B},{K}}$, function values $u(x_{i})$ on manifold $x_{i}\in M$, and function
values $u(\tilde{x}_{b,0})$ on the estimated interior ghosts, where both
$\{u(x_{i}),u(\tilde{x}_{b,0})\}$ will be eventually approximated from solving a
system of linear algebraic equations in our PDE applications. We stress that
the function values $\{u(\tilde{x}_{b,0})\}_{b=1}^{B}$ here is given exactly
like the  $u(x_{i})$ for any $x_{i}\in M$, even when the ghost
points $\tilde{x}_{b,0}$ do not lie on the manifold $M$. For the discussion
below, we assume that we are given the components of the column vector,
\begin{equation}
\mathbf{u}_{M}:=(u(x_{1}),\ldots ,u(\tilde{x}_{b,0}),\ldots ,u(x_{N}))\in
\mathbb{R}^{N},  \label{umh}
\end{equation}%
and we will use them to estimate the components of function values on ghost
points,%
\begin{equation}
\mathbf{u}_{G}:=(u(x_{1,1}),\ldots ,u(x_{{B},{K}}))\in \mathbb{R}^{{B}{K}}.
\label{ug}
\end{equation}%
We denote the corresponding estimates as components of the column vector,
\begin{equation}
\mathbf{U}_{G}:=(U_{1,1},\ldots ,U_{{B},{K}})\in \mathbb{R}^{{B}{K}}.
\label{ugh}
\end{equation}%
For the convenience of discussion below, we also define the column vector,
\begin{equation}
\mathbf{\tilde{u}}_{G}:=(u(\tilde{x}_{1,1}),\ldots ,u(\tilde{x}_{{B},{K}%
}))\in \mathbb{R}^{{B}{K}}.  \label{eqn:util}
\end{equation}%
Numerically, we will obtain the components of $\mathbf{U}_{G}$ by solving
the following linear algebraic equations for each $b=1,\ldots ,{B}$,
\begin{equation}
\begin{aligned} U_{b,1} - 2u(x_{b})+ u(\tilde{x}_{b,0}) &= 0,\\
U_{b,2}-2U_{b ,1}+u(x_{b})&=0, \\ U_{b ,k}-2U_{b,k-1}+U_{b ,k-2}&=0,
\quad\quad k = 3,\ldots {K}. \\ \end{aligned}  \label{Eqn:uvv_g2}
\end{equation}%
These algebraic equations are discrete analogs of matching the first-order
derivatives along the estimated normal direction, $\tilde{\boldsymbol{\nu }}$%
.

To understand the error induced by the solution of \eqref{Eqn:uvv_g2}, we
define the following set.

\begin{defn}
\label{epsball} $B_{\epsilon^r}(\partial M) := \cup_{x\in\partial M}
B_{\epsilon^r}(x)$, where $B_{\epsilon^r}(x) = \{y\in\mathbb{R}^m:|y-x|\leq
\epsilon^r\}$ is an $\epsilon^r$-ball in $\mathbb{R}^m$.
\end{defn}

Note that since $\gamma _{b}(h),x_{b,k},\tilde{x}_{b,k}\in B_{\epsilon
^{r}}(\partial M)$, one can estimate the function values at these points
under the appropriate regularity assumption on this set. For example,
assuming that $u\in C^{1}(B_{\epsilon ^{r}}(\partial M))$, one can deduce
that,
\begin{equation}
|(\mathbf{U}_{G}-\mathbf{u}_{G})_{b,k}|\leq |(\mathbf{U}_{G}-\mathbf{\tilde{u%
}}_{G})_{b,k}|+|(\mathbf{\tilde{u}}_{G}-\mathbf{u}_{G})_{b,k}|=O%
(h^{2},h\sqrt{\epsilon }),  \label{ugerror}
\end{equation}%
where the column vectors are defined above in (\ref{ug})-(\ref{eqn:util}).
Here, the first error rate of order-$h^{2}$, which is valid for the
well-sampled data, can be obtained by the Taylor expansion along $\tilde{%
\boldsymbol{\nu }}$. The second error rate of order-$h\sqrt{\epsilon } $,
which is valid for randomly sampled data, can be obtained by the Taylor
expansion along a straight path of distance $h\sqrt{\epsilon }$, connecting $%
x_{b,k}$ and $\tilde{x}_{b,k}$, for every $b$ and $k$. It is worthwhile to
note that the first error bound of rate-$h^{2}$ is one-order less accurate
compared to that in \cite{jiang2020ghost}, which uses a set of more
complicated extrapolation equations (relative to \eqref{Eqn:uvv_g2}) that is
only appropriate for solving elliptic problems. The second error rate-$h%
\sqrt{\epsilon }$ will be comparable to the dominating error rate for
randomly sampled data (see the rates reported in Proposition~3.6 in \cite%
{jiang2020ghost}), where $h=O(\epsilon )$ in elliptic problems.

\textbf{The GPDM estimator:} In the following discussion, we will define the
Ghost Point Diffusion Maps (GPDM) estimator for the differential operator $%
\mathcal{L}$ in \eqref{Eqn:oper3}. Since the discrete estimator will be constructed based on the
available training data $\{x_i\in M\}_{i=1}^N$ and the estimated ghost
points, $\{\tilde{x}_{b,k}\}_{b,k=1,0}^{{B},{K}}$, we will need a few more
notations to demonstrate its consistency for the differential operator
defined on functions that takes values on the extended $M\cup \Delta M$.

In particular, since the interior ghost points, $\{\tilde{x}_{b,0}\}_{b=1}^{{%
B}}$ may or may not coincide with any interior points on the manifold,
in the following discussion, we assume that
\BEA
X^{h}  :=\{x_{1},\ldots ,\tilde{x}%
_{b,0},\ldots ,x_{N}\}, \label{eqnXh} 
\EEA
has $N$ components that include the estimated ghost
points. With this notation, we define a non-square matrix,
\begin{equation}
\mathbf{L}^{h}:=(\mathbf{L}^{(1)},\mathbf{L}^{(2)})\in \mathbb{R}^{N\times
\bar{N}},  \label{eqnLtild}
\end{equation}%
constructed as in \eqref{classicDMmatrix} with $\bar{N}=N+BK$, by evaluating
the kernel on components of $X^{h}$ for each row and the components of $%
X^{h}\cup \{\tilde{x}_{b,k}\}_{b,k=1}^{{B},{K}}$ for each column. Here, $%
\mathbf{L}^{(1)}\in \mathbb{R}^{N\times N}$ corresponds to the standard
construction in \eqref{classicDMmatrix} over the components of $X^{h}$. For
the discussion below, we also define the matrix $\mathbf{L}\in \mathbb{R}%
^{N\times \bar{N}}$, constructed in analogous to $\mathbf{L}^{h}$, except
that the kernel is evaluated on the components of $X$ for each row and the
components of $X\cup \{x_{b,k}\}_{b,k=1}^{{B},{K}}$ for each column, where
all elements of
\BEA
X :=\{x_{1},\ldots ,\gamma _{b}(h),\ldots ,x_{N}\}, \label{eqnX1} 
\EEA
 lie on
the manifold $M$, with $\gamma _{b}(h)\in M$ replacing each estimated
interior ghost point $\tilde{x}_{b,0}\in X^{h}$. One can show that $\mathbf{L%
}^{h}=\mathbf{L}+O(h^{2}\epsilon ^{-3/2})$ (see Lemma~A.2 in \cite%
{jiang2020ghost}).

The GPDM estimator is defined as an $N\times N$ matrix,
\begin{equation}
\mathbf{\tilde{L}}:=\mathbf{L}^{(1)}+\mathbf{L}^{(2)}\mathbf{G},
\label{GPDMmatrix}
\end{equation}%
where $\mathbf{G}\in \mathbb{R}^{{B}{K}\times N}$ is defined as a solution
operator to \eqref{Eqn:uvv_g2} which in compact form is given as, $\mathbf{U}%
_{G}=\mathbf{G}\mathbf{u}_{M}$. we also define column vectors%
\begin{eqnarray}
{\mathbf{u}_{M}^{\gamma }} &:=&(u(x_{1}),\ldots ,u(\gamma _{b}(h)),\ldots
,u(x_{N}))\in \mathbb{R}^{N},  \label{uM} \\
{\mathbf{u}} &:=&({\mathbf{u}}_{M}^{\gamma },{\mathbf{u}}_{G})\in \mathbb{R}%
^{\bar{N}},  \label{u} \\
{\mathbf{U}} &:=&({\mathbf{u}}_{M},{\mathbf{U}}_{G})\in \mathbb{R}^{\bar{N}},
\label{uh}
\end{eqnarray}
With these definitions and those in (\ref{umh})-(\ref{eqn:util}), we note
that
\begin{equation*}
\mathbf{L}^{h}\mathbf{U}=\mathbf{L}^{(1)}\mathbf{u}_{M}+\mathbf{L}^{(2)}%
\mathbf{U}_{G}=(\mathbf{L}^{(1)}+\mathbf{L}^{(2)}\mathbf{G})\mathbf{u}_{M}=%
\mathbf{\tilde{L}}\mathbf{u}_{M},
\end{equation*}%
but $\mathbf{L}^{h}\mathbf{u=L}^{(1)}\mathbf{u}_{M}^{\gamma }+\mathbf{L}%
^{(2)}\mathbf{u}_{G}\neq \mathbf{\tilde{L}}\mathbf{u}_{M}^{\gamma }$. With
all these notation above, we can show that:

\begin{lem}
\label{theorem1} (Consistency of the GPDM) Let $u\in C^{3}\left( M\cup
B_{\epsilon ^{r}}(\partial M)\right) $, where the extended manifold $M\cup
\Delta M$, a smooth submanifold of $\mathbb{R}^{n}$, is constructed such
that all $x\in M$ has distance of at least of $O(\epsilon ^{r})$%
, where $0<r<1/2$, from the boundary of the extended manifold, $M\cup \Delta
M$. For each $x_{i}\in X$, then:
\begin{enumerate}
\item For
well-sampled data,
\begin{equation*}
\left\vert (\mathbf{\tilde{L}}\mathbf{u}_{M})_{i}-\mathcal{L}%
u(x_{i})\right\vert =O\left( h^{2}\epsilon ^{-1},\epsilon ,\bar{N}%
^{-1/2}\epsilon ^{-(1/2+d/4)}\right),
\end{equation*}
\item and for randomly sampled data,
\begin{equation*}
\left\vert (\mathbf{\tilde{L}}\mathbf{u}_{M})_{i}-\mathcal{L}%
u(x_{i})\right\vert =O\left( h\epsilon ^{-1/2},h^2\epsilon^{-3/2},\epsilon ,\bar{N}%
^{-1/2}\epsilon ^{-(1/2+d/4)}\right).
\end{equation*}%
\end{enumerate}
in high probability, where $\epsilon \rightarrow 0$ after $\bar{N}\rightarrow
\infty$ and $h\to 0$.

\end{lem}

\begin{proof}
For each $x_{i}\in X$, we use the definitions of vectors \eqref{umh}-%
\eqref{ugh}, \eqref{uM}-\eqref{uh}
\begin{eqnarray}
|(\mathbf{\tilde{L}u}_{M})_{i}-\mathcal{L}u(x_{i})| &=&|(\mathbf{L}^{h}%
\mathbf{U})_{i}-\mathcal{L}u(x_{i})|  \notag \\
&=&|(\mathbf{L}^{h}\mathbf{U})_{i}-(\mathbf{L}^{h}\mathbf{u})_{i}+(\mathbf{L}%
^{h}\mathbf{u})_{i}-(\mathbf{Lu})_{i}+(\mathbf{Lu})_{i}-\mathcal{L}u(x_{i})|
\notag \\
&\leq &\left\vert \big(\mathbf{L}^{(1)}(\mathbf{u}_{M}-\mathbf{u}%
_{M}^{\gamma })\big)_{i}\right\vert +\left\vert \big(\mathbf{L}^{(2)}(%
\mathbf{U}_{G}-\mathbf{u}_{G})\big)_{i}\right\vert +\left\vert \big((\mathbf{%
L}^{h}-\mathbf{L})\mathbf{u}\big)_{i}\right\vert +\left\vert (\mathbf{L}%
\mathbf{u})_{i}-\mathcal{L}u(x_{i})\right\vert,  \label{eqnconc}
\end{eqnarray}
where we notice that each element of $\mathbf{u}_{M}$ in \eqref{umh} is defined on
$X^h$ and each element of $\mathbf{u}^\gamma_{M}$ in \eqref{uM} is defined on $X$.

For randomly sampled data, we have%
\begin{eqnarray}
|(\mathbf{\tilde{L}u}_{M})_{i}-\mathcal{L}u(x_{i})| &=&O(h\epsilon
^{-1/2})+O(h\epsilon ^{-1/2})+O(h^{2}\epsilon ^{-3/2})+%
O(\epsilon ,\bar{N}^{-1/2}\epsilon ^{-(1/2+d/4)})  \notag \\
&=&O\left( h\epsilon ^{-1/2},h^2\epsilon^{-3/2},\epsilon ,\bar{N}^{-1/2}\epsilon
^{-(1/2+d/4)}\right) .  \label{eqn321}
\end{eqnarray}%
Here, the first error bound of order-$h\epsilon ^{-1/2}$ is only valid for
randomly sampled data. This error rate is due to the nonzero components of $%
\mathbf{u}_{M}-\mathbf{u}_{M}^{\gamma }$, which are, $u(\gamma _{b}(h))-u(%
\tilde{x}_{b,0})=O(h\sqrt{\epsilon })$ only when the data is
random and an additional order-$\epsilon ^{-1}$ from the components of $%
\mathbf{L}^{(1)}$. The second error bound is due to \eqref{ugerror}; as we
noted before, the rate-$h\epsilon ^{-1/2}$ is valid only for randomly
sampled data. The third error rate of order-$h^{2}\epsilon ^{-3/2}$ is
induced by the distance of the estimated ghost point from the manifold, $%
\gamma _{b}(h)-\tilde{x}_{b,0}=O(h\sqrt{\epsilon })$, which is
valid for randomly sampled data, as discussed in \cite{jiang2020ghost}. The
fourth error bound is the pointwise consistency in \eqref{Eqn:cbc} for the
interior data points. 

For well sampled data, we have

\begin{equation*}
|(\mathbf{\tilde{L}u}_{M})_{i}-\mathcal{L}u(x_{i})|=O%
(h^{2}\epsilon ^{-1})+O(\epsilon ,\bar{N}^{-1/2}\epsilon
^{-(1/2+d/4)}),
\end{equation*}%
where the first and third error bounds in \eqref{eqnconc} vanishes. The
second error bound in \eqref{eqnconc} of order-$%
h^{2}\epsilon ^{-1}$ is due to \eqref{ugerror} and this rate  is valid for well-sampled data. The fourth error bound
of $O(\epsilon ,\bar{N}^{-1/2}\epsilon ^{-(1/2+d/4)})$\ is still
the pointwise consistency in \eqref{Eqn:cbc} for the interior data points.

\end{proof}

%

\section{Kernel methods for time-dependent PDEs}\label{section3}

In this section, we discuss the discretization of the advection-diffusion
PDE problem (\ref{Eqn:utbf}) and then show the convergence of the proposed
scheme. We split the discussion into two subsections, concerning the no
boundary and boundary cases, respectively.

\subsection{No boundary case}

In the no boundary case, we will just employ the standard DM, discussed in
Section~\ref{basicDMtheory}, to approximate $\mathcal{L}$ in %
\eqref{Eqn:oper3}. Here, since the discrete estimator, $\mathbf{L}\in\mathbb{R}%
^{N\times N}$, is consistent in pointwise-sense, we can achieve a uniform error bound. To begin with, we use the notation $u_{j}^{n}:=u\left( x_{j},t_{n}\right) $
for the true solution at $x_{j}\in X\subset M$ and at time $t_{n}=n\Delta t$%
. Here, $\Delta t$ is the integration time step and $n$ is the number of
time steps such that $0\leq n\Delta t\leq T$. For a given forcing $%
f_{j}^{n}:=f(x_{j},t_{n})$, we denote $U_{j}^{n}$ as the numerical
approximate solution of $u_{j}^{n}$ that will be specified shortly.
For convenience of discussion, we also define the column vectors
\begin{eqnarray}
\mathbf{u}_{M}^{n} &=&(u_{1}^{n},\ldots ,u_{N}^{n}),   \notag \\
\mathbf{U}_{M}^{n} &=&(U_{1}^{n},\ldots ,U_{N}^{n}),  \label{eqnuUf} \\
\mathbf{f}^{n} &=&(f_{1}^{n},\ldots ,f_{N}^{n}).  \notag
\end{eqnarray}
In this paper, we consider the implicit Euler scheme for the time discretization. Together
with the spatial DM discrete estimate, we approximate the PDE in (\ref{Eqn:utbf})
as,%
\begin{equation}
\frac{U_{j}^{n+1}-U_{j}^{n}}{\Delta t}=(\mathbf{L}\mathbf{U}%
_{M}^{n+1})_{j}+f_{j}^{n+1},  \label{Eqn:TBLKS2}
\end{equation}%
where $(\mathbf{L}\mathbf{U}_{M}^{n+1})_{j}$ denotes the $j$th component of
the vector $\mathbf{L}\mathbf{U}_{M}^{n+1}$ corresponding to the function
value at $x_{j}\in M$. In matrix form, the update of the solution can be
written as,%
\begin{equation}
\mathbf{U}_{M}^{n+1}=(\mathbf{I-}\Delta t\mathbf{L})^{-1}(\mathbf{U}%
_{M}^{n}+\Delta t\mathbf{f}^{n+1}).  \label{eqn:UMnp1}
\end{equation}

First, we show the consistency for our scheme (\ref%
{Eqn:TBLKS2}) in solving the time dependent problem (\ref{Eqn:utbf}):

\begin{lem}
\label{thmconsis} (uniform consistency for the scheme (\ref%
{Eqn:TBLKS2})) Let $T_{j}^{n}:=T(x_{j},t_{n})$ be the pointwise truncation
error at $x_j\in X$ and $t_n\in [0,T]$ defined as,
\begin{equation}
T_{j}^{n+1}=\frac{u_{j}^{n+1}-u_{j}^{n}}{\Delta t}-(\mathbf{L}\mathbf{u}%
_{M}^{n+1})_{j}-f_{j}^{n+1}.  \label{Eqn:Tjn2}
\end{equation}%
Under the same hypothesis as in Lemma \ref{thm_LK} and $\Vert u_{tt}\Vert
_{\infty }<\infty $ for all $t$, then the scheme (\ref{Eqn:TBLKS2})\
is consistent uniformly. In particular, we have%
\begin{equation}
\Vert \mathbf{T}^{n}\Vert _{\infty}=O\left( \Delta t,\epsilon
,N^{-1/2}\epsilon ^{-(1/2+d/4)}\right) ,  \label{Eqn:Tnl3}
\end{equation}%
in high probability, as $\Delta t\rightarrow 0$ and $\epsilon \rightarrow 0$
after $N\rightarrow \infty $. Here, we use the notation $\mathbf{T}%
^{n}=(T_{1}^{n},\ldots ,T_{N}^{n})^{\top }$.
\end{lem}

\begin{proof}
We expand in a Taylor series in time around $(x_{j},t_{n+1})$,
\BEA
u_j^{n} = u_j^{n+1} -  u_t (x_j, t_{n+1})  \Delta t + \frac{(\Delta t)^2}{2} u_{tt}(x_j,\eta),\nonumber
\EEA
where $t_n<\eta<t_{n+1}$. Inserting this into the definition of the local truncation error (\ref{Eqn:Tjn2}), we obtain
\begin{equation*}
T_{j}^{n+1}=(u_{t}-\mathcal{L}u-f)|_{(x_{j},t_{n+1})}-\frac{\Delta t}{2}%
u_{tt}(x_{j},\eta )+\big(\mathcal{L}u_j^{n+1}-(\mathbf{L}\mathbf{U%
}^{n+1}_M)_{j}\big),
\end{equation*}%
and the proof is complete by using the fact that $\Vert u_{tt}\Vert _{\infty}< \infty$ and the pointwise consistency in \eqref{Eqn:cbc}.
\end{proof}

The second ingredient for convergence is to show the following stability
condition.

\begin{lem}
\label{lemstab} (maximum norm stability for the scheme (\ref%
{Eqn:TBLKS2})) Let $\mathbf{L\in }\mathbb{R}^{N\times N}$ denotes the
discrete estimator (\ref{classicDMmatrix}) for the advection-diffusion
operator $\mathcal{L}$. Then,
\begin{equation*}
\Vert (\mathbf{I}-\Delta t\mathbf{L})^{-1}\Vert _{\infty }\leq 1.
\end{equation*}
\end{lem}

\begin{proof}
Notice that $\mathbf{L}=(L_{ij})_{i,j=1}^{N}=\frac{\epsilon ^{-d/2-1}}{m_{0}N%
}\big(\mathbf{K}-\mathbf{D}\big)$ satisfy $L_{ii}<0$ and $L_{ij}>0$ for $%
j\neq i$, and its row sum is zero,
\begin{equation*}
\mathbf{L1}=\frac{\epsilon ^{-d/2-1}}{m_{0}N}\big(\mathbf{K1}-\mathbf{D1}%
\big)\mathbf{=0.}
\end{equation*}
Following this property, it is clear that $\mathbf{L}$ in (\ref%
{classicDMmatrix}) satisfies
\begin{equation*}
\left\vert 1-\Delta tL_{ii}\right\vert -\Delta t\sum _{j\neq i}\left\vert
L_{ij}\right\vert =1-\Delta t\sum _{j}L_{ij}=1,
\end{equation*}%
for all $i$. Thus, the matrix $\mathbf{I}-\Delta t\mathbf{L}$ is strictly
diagonally dominant (SDD). Using the Ahlberg-Nilson-Varah bound for SDD
matrices \cite{ahlberg1963convergence,varah1975lower}, one obtains%
\begin{equation*}
\Vert (\mathbf{I}-\Delta t\mathbf{L})^{-1}\Vert _{\infty }\leq \frac{1}{%
\min_{i}(|1-\Delta tL_{ii}|-\Delta t\sum _{j\neq i}|L_{ij}|)}=1.
\end{equation*}
\end{proof}

With these two results, we can now show the uniform convergence, following the standard textbook technique in \cite%
{Morton2005Num,Thomas1995Numerical}.

\begin{thm}
\label{convergeencenoboundary} Under the same hypotheses of the previous two
lemmas, then
\begin{equation*}
\max_{0\leq p\leq n}\Vert \mathbf{u}_{M}^{p}-\mathbf{U}_{M}^{p}\Vert _{{\infty }}\leq n\Delta t\max_{0\leq p\leq n}\Vert \mathbf{T}^{p}\Vert
_{{\infty }}=O\left( \Delta t,\epsilon ,N^{-1/2}\epsilon
^{-(1/2+d/4)}\right) .
\end{equation*}
\end{thm}

\begin{proof}
From (\ref{eqn:UMnp1}), we obtain
\begin{eqnarray*}
\mathbf{U}_{M}^{n+1}-\mathbf{u}_{M}^{n+1} &=&(\mathbf{I-}\Delta t\mathbf{L}%
)^{-1}\big[(\mathbf{U}_{M}^{n}+\Delta t\mathbf{f}^{n+1})-(\mathbf{I-}\Delta t%
\mathbf{L})\mathbf{u}_{M}^{n+1}\big] \\
&=&(\mathbf{I-}\Delta t\mathbf{L})^{-1}\big[\mathbf{U}_{M}^{n}+\Delta t\mathbf{f}%
^{n+1}-\mathbf{u}_{M}^{n+1}+\mathbf{u}_{M}^{n}-\mathbf{u}_{M}^{n}+\Delta t%
\mathbf{Lu}_{M}^{n+1}\big] \\
&=&(\mathbf{I-}\Delta t\mathbf{L})^{-1}\left[ (\mathbf{U}_{M}^{n}-\mathbf{u}%
_{M}^{n})-\Delta t\left( \frac{\mathbf{u}_{M}^{n+1}-\mathbf{u}_{M}^{n}}{%
\Delta t}-\mathbf{Lu}_{M}^{n+1}-\mathbf{f}^{n+1}\right) \right]  \\
&=&(\mathbf{I-}\Delta t\mathbf{L})^{-1}[(\mathbf{U}_{M}^{n}-\mathbf{u}%
_{M}^{n})-\Delta t\mathbf{T}^{n+1}].
\end{eqnarray*}%
Define $\mathbf{e}^{n}=\mathbf{U}_{M}^{n}-\mathbf{u}_{M}^{n}$ and take the maximum norm, we have
\begin{equation*}
\Vert \mathbf{e}^{n+1}\Vert _{{\infty }}\leq \Vert (\mathbf{I-}\Delta t%
\mathbf{L})^{-1}\Vert _{\infty }\big[ \Vert \mathbf{e}^{n}\Vert
_{{\infty }}+\Delta t\Vert \mathbf{T}^{n+1}\Vert _{{\infty }}%
\big] \leq \Vert \mathbf{e}^{n}\Vert _{{\infty }}+\Delta t\Vert
\mathbf{T}^{n+1}\Vert _{{\infty }}.
\end{equation*}%
Since $\Vert \mathbf{e}^{0}\Vert _{\infty }=0$, we have
\begin{equation*}
\Vert \mathbf{e}^{n}\Vert _{{\infty }}\leq n\Delta t\max_{0\leq p\leq
n}\Vert \mathbf{T}^{p}\Vert _{{\infty }}.
\end{equation*}%
for all time. Using Lemma \ref{thmconsis}, the proof is complete.
\end{proof}

\subsection{Initial-Boundary Value Problems}

In this section, we will consider the case with boundary. In the discussion
below, we consider both the Dirichlet and Neumann boundary conditions.
Specifically, we will employ the GPDM spatial discretization in tandem with
the temporal implicit Euler scheme. To facilitate the discussion, we split
the GPDM matrix as,
\begin{equation}
\mathbf{\tilde{L}}=%
\begin{pmatrix}
\mathbf{\tilde{L}}^{I,I} & \mathbf{\tilde{L}}^{I,B} \\
\mathbf{\tilde{L}}^{B,I} & \mathbf{\tilde{L}}^{B,B}%
\end{pmatrix}%
\in \mathbb{R}^{N\times N},  \label{eqnLtisep}
\end{equation}%
where we will use the matrices $\mathbf{\tilde{L}}^{I,I}\in \mathbb{R}%
^{(N-B)\times (N-B)}$, corresponding to components of interior points, and $%
\mathbf{\tilde{L}}^{I,B}\in \mathbb{R}^{(N-B)\times B}$, corresponding to
the cross terms between the interior and boundary points, for the numerical
scheme. We will also split each of following $N$-dimensional column vectors,
$\mathbf{U}_{M}=(\mathbf{U}_{I},\mathbf{U}_{B})$ and $\mathbf{u}_{M}=(%
\mathbf{u}_{I},\mathbf{u}_{B})$, into two parts, one that includes the
interior components and the another one that includes the boundary
components. With these notations, one can see that,
\begin{equation}
\begin{pmatrix}
\mathbf{\tilde{L}}^{I,I} & \mathbf{\tilde{L}}^{I,B}%
\end{pmatrix}%
\mathbf{U}_{M}=\mathbf{\tilde{L}}^{I,I}\mathbf{U}_{I}+\mathbf{\tilde{L}}%
^{I,B}\mathbf{U}_{B}.  \label{eqnLUM}
\end{equation}%
which we will use to simplify the numerical scheme. We should point out
that, since these matrices above are constructed on $X^{h}=\{x_{1},\ldots ,%
\tilde{x}_{b,0},\ldots ,x_{N}\}$, then $\mathbf{U}_{M}=\{U_{1},\ldots
,U_{N}\}$ in (\ref{eqnuUf})\ is the approximating solution of $\mathbf{u}%
_{M} $, whose components are the true solution evaluated at $X^{h}$, as
defined in \eqref{umh}. The forcing term is also given from $f_{j}\in
\{f(x_{1}),\ldots ,f(\tilde{x}_{b,0}),\ldots ,f(x_{N})\}$ which is evaluated at $X^h$.


\noindent \textbf{Dirichlet boundary:} In particular, the discrete solver,
\begin{eqnarray}
\frac{U_{j}^{n+1}-U_{j}^{n}}{\Delta t} &=&(\mathbf{\tilde{L}}\mathbf{U}%
_{M}^{n+1})_{j}+f_{j}^{n+1},\quad \ \ j=1,\ldots ,N-B,  \label{discretepde}
\\
U_{b}^{n} &=&g(x_{b}),\quad \text{\ \ \ \ \ \ \ \ \ \ }b=1,\ldots ,B,
\label{DBC}
\end{eqnarray}%
can be simplified to,
\begin{equation}
\frac{U_{j}^{n+1}-U_{j}^{n}}{\Delta t}=(\mathbf{\tilde{L}}^{I,I}\mathbf{U}%
_{I}^{n+1})_{j}+f_{j}^{n+1}+(\mathbf{\tilde{L}}^{I,B}\mathbf{g})_{j},\quad
j=1,\ldots ,N-B,  \label{schemedirichlet}
\end{equation}%
which is numerically iterated starting from the initial condition,
\begin{equation*}
U_{j}^{0}=u(x_{j},0),\quad \ \ {\textrm{for}}\ x_j \in X^h.
\end{equation*}%
Here, the column vector $\mathbf{g}=\left( g(x_{1}),\ldots ,g(x_{B})\right) \in
\mathbb{R}^{B}$ in (\ref{schemedirichlet}). The local truncation
error is defined as,%
\begin{equation}
T_{j}^{n+1}:= T(x_{j},t_{n+1})=\frac{u_{j}^{n+1}-u_{j}^{n}}{\Delta t}-(%
\mathbf{\tilde{L}}^{I,I}\mathbf{u}_{I}^{n+1})_{j}-f_{j}^{n+1}-(\mathbf{%
\tilde{L}}^{I,B}\mathbf{g})_{j},  \label{defD}
\end{equation}
where $u_j^n = u(x_j,t_n)$ for $x_j \in X^h$.

\noindent \textbf{Neumann boundary:} For Neumann BC, the numerical
approximation of the boundary condition is given by,
\begin{equation}
(\mathbf{B}\mathbf{U}_{M}^{n})_{b}:=\frac{U_{b}^{n}-U_{b,0}^{n}}{h}%
=g(x_{b}),\quad \ \ b=1,\ldots ,B,  \label{def:B}
\end{equation}%
where $h$ denotes the distance between boundary point $x_{b}$ and interior
ghost point $\tilde{x}_{b,0}$. In compact notation, we can rewrite the
boundary conditions as,
\begin{equation}
\left( \mathbf{U}_{B}^{n}\right) _{b}=U_{b}^{n}=U_{b,0}^{n}+hg(x_{b}):=(%
\mathbf{E}^{B,I}\mathbf{U}_{I}^{n})_{b}+hg(x_{b}),  \label{NeumannD}
\end{equation}%
where $\mathbf{E}^{B,I}\in \mathbb{R}^{B\times (N-B)}$ is a matrix whose $b$%
th row equals to one on the column corresponding to the $\tilde{x}_{b,0}$
and zero everywhere else. One can show that,
\begin{eqnarray}
\big(\mathbf{B}(\mathbf{u}_{M}-\mathbf{U}_{M})\big)_{b} &=&\frac{1}{h}%
(u(x_{b})-u(\tilde{x}_{b,0}))-\frac{1}{h}(U_b-U_{b,0})=\frac{1}{h}%
(u(x_{b})-u(\tilde{x}_{b,0}))-g(x_{b})  \notag \\
&=&\frac{1}{h}(u(x_{b})-u(\tilde{x}_{b,0})-h\partial _{\boldsymbol{\nu }%
}u(x_{b}))=O(h,\epsilon ^{1/2}),  \label{Errorinb}
\end{eqnarray}%
where we have used the fact that $\partial _{\boldsymbol{\nu }%
}u(x_{b})=\partial _{\boldsymbol{\tilde{\nu}}}u(x_{b})+O%
(h,\epsilon ^{1/2})$. Here, $O(h)$ is only applied to well-sampled data and $O(\epsilon^{1/2})$ is only applied to randomly sampled data as discussed in \eqref{approxghostpoints}. By definition in \eqref{NeumannD}, this also means,
\begin{eqnarray}
\left( \left( \mathbf{u}_{B}-\mathbf{E}^{B,I}\mathbf{u}_{I}\right) -(\mathbf{%
U}_{B}-\mathbf{E}^{B,I}\mathbf{U}_{I})\right) _{b} &=&\left( u(x_{b})-u(%
\tilde{x}_{b,0})\right) -\left( U_{b}-U_{b,0}\right)  \notag \\
&=&\big(h\mathbf{B}(\mathbf{u}_{M}-\mathbf{U}_{M})\big)_{b}=O%
(h^{2},h\epsilon ^{1/2}).  \label{eqneib}
\end{eqnarray}%
In this case, Eq.~\eqref{discretepde} with the boundary condition in %
\eqref{NeumannD} can be written as,
\begin{equation}
\frac{U_{j}^{n+1}-U_{j}^{n}}{\Delta t}=((\mathbf{\tilde{L}}^{I,I}+\mathbf{%
\tilde{L}}^{I,B}\mathbf{E}^{B,I})\mathbf{U}_{I}^{n+1})_{j}+f_{j}^{n+1}+(%
\mathbf{\tilde{L}}^{I,B}h\mathbf{g})_{j},\quad \text{\ \ }j=1,\ldots ,N-B.
\label{schemeneumann}
\end{equation}%
For $u_j^n = u(x_j,t_n)$ with $x_j \in X^h$,  the local trunction error is defined as,%
\begin{equation}
T_{j}^{n+1}\equiv T(x_{j},t_{n+1}):=\frac{u_{j}^{n+1}-u_{j}^{n}}{\Delta t}-((%
\mathbf{\tilde{L}}^{I,I}+\mathbf{\tilde{L}}^{I,B}\mathbf{E}^{B,I})\mathbf{u}%
_{I}^{n+1})_{j}-f_{j}^{n+1}-(\mathbf{\tilde{L}}^{I,B}h\mathbf{g})_{j}.
\label{defN}
\end{equation}

We first show the estimate for the truncation error in the same way as we
did in the previous section.

\begin{lem}
\label{thmconsis2} (local truncation error of the schemes %
\eqref{schemedirichlet} and \eqref{schemeneumann}) Let $%
T_{j}^{n}:=T(x_{j},t_{n})$ be the pointwise truncation errors at $%
(x_{j},t_{n})\in X^h\times \lbrack 0,T]$ defined in \eqref{defD} and %
\eqref{defN} for the Dirichlet and Neumann cases, respectively. Let $R$ be the number of interior points in $X^{h}$ with
distance within $\epsilon ^{r}$, $0<r<1/2$ to the boundary $\partial M$. The
same hypothesis holds as in Lemma \ref{theorem1}.
\begin{enumerate}
\item For well-sampled data, assume that $R=O(N^{1/2})$ as $%
N\rightarrow \infty $ (where $h$ and $\epsilon $ are also functions of $N$)
and $\Vert
u_{tt}\Vert _{\ell ^{2}}<\infty $ for all $t$.
Then, the schemes \eqref{schemedirichlet} and \eqref{schemeneumann} are
consistent in the $\ell ^{2}$-norm, that is,
\begin{equation}
\Vert \mathbf{T}^{n}\Vert _{\ell ^{2}}= O\left( \Delta t,\epsilon
,N^{-1/4}h^2\epsilon^{-1}, \bar{N}^{-1/2}\epsilon ^{-(1/2+d/4)}\right),  \label{Eqn:Tnl2}
\end{equation}
in high probability, as $\Delta t\rightarrow 0$ and $\epsilon \rightarrow 0$
after $\bar{N}\rightarrow \infty $. Here, the notation $\mathbf{T}%
^{n}=(T_{1}^{n},\ldots ,T_{N-B}^{n})$ denotes the error at the interior
points.
\item For randomly sampled data, assume that $\Vert
u_{tt}\Vert _{\infty}<\infty $ for all $t$. Then, the schemes \eqref{schemedirichlet} and %
\eqref{schemeneumann} are consistent  uniformly. In
particular, one has
\begin{equation}
\Vert \mathbf{T}^{n}\Vert _{{\infty }}=O\left( \Delta t,h\epsilon
^{-1/2},h^2\epsilon^{-3/2},\epsilon ,\bar{N}^{-1/2}\epsilon ^{-(1/2+d/4)}\right) ,  \label{eqn:Tnlinf}
\end{equation}%
in high probability, as $\Delta t\rightarrow 0$ and $\epsilon \rightarrow 0$
after $\bar{N}\rightarrow \infty $ and $h\to 0$.
\end{enumerate}
\end{lem}

\begin{proof}
We will prove only the Neumann condition. The proof for Dirichlet problem
follows exactly the same argument and is slightly simpler. We also use the notation
$X \equiv \{x_j^\gamma\}_{j=1}^N := \{x_{1},\ldots ,\gamma _{b}(h),\ldots ,x_{N}\} \subseteq M$ for the ease of the following discussion. Then, by definition of
the local truncation error on $x_{j}\in X^h$, one can deduce that,
\begin{eqnarray}
T_{j}^{n+1}&:= &\frac{u_{j}^{n+1}-u_{j}^{n}}{\Delta t}-\big((\mathbf{\tilde{L%
}}^{I,I}+\mathbf{\tilde{L}}^{I,B}\mathbf{E}^{B,I})\mathbf{u}_{I}^{n+1}\big)%
_{j}-f_{j}^{n+1}-(\mathbf{\tilde{L}}^{I,B}h\mathbf{g})_{j}  \notag \\
&=&(u_{t}-f)|_{(x_{j},t_{n+1})}-\frac{\Delta t}{2}u_{tt}(x_{j},\eta )-\big((%
\mathbf{\tilde{L}}^{I,I}+\mathbf{\tilde{L}}^{I,B}\mathbf{E}^{B,I})\mathbf{u}%
_{I}^{n+1}\big)_{j}-(\mathbf{\tilde{L}}^{I,B}h\mathbf{g})_{j}  \notag \\
&=&(u_{t}-\mathcal{L}u-f)|_{(x_{j}^{\gamma },t_{n+1})}-\frac{\Delta t}{2}%
u_{tt}(x_{j},\eta )+\big(\mathcal{L}u^{n+1}|_{x_{j}^{\gamma }}-\big(\mathbf{(%
}\mathbf{\tilde{L}}^{I,I}+\mathbf{\tilde{L}}^{I,B}\mathbf{E}^{B,I})\mathbf{u}%
_{I}^{n+1}\big)_{j}-(\mathbf{\tilde{L}}^{I,B}h\mathbf{g})_{j}\big)+O(h\sqrt{%
\epsilon })  \notag \\
&=&-\frac{\Delta t}{2}u_{tt}(x_{j},\eta )+\left( \mathcal{L}%
u^{n+1}|_{x_{j}^{\gamma }}-\big(\mathbf{\tilde{L}}\mathbf{u}_{M}^{n+1}\big)%
_{j}\right) +\left( \big(\mathbf{\tilde{L}}\mathbf{u}_{M}^{n+1}\big)_{j}-%
\big(\mathbf{(}\mathbf{\tilde{L}}^{I,I}+\mathbf{\tilde{L}}^{I,B}\mathbf{E}%
^{B,I})\mathbf{u}_{I}^{n+1}\big)_{j}-(\mathbf{\tilde{L}}^{I,B}h\mathbf{g}%
)_{j}\right) +O(h\sqrt{\epsilon })  \notag \\
&=&-\frac{\Delta t}{2}u_{tt}(x_{j},\eta )+\left( \mathcal{L}%
u^{n+1}|_{x_{j}^{\gamma }}-\big(\mathbf{\tilde{L}}\mathbf{u}_{M}^{n+1}\big)%
_{j}\right) +\big(\mathbf{\tilde{L}}^{I,B}\mathbf{u}_{B}^{n+1}-\mathbf{%
\tilde{L}}^{I,B}\mathbf{E}^{B,I}\mathbf{u}_{I}^{n+1}-\mathbf{\tilde{L}}%
^{I,B}(\mathbf{U}_{B}^{n+1}-\mathbf{E}^{B,I}\mathbf{U}_{I}^{n+1})\big)%
_{j}+O(h\sqrt{\epsilon })  \notag \\
&=&-\frac{\Delta t}{2}u_{tt}(x_{j},\eta )+\big(\mathcal{L}%
u^{n+1}|_{x_{j}^{\gamma }}-\mathbf{\tilde{L}}\mathbf{u}_{M}^{n+1}\big)_{j}+%
\big(\mathbf{\tilde{L}}^{I,B}h\mathbf{B}(\mathbf{u}_{M}^{n+1}-\mathbf{U}%
_{M}^{n+1})\big)_{j}+O(h\sqrt{\epsilon })  \notag \\
&=&\left\{
\begin{array}{ll}
-\frac{\Delta t}{2}u_{tt}(x_{j},\eta )+O(h^{2}\epsilon ^{-1},\epsilon ,\bar{N%
}^{-1/2}\epsilon ^{-(1/2+d/4)})+O(h^{2}\epsilon ^{-1}), & \text{well-sampled
data,} \\
-\frac{\Delta t}{2}u_{tt}(x_{j},\eta )+O(h\epsilon ^{-1/2},h^{2}\epsilon
^{-3/2},\epsilon ,\bar{N}^{-1/2}\epsilon ^{-(1/2+d/4)})+O(h\epsilon
^{-1/2}), & \text{randomly-sampled data,}%
\end{array}%
\right. \label{eqn51}
\end{eqnarray}%
where $t_{n}<\eta <t_{n+1}$. %
In the second equality above, we employ the Taylor expansion with respect to time as in the proof
of Lemma~\ref{thmconsis}. In the third line, we have used \eqref{eqngamtid} to account for errors between
$\gamma_b(h) \in M$ and interior ghosts $\tilde{x}_{b,0}$, which is only applied to the randomly sampled case.
In the fourth line, we employ the PDE in \eqref{Eqn:utbf} on points $x_j^\gamma \in X \subseteq M$.
In the fifth line, we use the definitions of $%
\mathbf{\tilde{L}}$ in \eqref{eqnLUM} and $\mathbf{E}^{B,I}$ in \eqref{NeumannD}.
In the sixth line, we use the derivation in \eqref{eqneib}. In the last
lines,
we  use the
consistency of GPDM in Lemma~\ref{theorem1} for the second error term. In this line, we also use
\eqref{eqneib} for the third error term, in particular, this third
error term is on order of $h^{2}\epsilon ^{-1}$ for well-sampled data whereas
is on order of $h\epsilon ^{-1/2}$ for randomly sampled data. The error rate of $O(h\sqrt{\epsilon})$
can be ignored since this term is only applied to the randomly sampled case and is always smaller than $O(h\epsilon^{-1/2})$ in \eqref{eqn51}.

For well-sampled data, the $\ell^2$-error bound is given as,
{\begin{eqnarray}
\|\mathbf{T}^n\|_{\ell^2} &=& \left[\frac{1}{N-B}\Big(\sum_{j=1}^R (T_j^n)^2 +
\sum_{j=R+1}^{N-B} (T_j^n)^2 \Big)\right]^{1/2}  \notag \\
&=& \Big( c_1\frac{R}{N-B}  h^4\epsilon^{-2} + c_2\frac{%
N-B-R}{N-B}\epsilon^2 + O\big(\Delta t^2,(\bar{N}^{-1/2}%
\epsilon^{-(1/2+d/4)})^2\big) \Big)^{1/2},  \notag \\
&\leq &  c_1' N^{-1/4}h^2 \epsilon^{-1} + O(\epsilon,\Delta t,\bar{N}^{-1/2}\epsilon^{-(1/2+d/4)}),
\notag
\end{eqnarray}}

for some constants $c_1, c_2$ that are independent of $h,\epsilon, N$.
To obtain the last equality, we have used the fact that $R = O(N^{1/2})$ as $N\to \infty$. In the last term,
we also use the $\ell^2$-norm boundedness of $u_{tt}$ for  well-sampled data as in the assumption so that the order-$\Delta t$  absorbs the $\|u_{tt}\|_{\ell^2}<\infty$ term.

For the randomly-sampled data, Eq.~\eqref{eqn:Tnlinf} is immediate by taking the maximum norm.
\end{proof}

For well-sampled data, we numerically set $\epsilon=O(h^2)$ for convergence solutions. This choice of parameter implies that each of the $R$ interior points that is within $\epsilon^r$, $0<r<1/2$, contributes to error of order-1 . Recall that in the standard diffusion, the error of these interior points are of order-$\epsilon^{-1/2}$ (see \cite{coifman2006diffusion}) unless Neumann condition is imposed. With the order-1 pointwise error near the boundary, we can only theoretically establish the convergence in $\ell^2$-sense  under the condition that $ R \leq  C N^{1/2}$ as $N\to\infty$ for some constant $C>0$, which motivates the first result in the Lemma above. For this well-sampled data we numerically find that the  convergence rate is faster than the last two error terms, $O(N^{-1/4}h^2\epsilon^{-1},\bar{N}^{-1/2}\epsilon^{-(1/2+d/4)})$ in \eqref{Eqn:Tnl2}.

On the other hand, we can consider the uniform convergence
for the randomly-sampled data. This is because the parameter $\epsilon^2=O(h)$ for convergence solutions is rather small (see Section~\ref{numericrandom}). In light of this, the leading error term in \eqref{eqn:Tnlinf} will be of order-$\epsilon$, which we shall verify numerically.

Next, we will study the stability of the numerical scheme in \eqref{schemeneumann}.

\begin{lem}
\label{lemstabNeubou} (Stability for Neumann problem \eqref{schemeneumann})
Define {$\mathbf{N}=\mathbf{\tilde{L}}^{I,I}+\mathbf{\tilde{L}}^{I,B}\mathbf{%
E}^{B,I}$ in the Neumann problem }\eqref{schemeneumann}.
\begin{enumerate}
\item For any $\Delta t>0$ and any integer $N$, the maximum norm of $(\mathbf{I%
}-\Delta t\mathbf{N})^{-1}$ can be bounded by \newline
\begin{equation}
\Vert (\mathbf{I}-\Delta t\mathbf{N})^{-1}\Vert _{\infty }\leq 1.
\label{lembound1}
\end{equation}
\item Let $R$ be the number of interior points of $X^{h}$ with distance within
$\epsilon ^{r}$, $0<r<1/2$, to the boundary $\partial M$ and assume that $%
R=O(\sqrt{N})$. Then, for sufficiently large $N$ and sufficiently small $%
\Delta t$, there exists a positive constant $C$ such that
\begin{equation}
\Vert (\mathbf{I}-\Delta t\mathbf{N})^{-1}\Vert _{2}\leq 1+C\Delta t.
\label{negdefLtilde}
\end{equation}
Here $\|\cdot\|_2$ denotes the usual matrix 2-norm.
\end{enumerate}
\end{lem}

See Appendix~\ref{App:B} for the proof.
\begin{lem}
\label{lemstabDir} (Stability for Dirichlet problem \eqref{schemedirichlet})
Define $\mathbf{H}\equiv \mathbf{\tilde{L}}^{I,I}$ in the Dirichlet problem %
\eqref{schemedirichlet}. Let $R$ be the number of interior points of $X^{h}$
with distance within $\epsilon ^{r}$, $0<r<1/2$, to the boundary $\partial M$
and assume that $R=O(\sqrt{N})$. Then, for sufficiently large $N$ and
sufficiently small $\Delta t$, there exists positive constants $C_{1}$ and $%
C_{2}$\ such that

\begin{enumerate}
\item the maximum norm of $(\mathbf{I}-\Delta t\mathbf{H})^{-1}$ can be
bounded by \newline
\begin{equation}
\Vert (\mathbf{I}-\Delta t\mathbf{H})^{-1}\Vert _{\infty }\leq 1+C_{1}\Delta
t,  \label{lembounddir2}
\end{equation}

\item and the usual matrix 2-norm, $\Vert \cdot \Vert _{2}$, of $(\mathbf{I}%
-\Delta t\mathbf{N})^{-1}$ can be bounded by
\begin{equation}
\Vert (\mathbf{I}-\Delta t\mathbf{H})^{-1}\Vert _{2}\leq 1+C_{2}\Delta t.
\label{negdefdir2}
\end{equation}
\end{enumerate}
\end{lem}

See Appendix~\ref{App:C} for the proof.

\comment{
For Dirichlet problem, one can follow similar steps as in the Neumann case.
{\color{red} I suspect that the statement of (49) is slightly changed for Dirichlet; you don't have the property 1) in the Appendix and the row sum property 2) is strictly negative. So for some small $\Delta t$, (49) holds. Eq. (57) may not hold.
}}

Using the consistency in Lemma {\ref{thmconsis2}}\ and stability in Lemma {%
\ref{lemstabNeubou}} or Lemma \ref{lemstabDir}, we immediately have the convergence.

\begin{thm}
\label{convergewithbound} Under the same hypotheses of the previous Lemma {%
\ref{thmconsis2} and Lemma \ref{lemstabNeubou}} for Neumann case (or Lemma \ref{lemstabDir} for Dirichlet case), then:
\begin{enumerate}
\item For well-sampled data, there exists constants $C_{1}$ and $C_{2}$
such that
\BEA
\max_{0\leq p\leq n}\Vert \mathbf{u}_{M}^{p}-\mathbf{U}_{M}^{p}\Vert _{\ell
^{2}}&\leq& C_{1}\exp (C_{2}n\Delta t)\max_{0\leq p \leq n}\Vert \mathbf{T}%
^{p}\Vert _{\ell ^{2}}+O(B^{1/2}h^{2}N^{-1/2})\notag \\ &=&O\left( \Delta
t,B^{1/2}h^{2}N^{-1/2},\epsilon,N^{-1/4}h^2\epsilon^{-1} ,\bar{N}^{-1/2}\epsilon ^{-(1/2+d/4)}\right) ,\notag
\EEA
in high probability, as $\Delta t\rightarrow 0$ and $\epsilon \rightarrow 0$
after $\bar{N}\rightarrow \infty $ and $h\to 0$. Here, the second error bound of order-$B^{1/2}h^2N^{-1/2}$ only applies for the Neumann case and does not apply for the Dirichlet case.
\item For randomly sampled data, one has%
\begin{equation*}
\max_{0\leq p\leq n}\Vert \mathbf{u}_{M}^{p}-\mathbf{U}_{M}^{p}\Vert _{{\infty }}\leq n\Delta t\max_{0\leq p\leq n}\Vert \mathbf{T}^{p}\Vert
_{{\infty }}=O\left( \Delta t,h\epsilon ^{-1/2},h^2\epsilon^{-3/2},\epsilon
,\bar{N}^{-1/2}\epsilon ^{-(1/2+d/4)}\right),
\end{equation*}
for Neumann problem, and
\BEA
\max_{0\leq p\leq n}\Vert \mathbf{u}_{M}^{p}-\mathbf{U}_{M}^{p}\Vert _{{\infty}}&\leq& C_{1}\exp (C_{2}n\Delta t)\max_{0\leq p \leq n}\Vert \mathbf{T}%
^{p}\Vert _{\ell ^{2}}\notag \\ &=&O\left( \Delta
t,h \epsilon^{-1/2},h^{2} \epsilon^{-3/2},\epsilon ,\bar{N}^{-1/2}\epsilon ^{-(1/2+d/4)}\right) ,\notag
\EEA
for Dirichlet problem, in high probability, as $\Delta t\rightarrow 0$ and $\epsilon \rightarrow 0$
after $\bar{N}\rightarrow \infty $ and $h\to 0$.
\end{enumerate}
\end{thm}

\begin{proof}
We only prove the Neumann condition. We first consider the well-sampled case. From (\ref%
{schemeneumann}), we obtain
\begin{eqnarray*}
\mathbf{U}_{I}^{n+1}-\mathbf{u}_{I}^{n+1} &=&(\mathbf{I-}\Delta t\mathbf{N}%
)^{-1}[(\mathbf{U}_{I}^{n}+\Delta t\mathbf{f}^{n+1}+\Delta t\mathbf{\tilde{L}%
}^{I,B}h\mathbf{g})-(\mathbf{I-}\Delta t\mathbf{N})\mathbf{u}_{I}^{n+1}] \\
&=&(\mathbf{I-}\Delta t\mathbf{N})^{-1}[\mathbf{U}_{I}^{n}+\Delta t\mathbf{f}%
^{n+1}+\Delta t\mathbf{\tilde{L}}^{I,B}h\mathbf{g}-\mathbf{u}_{I}^{n+1}+%
\mathbf{u}_{I}^{n}-\mathbf{u}_{I}^{n}+\Delta t\mathbf{Nu}_{I}^{n+1}] \\
&=&(\mathbf{I-}\Delta t\mathbf{N})^{-1}\left[ (\mathbf{U}_{I}^{n}-\mathbf{u}%
_{I}^{n})-\Delta t\left( \frac{\mathbf{u}_{I}^{n+1}-\mathbf{u}_{I}^{n}}{%
\Delta t}-\mathbf{Nu}_{I}^{n+1}-\mathbf{f}^{n+1}-\mathbf{\tilde{L}}^{I,B}h%
\mathbf{g}\right) \right]  \\
&=&(\mathbf{I-}\Delta t\mathbf{N})^{-1}[(\mathbf{U}_{I}^{n}-\mathbf{u}%
_{I}^{n})-\Delta t\mathbf{T}^{n+1}].
\end{eqnarray*}%
Define $\mathbf{e}_{I}^{n}=\mathbf{U}_{I}^{n}-\mathbf{u}_{I}^{n}$ and take
the spectral norm, we have
\begin{eqnarray*}
\Vert \mathbf{e}_{I}^{n+1}\Vert _{\ell ^{2}} &\leq &\Vert (\mathbf{I-}\Delta
t\mathbf{N})^{-1}\Vert _{2}\left[ \Vert \mathbf{e}_{I}^{n}\Vert _{\ell
^{2}}+\Delta t\Vert \mathbf{T}^{n+1}\Vert _{\ell ^{2}}\right] \leq
(1+C\Delta t)\left( \Vert \mathbf{e}_{I}^{n}\Vert _{\ell ^{2}}+\Delta t\Vert
\mathbf{T}^{n+1}\Vert _{\ell ^{2}}\right)  \\
&\leq &\left( 1+C\Delta t\right) ^{2}\Vert \mathbf{e}_{I}^{n-1}\Vert _{\ell
^{2}}+\Delta t\left( (1+C\Delta t)\Vert \mathbf{T}^{n+1}\Vert _{\ell
^{2}}+\left( 1+C\Delta t\right) ^{2}\Vert \mathbf{T}^{n}\Vert _{\ell
^{2}}\right)  \\
&\leq &\left( 1+C\Delta t\right) ^{n+1}\Vert \mathbf{e}_{I}^{0}\Vert _{\ell
^{2}}+\Delta t\sum_{p=0}^{n}\left( 1+C\Delta t\right) ^{p+1}\Vert \mathbf{T}%
^{n+1-p}\Vert _{\ell ^{2}}.
\end{eqnarray*}%
Since $\Vert \mathbf{e}_{I}^{0}\Vert _{\ell ^{2}}=0$, we have
\begin{equation*}
\Vert \mathbf{e}_{I}^{n}\Vert _{\ell ^{2}}\leq \Delta t\max_{0\leq p\leq
n}\Vert \mathbf{T}^{p}\Vert _{\ell ^{2}}\sum_{p=0}^{n-1}\left( 1+C\Delta
t\right) ^{p+1}\leq \frac{2\exp (Cn\Delta t)}{C}\max_{0\leq p\leq n}\Vert
\mathbf{T}^{p}\Vert _{\ell ^{2}}.
\end{equation*}%
for all time.

For the boundary, from (\ref{eqneib}), we have%
\begin{eqnarray*}
\sum_{b=1}^{B}(u(x_{b})-U_{b})^{2} &=&\sum_{b=1}^{B}(u(\tilde{x}%
_{b,0})-U_{b,0}+O(h^{2}))^{2}\leq 2\sum_{b=1}^{B}(u(\tilde{x}%
_{b,0})-U_{b,0})^{2}+O(h^{4}) \\
&\leq &2(N-B)\Vert \mathbf{u}_{I}^{n}-\mathbf{U}_{I}^{n}\Vert _{\ell
^{2}}^{2}+O(Bh^{4})=2(N-B)\Vert \mathbf{e}_{I}^{n}\Vert _{\ell
^{2}}^{2}+O(Bh^{4}).
\end{eqnarray*}%
Thus, for $\mathbf{e}_{M}^{n}=\mathbf{U}_{M}^{n}-\mathbf{u}_{M}^{n}$, we
have
\begin{eqnarray*}
\Vert \mathbf{e}_{M}^{n}\Vert _{\ell ^{2}} &=&\left[ \frac{1}{N}\left(
(N-B)\Vert \mathbf{e}_{I}^{n}\Vert _{\ell
^{2}}^{2}+\sum_{b=1}^{B}(u(x_{b})-U_{b})^{2}\right) \right] ^{1/2} \\
&\leq &\left[ \frac{N-B}{N}\Vert \mathbf{e}_{I}^{n}\Vert _{\ell ^{2}}^{2}+2%
\frac{N-B}{N}\Vert \mathbf{e}_{I}^{n}\Vert _{\ell ^{2}}^{2}+O(Bh^{4}N^{-1})%
\right] ^{1/2} \\
&\leq &\sqrt{3}\Vert \mathbf{e}_{I}^{n}\Vert _{\ell
^{2}}+O(B^{1/2}h^{2}N^{-1/2}).
\end{eqnarray*}
Using Lemma \ref{thmconsis2}, the proof for well-sampled data case is
completed.

For randomly sampled data of Neumann case, one can combine the stability in Eq.~\eqref{lembound1} and the consistency in
Eq.~\eqref{eqn:Tnlinf},
and then
follow similar steps as in the proof in Theorem \ref{convergeencenoboundary}.

The proof for Dirichlet problem follows
the same argument.
One only needs to notice that  the error bound of $O(B^{1/2}h^2N^{-1/2})$ does not apply for Dirichlet problem, since there is no error at the boundary points.

\end{proof}

\section{Numerical results}\label{section4}

In this section, we show numerical examples of parabolic PDE problems defined on three different manifolds. First, we test the proposed method on a two-dimensional annulus embedded in $\mathbb{R}^5$ with homogeneous Dirichlet and Neumann boundary conditions. In the second example, we demonstrate the proposed method on a two-dimensional unknown "cow" manifold. In the third example,  we verify the effectiveness of the proposed method when the grid points are randomly sampled data of a one-dimensional ellipse and a two-dimensional semi-torus.  In these examples, we will verify the accuracy and robustness of the proposed spatial discretization coupled with the backward difference (BD) temporal scheme formulated in previous section.

\subsection{Advection diffusion equation on 2D annulus}

First, we consider solving a PDE problem (\ref{Eqn:utbf}) on a two-dimensional annulus embedded in $\mathbb{R}^{5}$ via the following map,
\begin{equation*}
x:=\left( x_{1},x_{2},x_{3},x_{4},x_{5}\right) =(\sin \varphi \cos \theta
,\sin \varphi \sin \theta ,\sin \varphi \cos 2\theta ,\sin \varphi \sin
2\theta ,\sqrt{2}\cos \varphi )\in M\subset \mathbb{R}^{5}
\end{equation*}%
for $0\leq \theta <2\pi ,\pi /4\leq \varphi \leq \pi /2$, where $\varphi
=\pi /4$ and $\varphi =\pi /2$ are the two boundaries of the annulus. The
Riemannian metric is given by%
\begin{equation*}
g_x(v,w)=v^{{\top }}\left(
\begin{array}{cc}
5\sin ^{2}\varphi & 0 \\
0 & 2%
\end{array}%
\right) w,\text{ }v,w\in T_xM.
\end{equation*}%
For the PDE problem (\ref{Eqn:utbf}), the vector field is set to be
\begin{equation*}
a\left( x \right) =\left(
\begin{array}{c}
a^{1} \\
a^{2}%
\end{array}%
\right) =\left(
\begin{array}{c}
0.5+0.1\sin (\theta ) \\
0%
\end{array}%
\right) ,
\end{equation*}%
where $\sin(\theta)=x_2/(x_1^2+x_2^2)$. For any $u(x,t)$, the forcing term $f$ can
be calculated explicitly as
\begin{equation}
f:=u_{t}-\left( b\cdot \nabla _{g}u+\Delta _{g}u\right).  \label{Eqn:fb}
\end{equation}%
Numerically, the grid points $\left\{ \theta _{i},\varphi _{j}\right\}
_{i=1,\ldots ,I}^{j=1,\ldots ,J}$ are well-sampled and uniformly distributed on $[0,2\pi
)\times \left[ \pi /4,\pi /2\right] $, with $%
(I,J)=(45,12),(64,16),(90,23),(128,32),(181,45),$ for different number of
total grid points of $N=540,1024,2070,4096,8145$. To approximate the
advection-diffusion operator using local kernel, we use $k$ nearest
neighbors algorithm for computational efficiency and estimate the bandwidth parameter $\epsilon$ using the automated tuning algorithm discussed in Section~\ref{basicDMtheory}.

For Neumann boundary, the true solution is set be $u_{true}=\sin (4\varphi
)^{2}\cos (\theta )\exp (-t)$ so that $u$ satisfies homogeneous Neumann
boundary condition. Then, we approximate the solution of the PDE problem in (%
\ref{Eqn:utbf}) subjected to the manufactured $f$ as analytically calculated
in (\ref{Eqn:fb}). Figure \ref{Fig1_diskNeum}(a) shows the true solution in
the polar coordinate of $(\theta ,\varphi )\in \lbrack 0,2\pi )\times \left[
\pi /4,\pi /2\right] $, in which the inner circle corresponds to one
boundary $\varphi =\pi /4$ and the outer circle corresponds to the other
boundary $\varphi =\pi /2$. Figures \ref{Fig1_diskNeum}(b) and (c) show the
comparison of pointwise errors for the DM solution and GPDM solution,
respectively, for $N=2070$ at $t=0.05$, where we integrate with $\Delta t=0.001$. To apply the DM and GPDM, we set $%
k=120$ and use the auto-tuned kernel bandwidth parameter $\epsilon =0.0026$. Both methods use
finite difference for the discretization of the Neumann boundary condition.
One can clearly observe from Fig. \ref{Fig1_diskNeum} that the pointwise
error of GPDM is relatively small compared to the pointwise error of DM at $%
t=0.05$.

\begin{figure*}[tbp]
{\scriptsize \centering
\begin{tabular}{ccc}
{\normalsize (a) Truth, $t=0.05$} & {\normalsize (b) Error of DM, $t=0.05$}
& {\normalsize (c) Error of GPDM, $t=0.05$} \\
\includegraphics[width=0.3\linewidth]{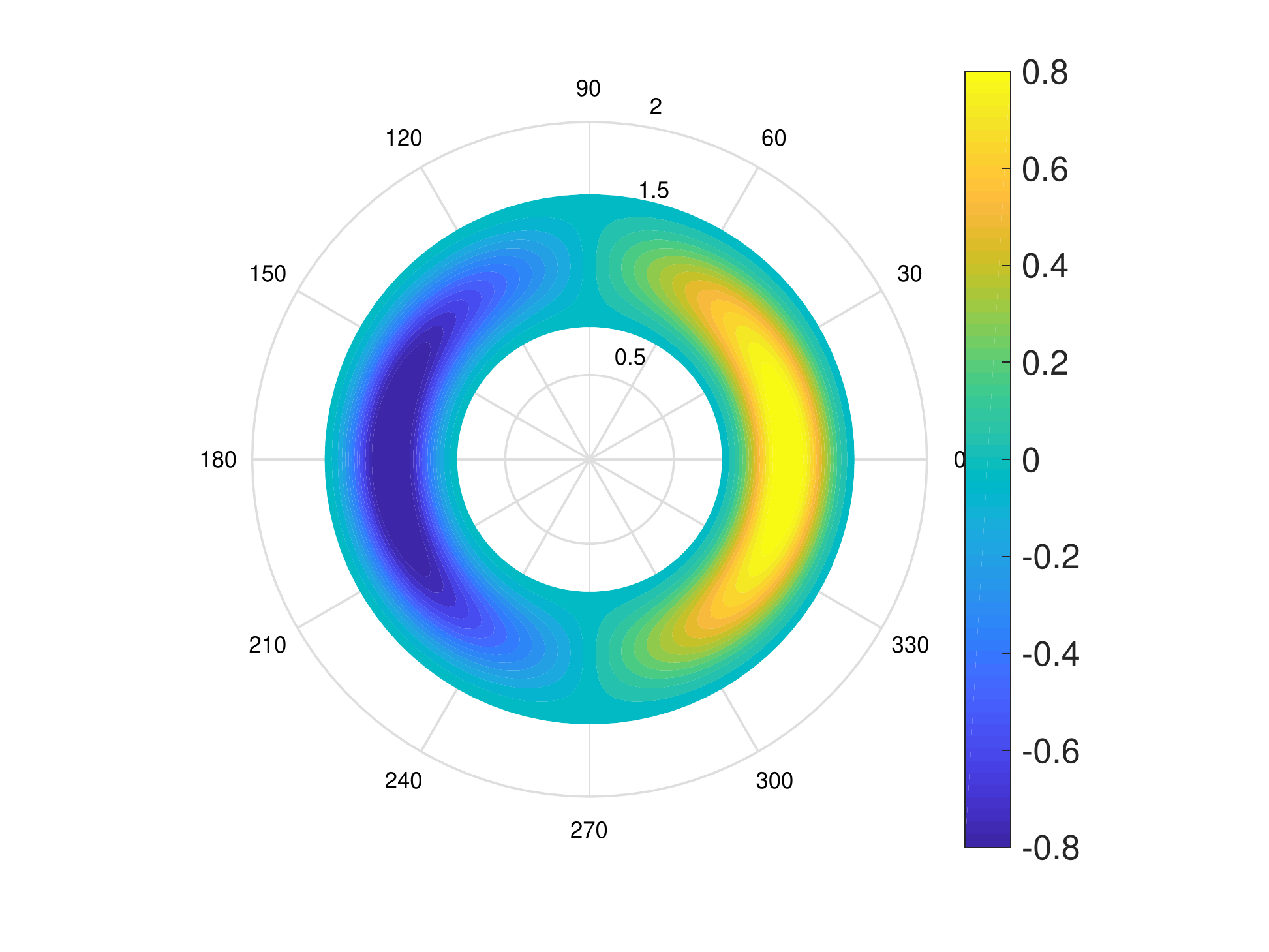} & %
\includegraphics[width=0.3\linewidth]{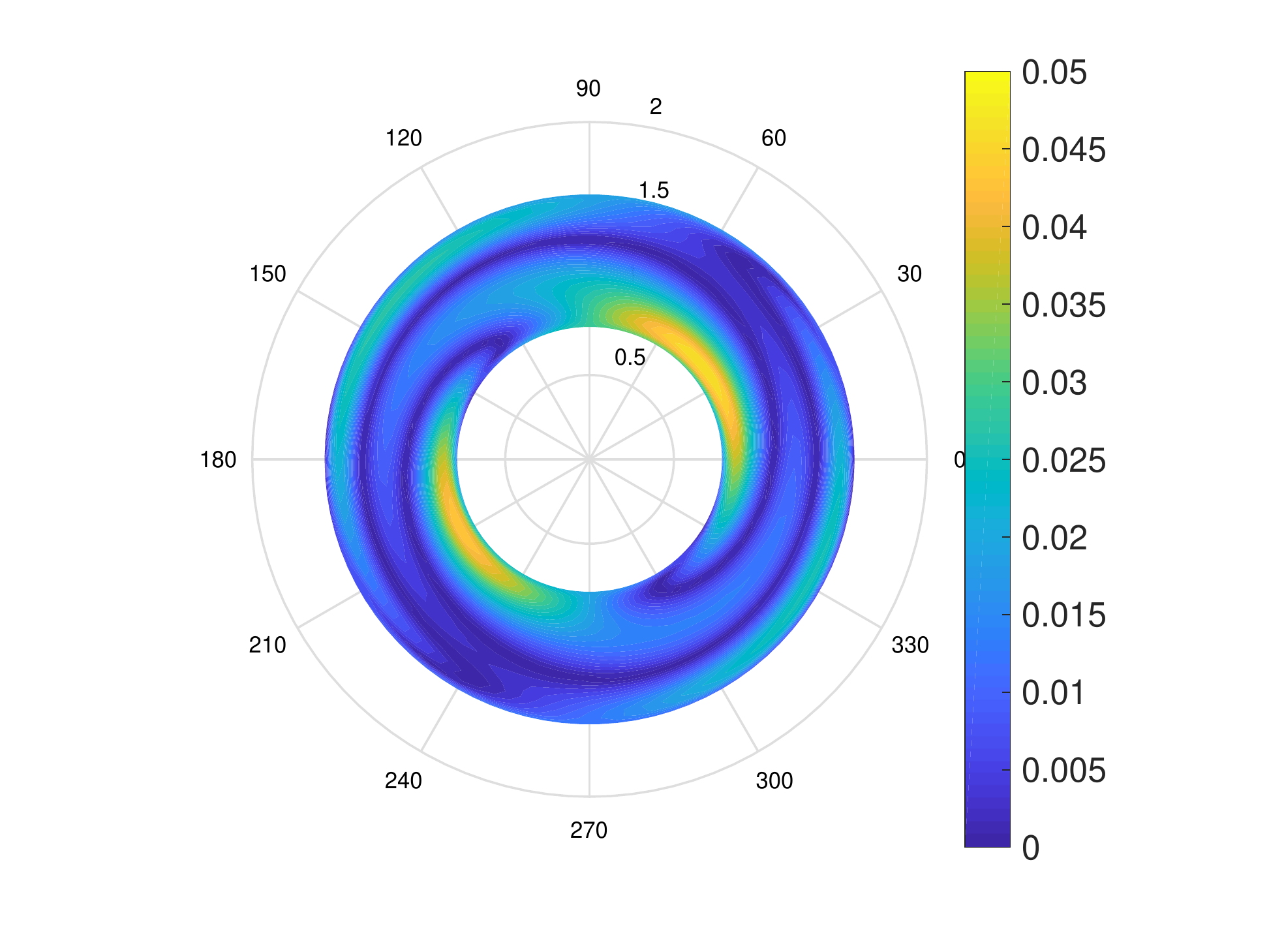} & %
 \includegraphics[width=0.3\linewidth]{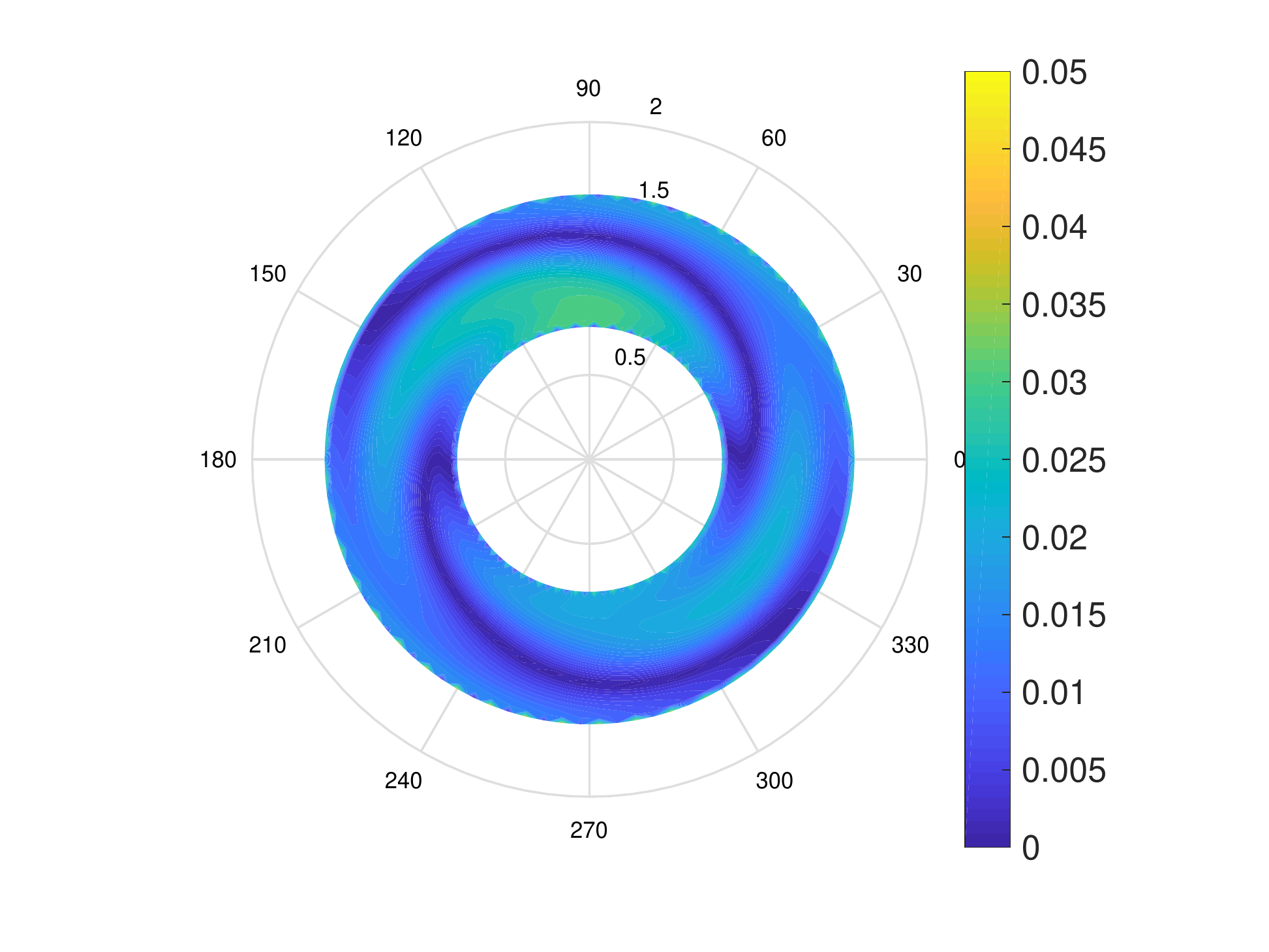} \\
&  &
\end{tabular}
}
\caption{Annulus example for Neumann boundary condition with $N=2070$, $%
t=0.05$, $\Delta t=0.001$: (a) True solution. (b) DM method and (c) GPDM
method for pointwise absolute error as a function of $\protect\theta$ and $%
\protect\varphi$. }
\label{Fig1_diskNeum}
\end{figure*}

For Dirichlet boundary, the true solution is set to be $u_{true}=\sin (4\varphi
)\cos (\theta )\exp (-t)$ so that $u$ satisfies homogeneous Dirichlet
boundary condition. Figure \ref{Fig2_diskDiri}(a) shows the true solution $u$ at $t=0.05$. For diagnostic purposes, we cannot compare GPDM directly to the standard DM since the latter only takes functions that satisfy homogeneous Neumann boundary conditions. Instead, we compare our result to the Volume Constraint Diffusion Maps (VCDM) that was introduced for the Dirichlet boundary problem \cite{shi2015enforce}, which basic idea is to truncate the rows and columns of the local kernel matrix corresponding to the points close enough to the boundaries.
Figures \ref{Fig2_diskDiri}(b) and (c) show the comparison of pointwise
errors for the VCDM and GPDM solutions, respectively, for $N=2070$ at $t=0.05$, obtained with $\Delta t=0.001$. In this numerical test, we take $%
k=120$ and use the auto-tuned kernel bandwidth parameter $\epsilon =0.0026$. Here, at each boundary, $\varphi =\pi /4$ and $\varphi =\pi /2$, three layers close to the boundary are truncated. One can observe that the pointwise error of GPDM is relatively small compared to the pointwise error of VCDM.

\begin{figure*}[tbp]
{\scriptsize \centering
\begin{tabular}{ccc}
{\normalsize (a) Truth, $t=0.05$} & {\normalsize (b) Error of VCDM, $t=0.05$}
& {\normalsize (c) Error of GPDM, $t=0.05$} \\
\includegraphics[width=0.3\linewidth]{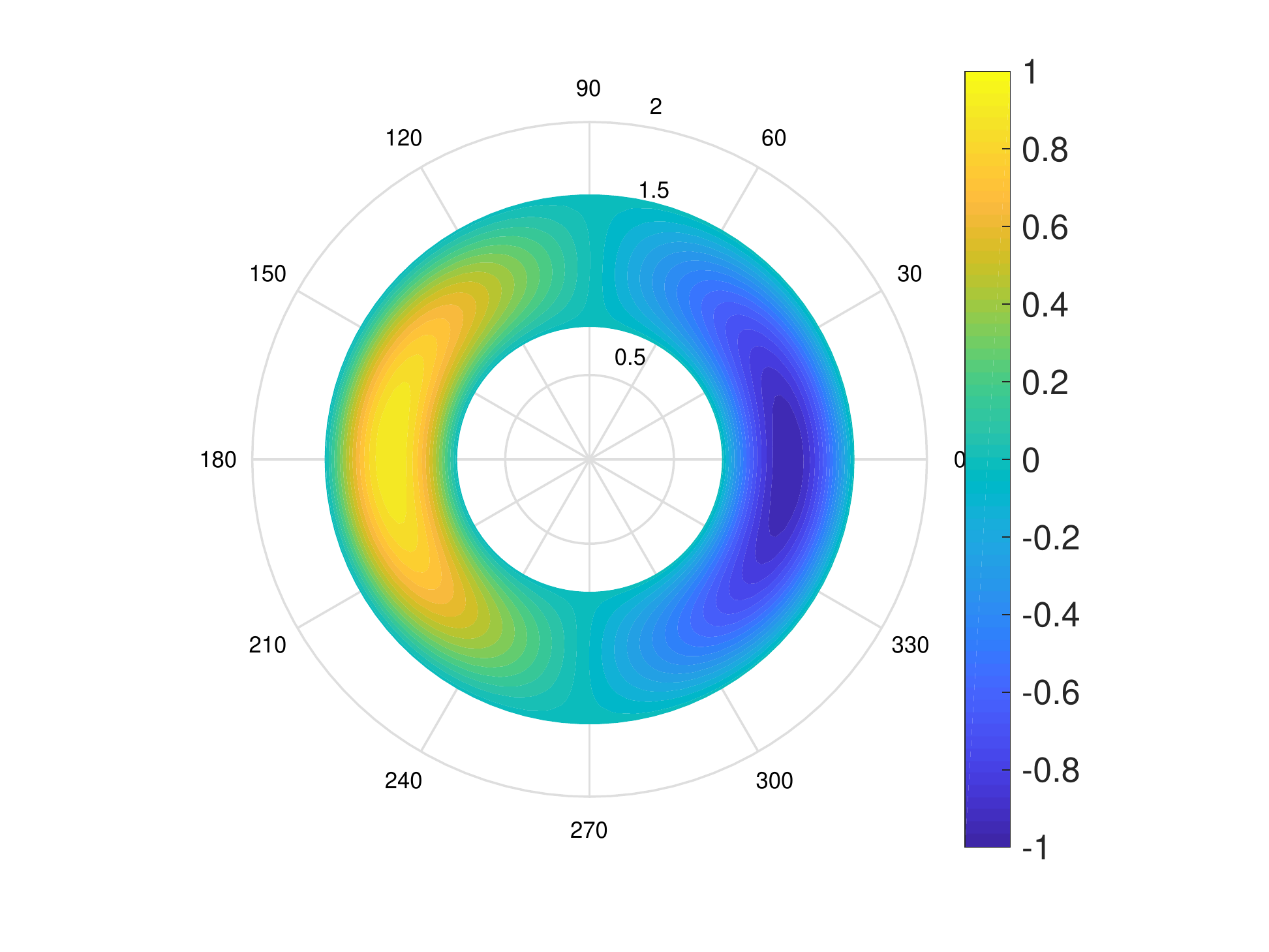} & %
\includegraphics[width=0.3\linewidth]{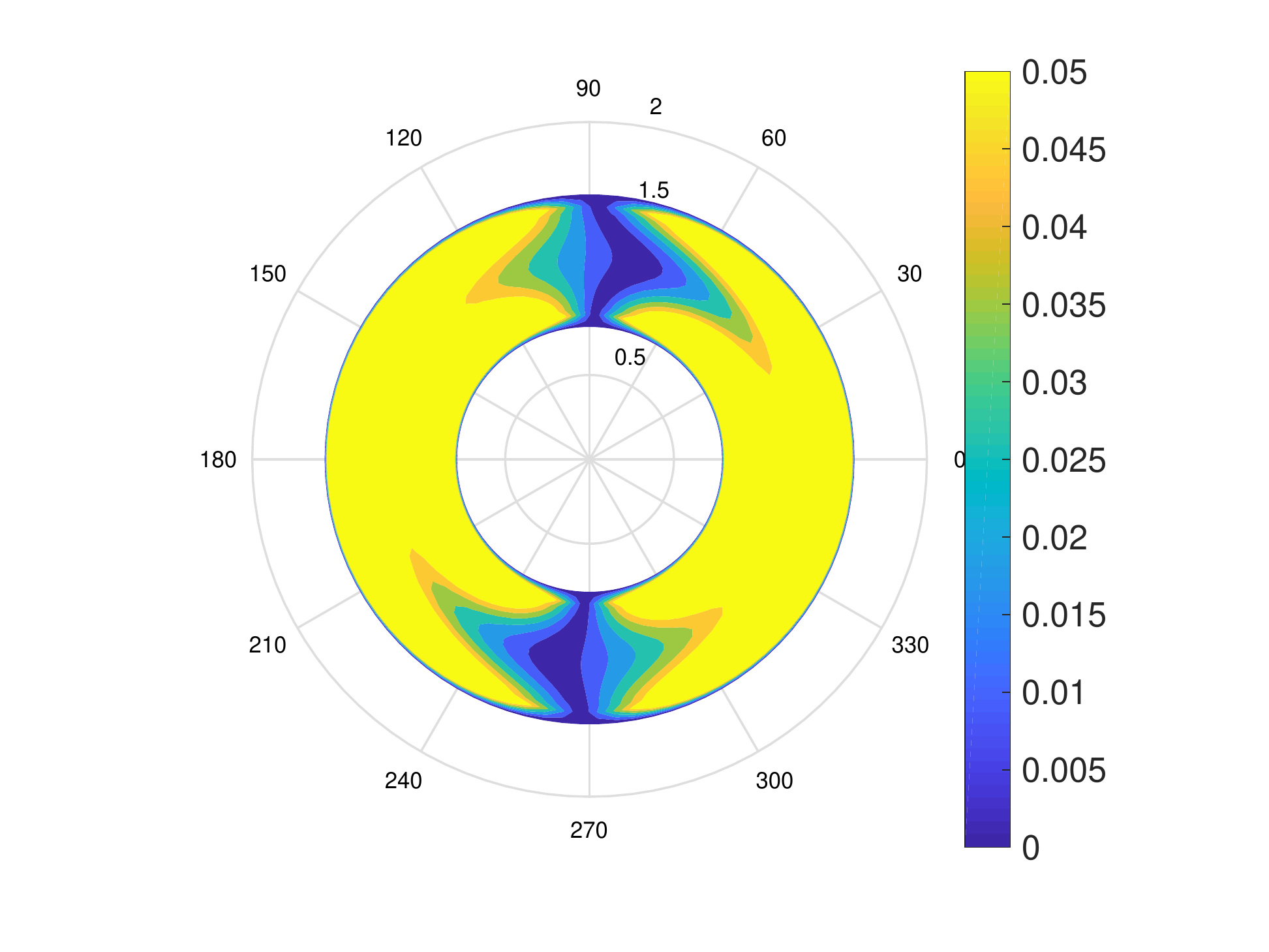} & %
\includegraphics[width=0.3\linewidth]{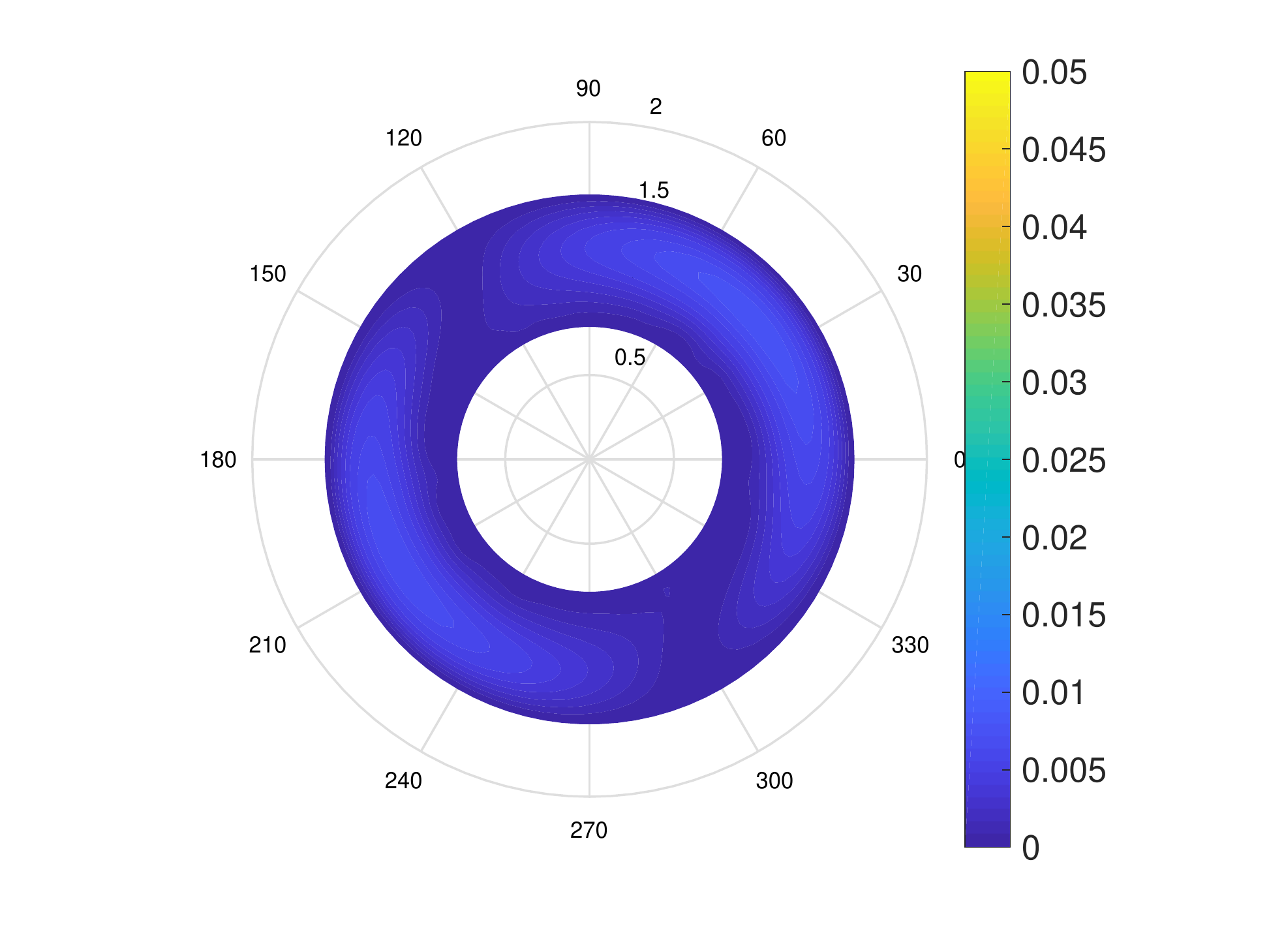} \\
&  &
\end{tabular}
}
\caption{Annulus example for Dirichlet boundary condition with $N=2070$, $%
t=0.05$, $\Delta t=0.001$: (a) True solution. (b) VCDM method truncated with
3 layers, and (c) GPDM method for pointwise absolute error as a function of $%
\protect\theta$ and $\protect\varphi$. }
\label{Fig2_diskDiri}
\end{figure*}

{To illustrate the convergence of solutions over number of points $N$, we fix $k=200$ for each $N$. In Figure~\ref{Fig3_err_2Ddisk}, we report the kernel bandwidth parameters  $\epsilon$, which are chosen to decay on the order of-$N^{-1}$ for both boundary conditions. To achieve convergence solution with these parameter values, we should point out that we have used smaller time step $\Delta t=10^{-4}$ relative to the results shown in the previous two figures. This smaller $\Delta t$ is consistent with the requirement for stability in $\ell^2$ sense as assumed in Lemmas~\ref{lemstabNeubou} and \ref{lemstabDir}. Specifically,
in Figure~\ref{Fig3_err_2Ddisk}, we show the $\ell^2$-errors of numerical solutions at $t=0.005$ as functions of $N$. Using the notation in Theorem \ref{convergewithbound}, the error in this figure is defined as $\|\mathbf{u}^{50}_M-\mathbf{U}^{50}_M\|_{\ell^2}$.
For Neumann and Dirichlet boundary conditions, we can see that the $\ell^2$ error rate is close to an order of bandwidth parameter, $\epsilon\sim N^{-1}$.  Numerically, we  find that the  last two error terms in convergence of Theorem \ref{convergewithbound} is not applied  in this example.


\begin{figure*}[tbp]
{\scriptsize \centering
\begin{tabular}{cc}
{\normalsize (a) Neumann} & {\normalsize (b) Dirichlet} \\
\includegraphics[width=.45\linewidth]{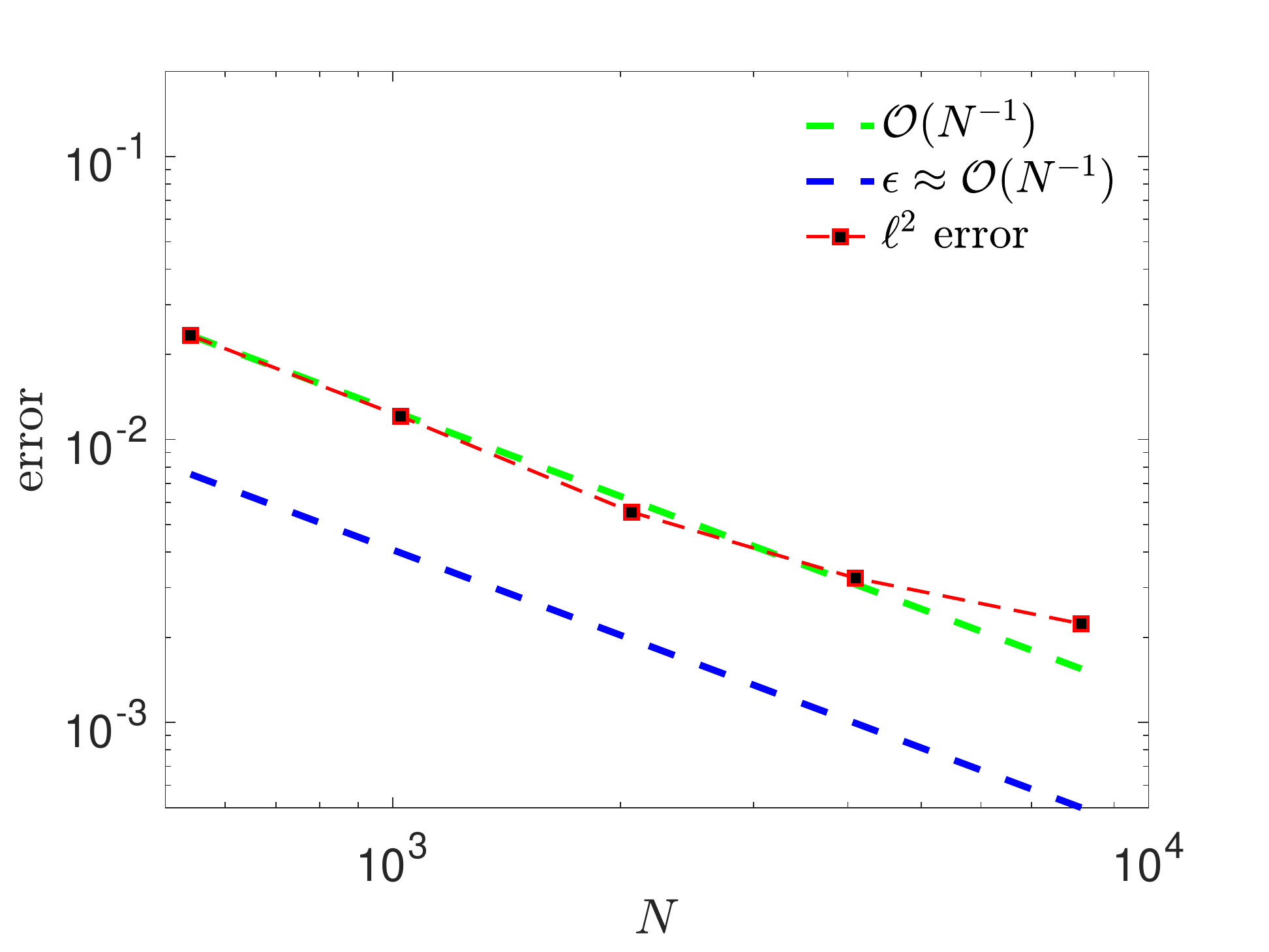} & %
\includegraphics[width=.45\linewidth]{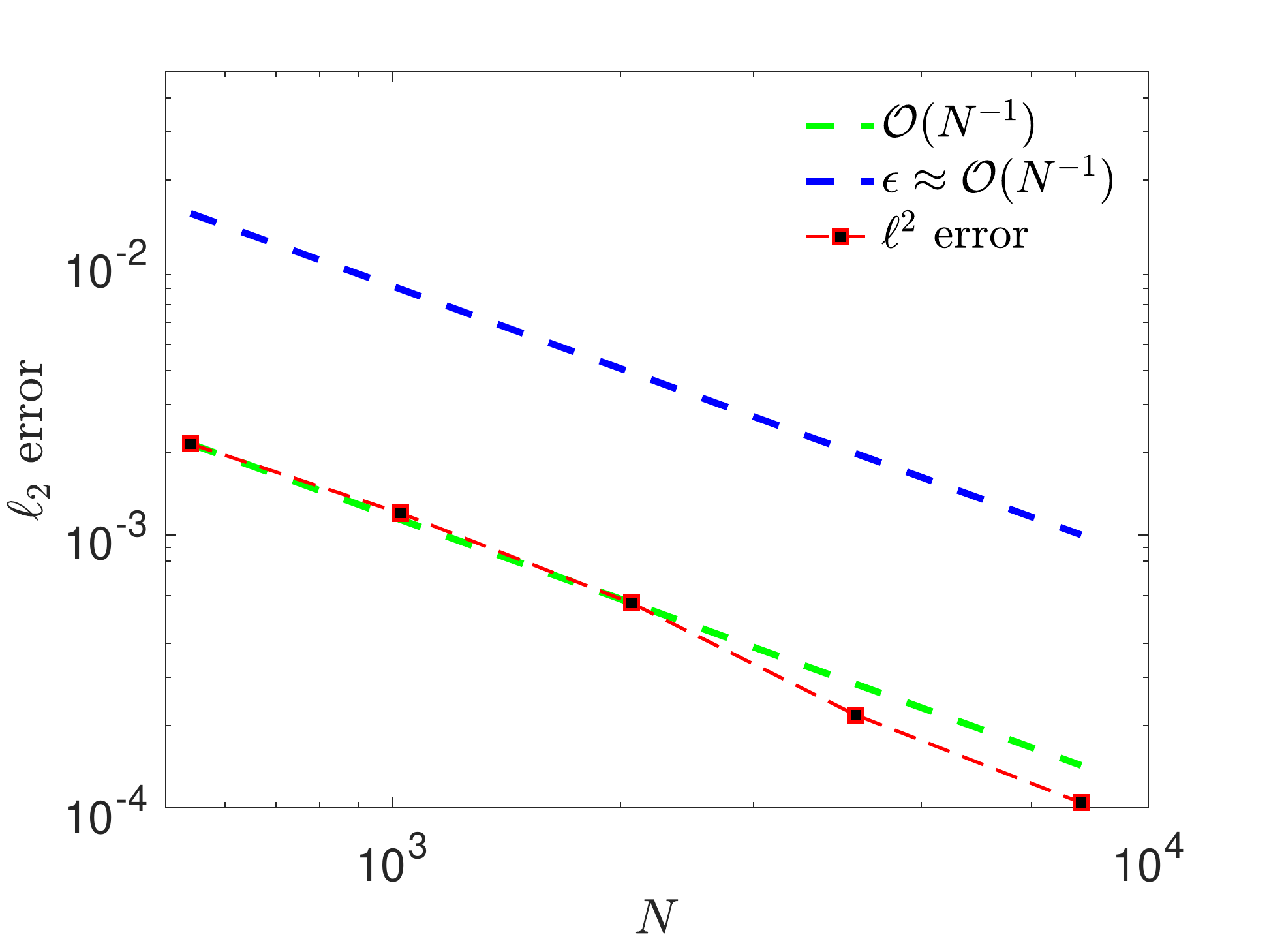}%
\end{tabular}
}
\caption{Annulus example: The $\ell^2$-error at $t=0.005$ with $\Delta t=10^{-4}$: (a)
Neumann boundary conditon. (b). Dirichlet boundary condition. }
\label{Fig3_err_2Ddisk}
\end{figure*}

\subsection{Parabolic equation on an unknown "cow" manifold}\label{cownumericalsection}

In this section, we are solving the PDE problem (\ref{Eqn:utbf}) on a two-dimensional "cow" manifold $x=(x_{1},x_{2},x_{3})\in M\subset \mathbb{R}^{3} $ with an unknown embedding function. The manifold is homeomorphic to the unit sphere, $S^{2},$ so there is no boundary for this example. The data of this surface is obtained from Keenan Crane's 3D repository \cite{crane}. In this example, we have no access to the analytic solution due to the unknown embedding function. For comparisons, we numerically solve the problem with the finite element method
(FEM) using the FELICITY FEM Matlab toolbox \cite{walker2018felicity}.

We first consider solving the heat equation $u_{t}=\Delta _{g}u+f$ with $f=1$
and the initial condition $u|_{t=0}(x_{1},x_{2},x_{3})=x_{1}+x_{2}+x_{3}$.
For the FEM solution, we provided FELICITY with the 2930 vertices and a
connectivity matrix for the triangle elements. For DM, the method is mesh
free and we used $k=200$ nearest neighbors and tuned the kernel bandwidth to
be $\epsilon =0.011$. The first row of Fig. \ref{Fig5_cow} shows the
relative difference between the two solutions at $t=0.01$ and $0.1$.
The relative difference is defined with respect to the FEM solution, that is,
$\Delta U^{\mathrm{rel}}(x_j,t):=
|U_{\mathrm{DM}}(x_j,t)-U_{\mathrm{FEM}}(x_j,t)|/|U_{\mathrm{FEM}}(x_j,t)|,$
where the denominator is not zero for this example. One can see from Fig. \ref{Fig5_cow}(a)
that the large relative difference is on the feet of the "cow" at the
initial $t=0.01$ and dissipates to smaller value at $t=0.1$. In Figure \ref{Fig4_err_cow}(a), we show that the maximum relative difference
between the two solutions, which is about $0.27$ at short time and dissipates to smaller than $0.05$ as time increases. This suggests that the DM solution and the FEM solution
are close to each other on almost all of the vertices for larger times.

\begin{figure*}[tbp]
{\scriptsize \centering
\begin{tabular}{cc}
{\normalsize (a) Difference b/w DM and FEM  } & {\normalsize (b) Difference b/w DM and FEM }\\
{\normalsize for Heat Eq. ($a=0$), at $t=0.01$} & {\normalsize for Heat Eq.($a=0$), at $t=0.1$}\\ 
\includegraphics[width=0.45\linewidth]{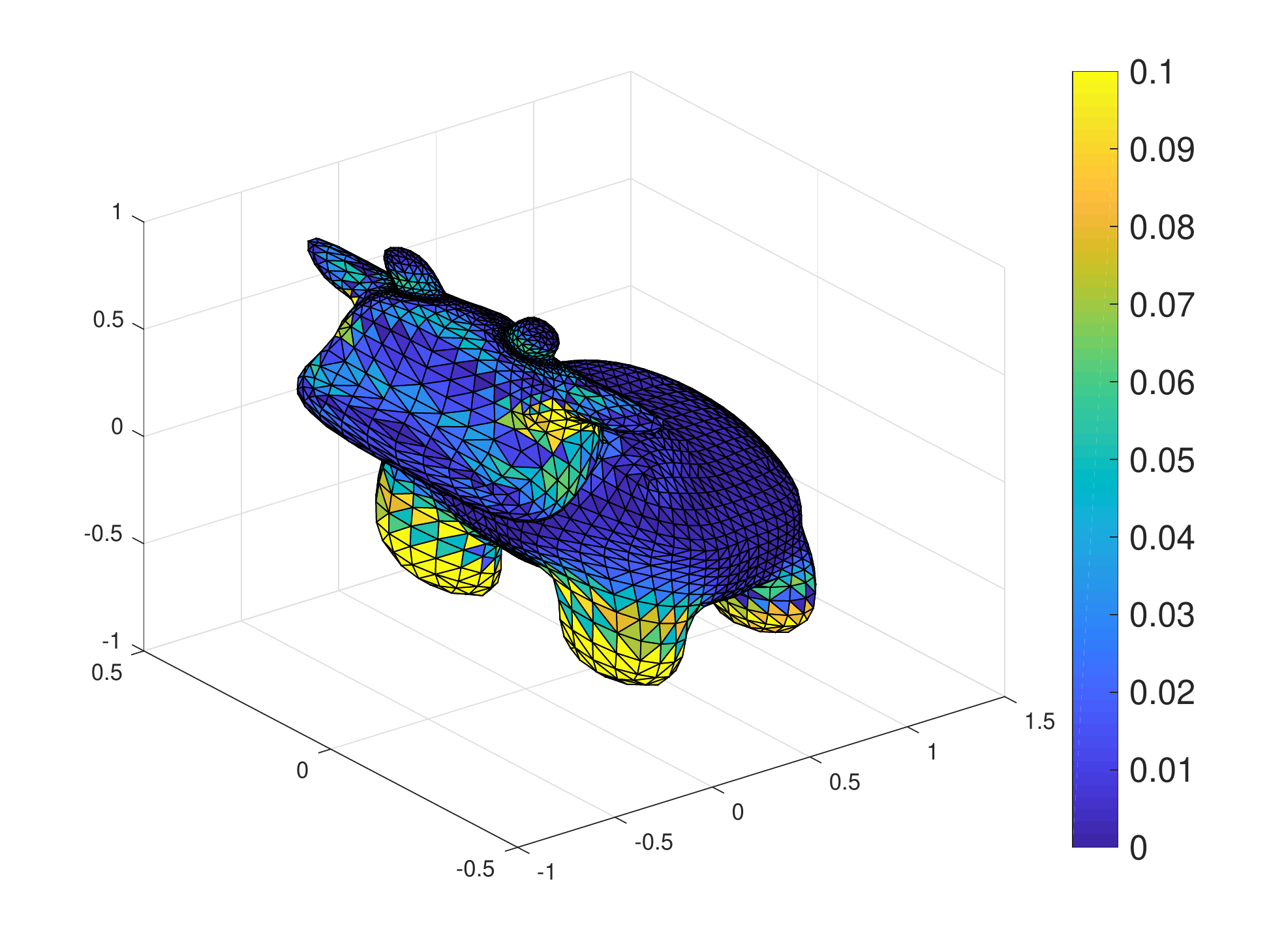} & %
\includegraphics[width=0.45\linewidth]{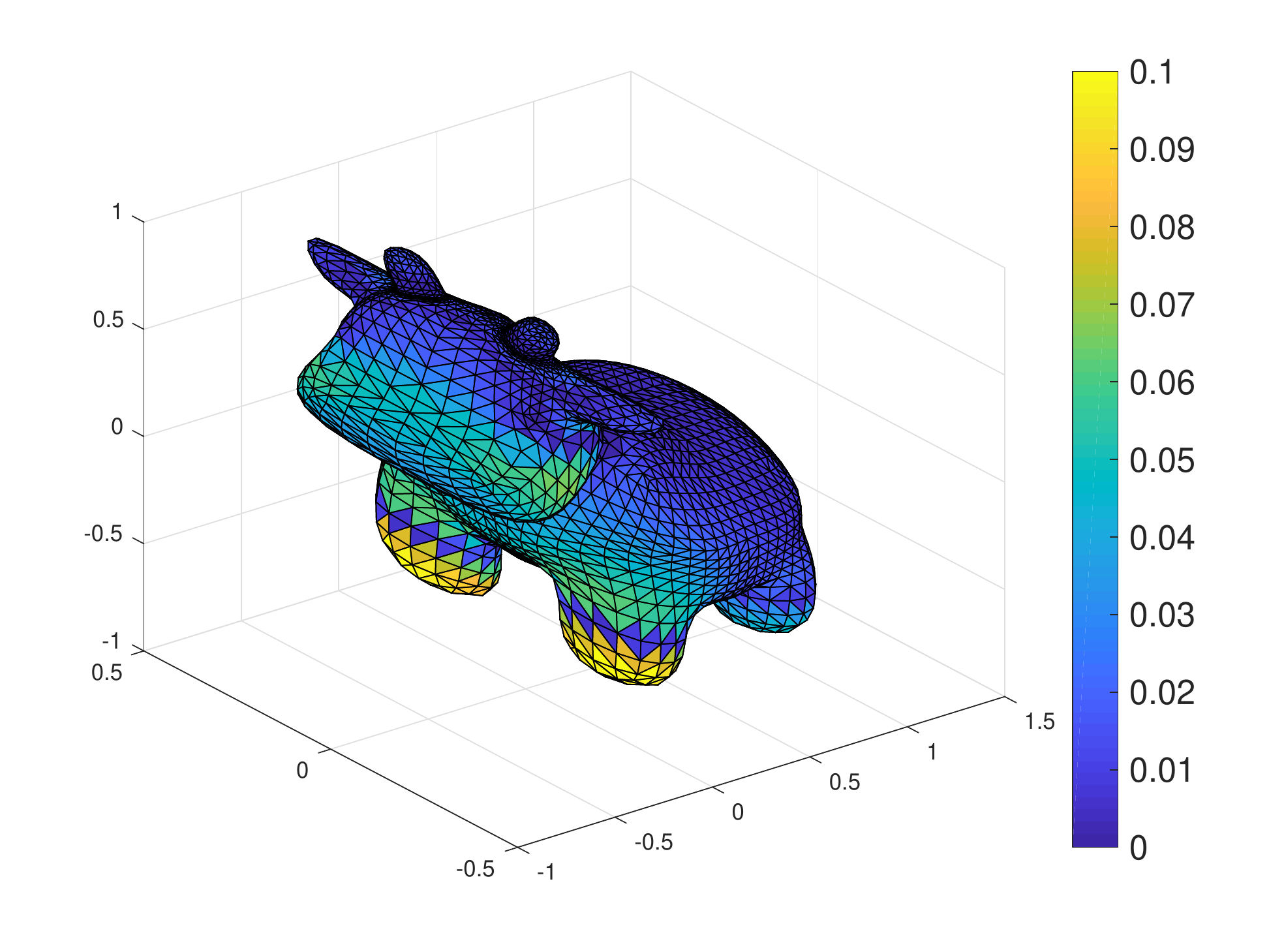}\\  %
{\normalsize (d) Difference b/w DM and FEM } & {\normalsize (e) Difference b/w DM and FEM }\\
{\normalsize for Adv.-Diff. Eq.($a\neq0$), at $t=0.01$} & {\normalsize for Adv.-Diff. Eq.($a\neq0$), at $t=0.1$%
} \\
\includegraphics[width=0.45\linewidth]{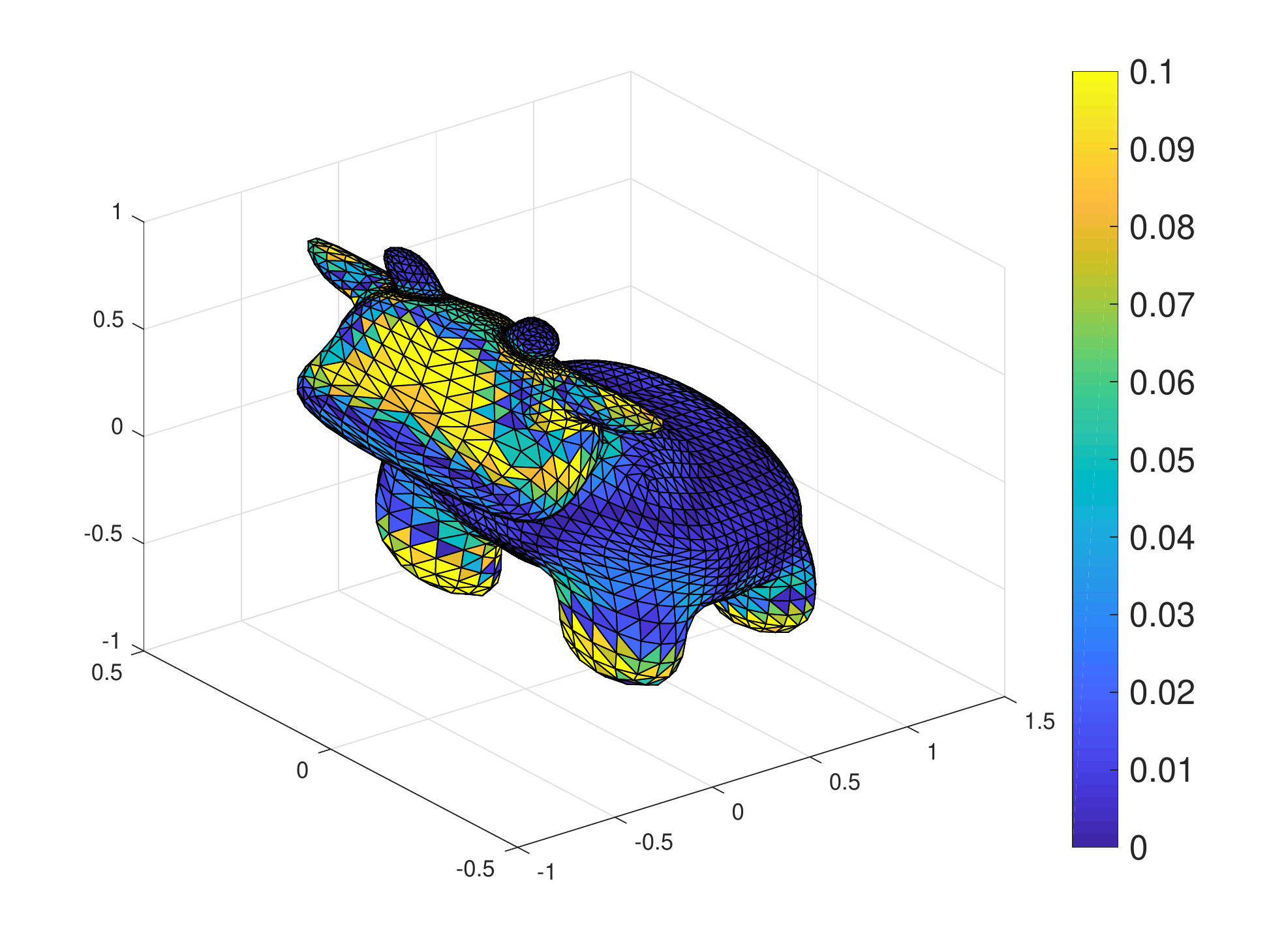} & %
\includegraphics[width=0.45\linewidth]{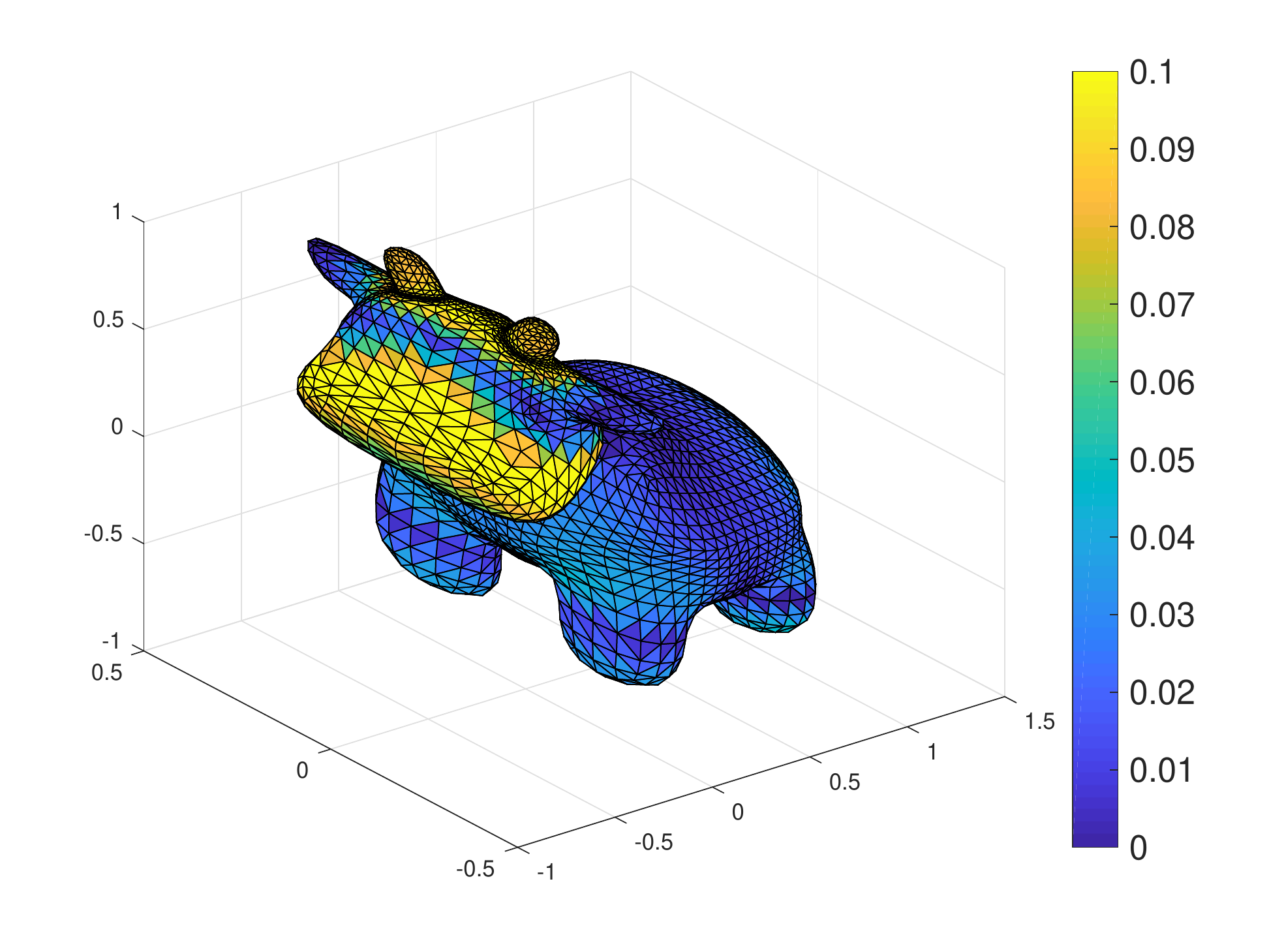}  %
\end{tabular}
}
\caption{Relative difference, defined as $\Delta U^{\mathrm{rel}}(x_j,t) =
|U_{\mathrm{DM}}(x_j,t)-U_{\mathrm{FEM}}(x_j,t)|/|U_{\mathrm{FEM}}(x_j,t)|$,
of DM solution to FEM solution on the "cow" manifold. First row: heat
equation (a) $t=0.01$. (b) $t=0.1$. Second row:
advection-diffusion equation (a) $t=0.01$. (b) $t=0.1$.
}
\label{Fig5_cow}
\end{figure*}

Next, we consider solving the advection diffusion equation $u_{t}=a\cdot
\nabla _{g}u+\Delta _{g}u$ with the same initial condition $%
u(x_{1},x_{2},x_{3})|_{t=0}=x_{1}+x_{2}+x_{3}$. In this numerical example, we define vector $a$ to be the projection of the vector $(1,1,1)$ onto the tangent space at each point on the manifold.  Here, we should point out that numerically $a(x)$ are specified differently between FEM and DM solvers. For the FEM, we provide the vector $a$
as the projection of $(1,1,1)$ onto the surface of each triangle element.
For the DM method, since this is a mesh-free solver, we estimate $a(x)$ at each grid point on the "cow" as follows. First, we use the SVD method introduced in \cite{berry2018iterated} (see also a short review in  Appendix~\ref{App:A}) to estimate the normal directions at each grid point. Then we can align the corresponding unit normal vectors in a consistent orientation using the method introduced in \cite{singer2011orientability}. Finally, we estimate vector $a$ by projecting the vector $(1,1,1)$ onto the tangent space spanned by the aligned estimated normal vectors.

The second
row of Fig. \ref{Fig5_cow} shows the the relative difference at $t=0.01$ and $%
0.1$. Fig. \ref{Fig4_err_cow}(b) shows the maximum relative
difference between two solutions. One can observe from Fig. \ref%
{Fig4_err_cow}(b) that the maximum relative difference is small about $0.12$
even if the vector fields $a(x)$ are estimated differently in DM and FEM, and it decays as time increases to about $0.02$. Thus,
even when the embedding function of a manifold is unknown, the DM method can
be used to accurately approximate solutions of (\ref{Eqn:utbf}) with only
vertices provided on the manifold.

\begin{figure*}[tbp]
{\scriptsize \centering
\begin{tabular}{cc}
{\normalsize (a) Heat Equation} & {\normalsize (b) Advection-Diffusion
Equation} \\
\includegraphics[width=.45\linewidth]{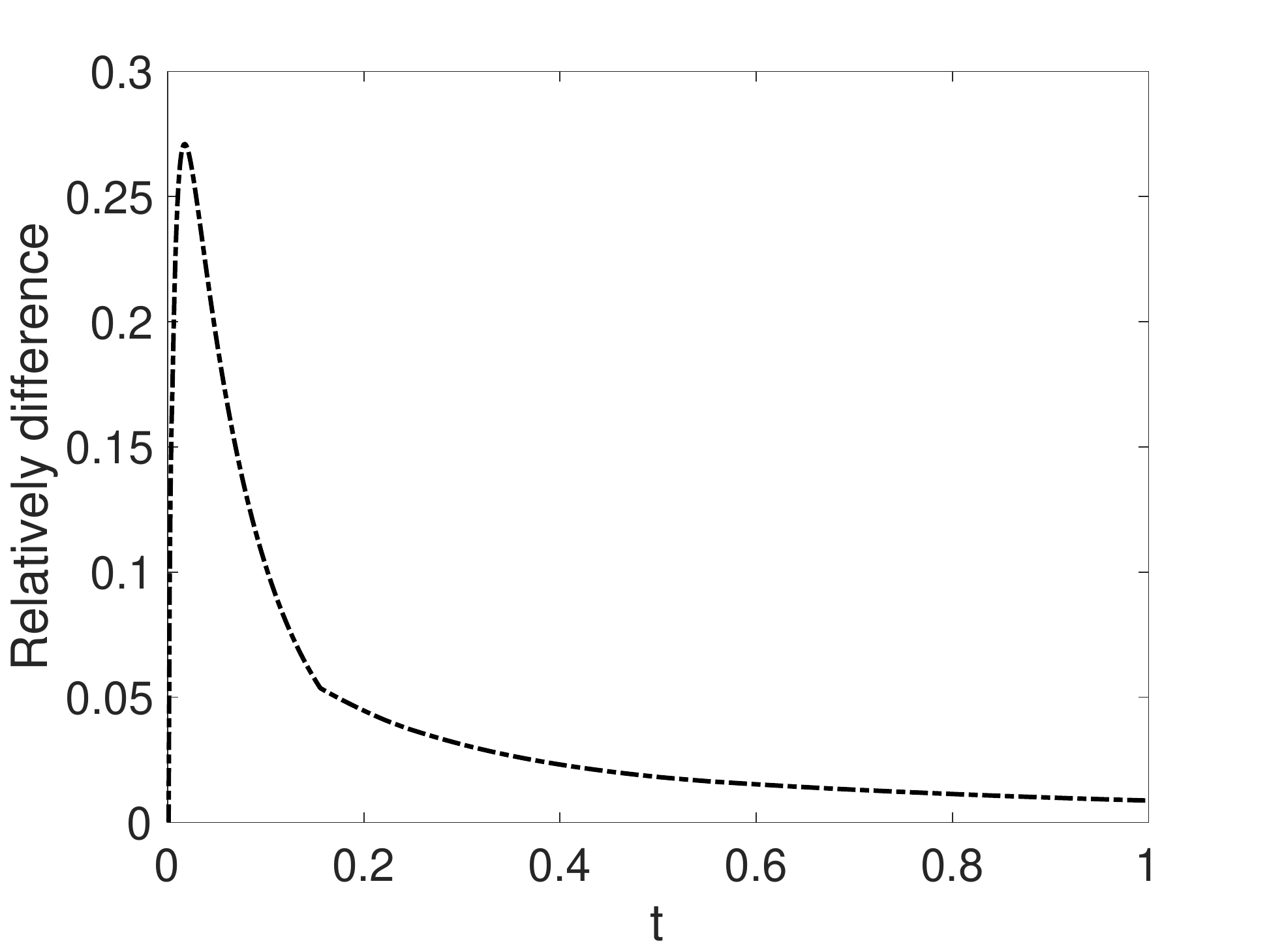} & %
\includegraphics[width=.45%
\linewidth]{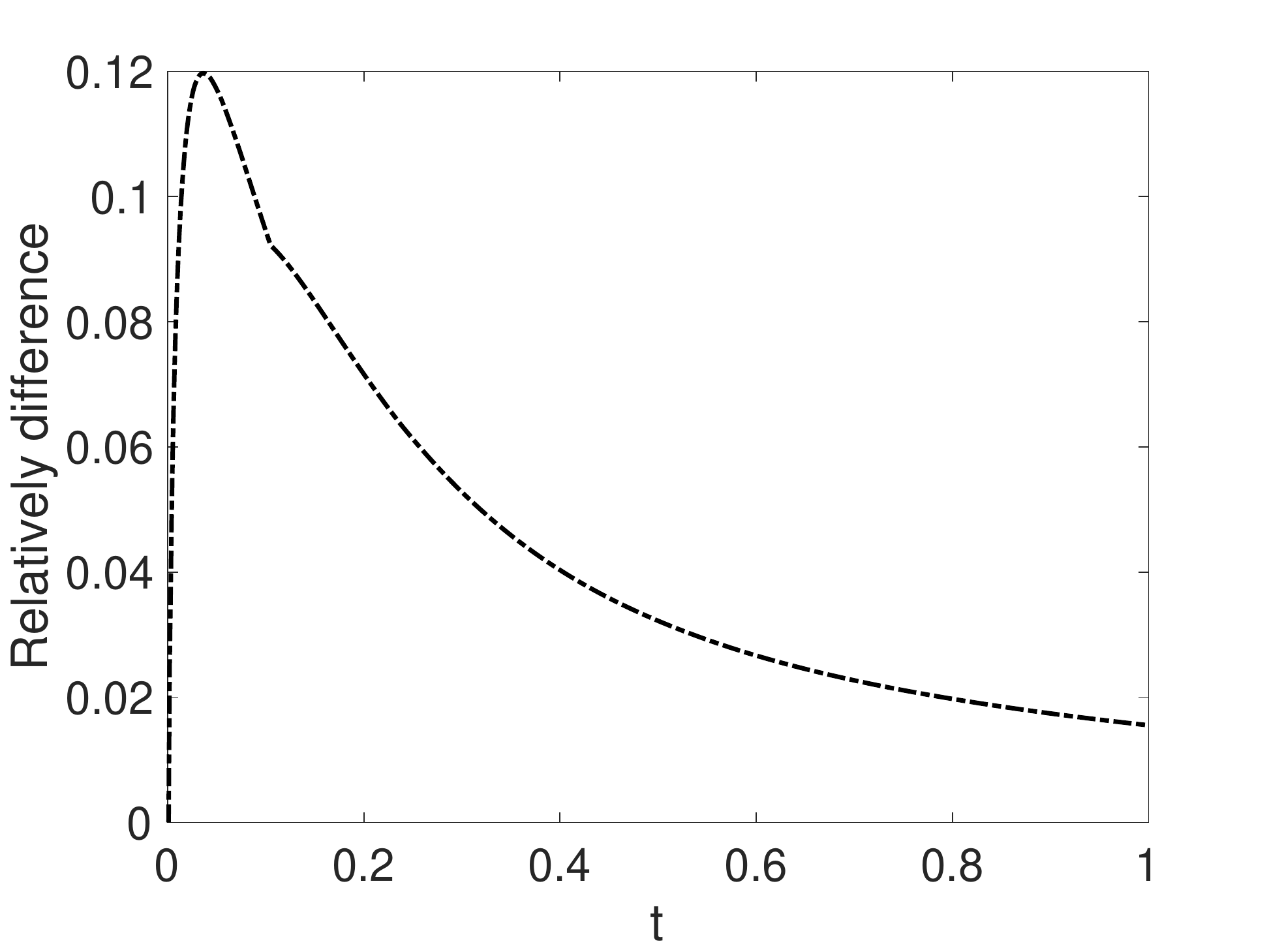}%
\end{tabular}
}
\caption{Maximum of relative difference, defined as $\max_{j=1,\ldots,N}%
\Delta U^{\mathrm{rel}}(x_j,t)$, as a function of time with $\Delta t=0.001 $%
: (a) Heat equation. (b) Advection-diffusion equation.}
\label{Fig4_err_cow}
\end{figure*}

\subsection{Parabolic equation on two simple manifolds with random data}\label{numericrandom}
In this section, we consider solving the heat equation on a full ellipse and a semi-torus with Dirichlet boundary conditions. The focus of this section is to demonstrate the robustness of the proposed method on randomly distributed grid points, sampled from a uniform distribution. 

\textbf{Full ellipse example}: We first consider solving the PDE on an ellipse which is defined with the usual embedding function,
$$
x:(x_1,x_2)=(\cos \theta, 2 \sin \theta)^{\top}, \qquad \theta\in[0,2\pi].
$$
\par The true solution is set to be $u(\theta,t)=e^{-t}\sin( \theta).$
Numerically, the grid points $\{\theta_i\}$ are randomly sampled from the uniform distribution on $[0,2\pi]$. For $N=[100,200,400,800,1600]$ grid points, we fix $k=100$ for all $N$. To demonstrate the robustness of the results, for each $N$, we will show solutions of $10$ independent trails. In each trail, we specify $\epsilon$ to be on the order of $N^{-2/7}$, which is the pointwise error estimate obtained by balancing the two terms in \eqref{Eqn:cbc}. For the time step, we set $\Delta t=10^{-4}$.
Fig. \ref{fig_semi_torus}(a) shows the results of 10 independent trials at $t=0.005$. Each red cross shows the uniform error for one trial, i.e., $\|\mathbf{U}^{50}_M -\mathbf{u}^{50}_M\|_{\infty}$, and the red dashes line shows the corresponding average over 10 trials.
Each magenta cross shows the $\ell^2$-error for the same trial, $\|\mathbf{U}^{50}_M -\mathbf{u}^{50}_M\|_{\ell^2}$, and the magenta dashes line shows the corressponding average over 10 trials. One can see that the mean of errors decrease roughly on order of bandwidth parameter  $\epsilon\sim N^{-2/7}$.

\textbf{Semi-torus example}: Here, we consider an example on a semi-torus defined with the following embedding function,
$$
x:=(x_1,x_2,x_3)=(2+\cos \theta) \cos \phi,(2+\cos \theta) \sin \phi,\sin \theta)
$$
for $0\leq\theta\leq 2\pi,0\leq\phi\leq\pi$ where $\phi=0$ and $\phi=\pi$ are the two boundaries of the semi-torus. The Riemannian metric is given by
$$
g_{x}(v, w)=v^{\top}\left(\begin{array}{cc}1 & 0 \\ 0 & (2+\cos \theta)^{2}\end{array}\right) w,\quad v,w\in T_xM.
$$
The boundary conditions are $u=0$ at both boundaries $\phi=0$ and $\phi=\pi$. The true solution is set to be $u=e^{-t}\sin( \theta)\sin(\phi).$ Numerically, the grid points $\{\theta_i,\phi_j\}$ are randomly sampled from the uniform distribution on $[0,2\pi]\times[0,\pi]$. For $N=[16^2, 23^2, 32^2, 45^2, 64^2]$ grid points, we fix $k$ such that  $k=200$ for all $N$. As in the experiment above, we also take $10$ independent trails. In each trail, we fix $\epsilon\sim N^{-1/4}$, which is the pointwise error estimate obtained by balancing the two terms in \eqref{Eqn:cbc}. For the time step, we set $\Delta t=10^{-4}$ as before.
Fig. \ref{fig_semi_torus}(b) shows results for 10 independent trials at $t=0.005$. One can see that both, the mean of the uniform and $\ell^{2}$, errors (averaged over 10 trials) decrease on order of $N^{-1/5}$, slightly slower than the rate of the bandwidth parameter, $\epsilon\sim N^{-1/4}$.

  \begin{figure}[tbp]%
  {\scriptsize \centering
\begin{tabular}{cc}
{\normalsize (a) On a full ellipse} & {\normalsize (b) On a semi-torus} \\
\includegraphics[width=.45\linewidth]{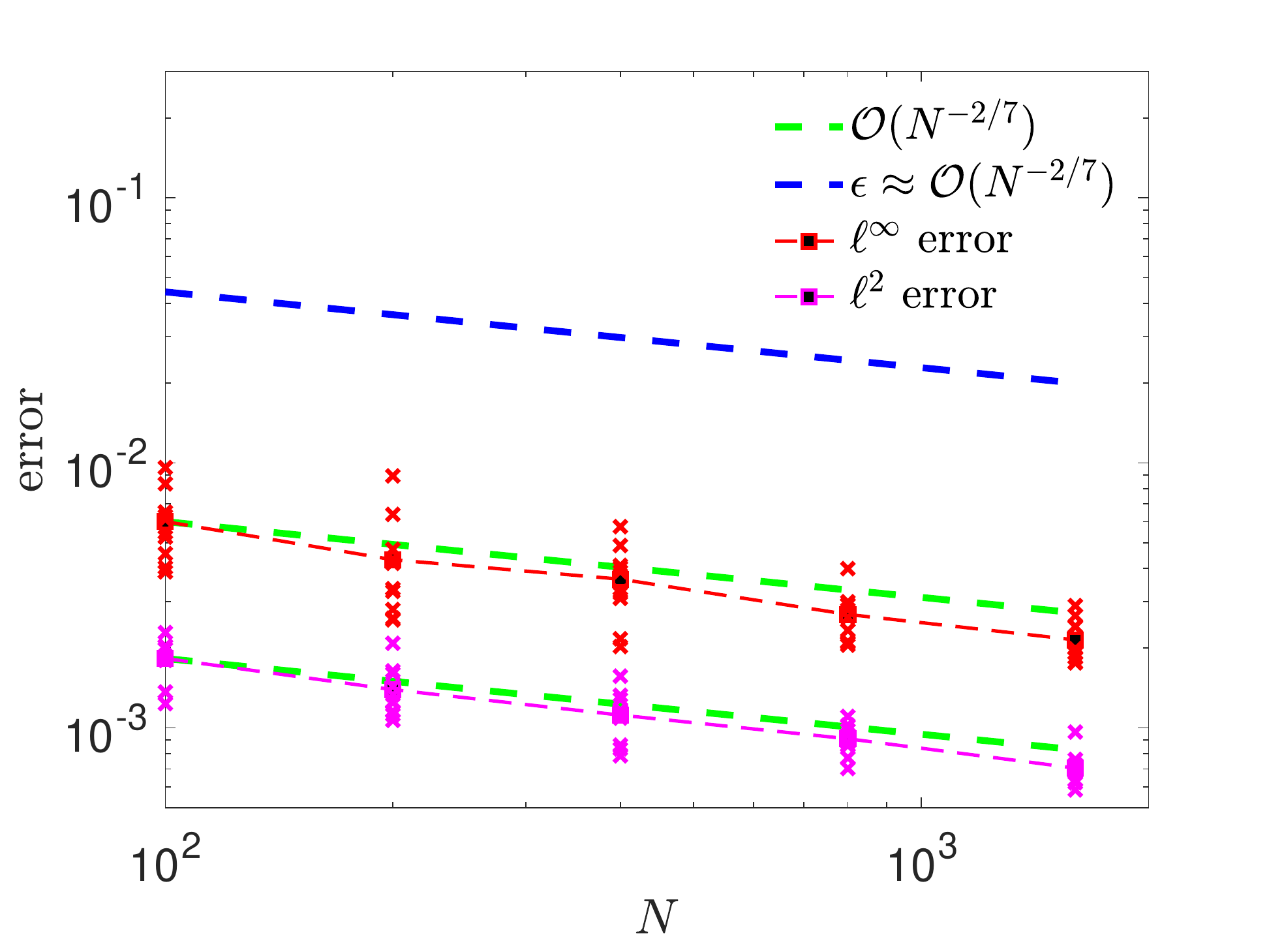} & %
\includegraphics[width=.45%
\linewidth]{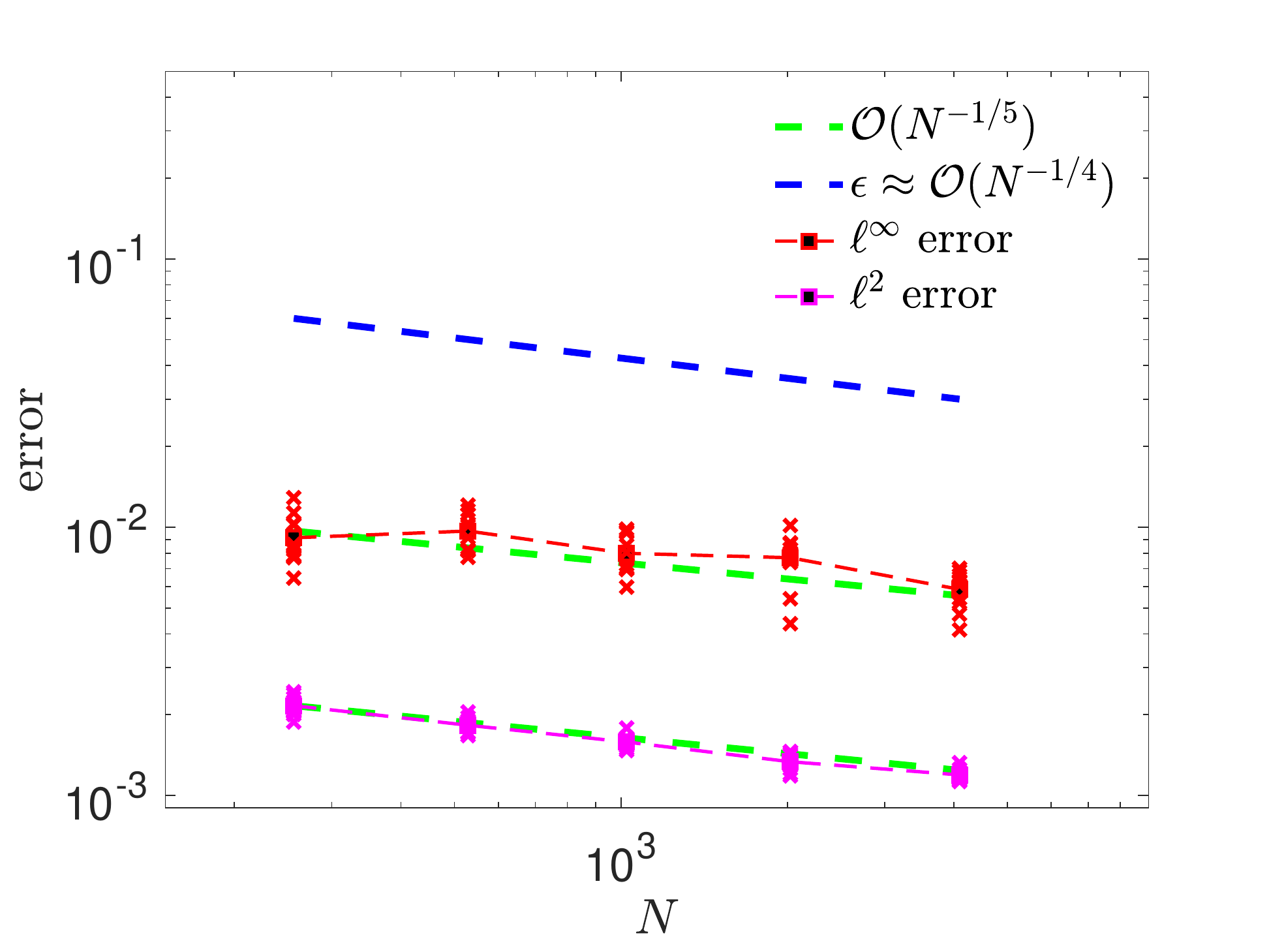}%
\end{tabular}
}
 \caption{Heat equation on simple manifolds with random data. We show $10$ independent trials for each $N$. Each red cross is the uniform error for one trial. Each magenta cross is the $\ell^2$-error for one trial. (a). Full ellipse. (b). Semi-torus with Dirichlet boundary.
 }
 \label{fig_semi_torus}
 \end{figure}

\section{Numerical solution of viscous Burger's equation}\label{section5}
In this section, we consider solving the nonlinear Burger's equation,
\begin{equation}
u_{t}=u \nabla _{g}u+\Delta _{g}u+f,
\label{burger}
\end{equation}%
for $u:M\to\mathbb{R}$ on a one-dimensional curve $M$ with homogeneous Dirichlet boundary condition. To solve the Burger's equation, we apply the pseudo-spectral method where we will use the approximated eigenfunctions and eigenvalues of the Laplace-Beltrami operator $\Delta_g$.

More precisely, we look for a solution of the form
$$
u(x,t_n)=\sum_{k=0}^K\hat{u}_k(t_n)\varphi_k(x),\quad (x,t)\in M\times [0,T],
$$
where $\{\varphi_k\}_{k=0}^\infty$ are eigenfunctions of Laplace-Beltrami operator. As before, we take backward difference for temporal discretization. Then we have the following scheme to approximate (\ref{burger})
\begin{equation}
\left\{\begin{aligned}
\sum_{k=0}^K\frac{\hat{u}_k(t_{n+1})-\hat{u}_k(t_n)}{\Delta t}\varphi_k(x)&=\sum_{k=0}^K\hat{u}_k(t_{n+1}) (u(x,t_n)\nabla_g\varphi_k(x)+\Delta_g\varphi_k(x)) \\
\sum_{k=0}^n\hat{u}_k(0)\varphi_k(x)&=u_{0}(x),\quad x\in M
\end{aligned}\right..
\label{2-2}
\end{equation}

\par For $1\leq k\leq n$, multiply the first equation in Eq.(\ref{2-2}) by $\varphi_k$ and integrate both sides over $M$,
$$
\frac{\hat{u}_k(t_{n+1})-\hat{u}_k(t_n)}{\Delta t}=\sum_{l=0}^K \langle u(x,t_n) \nabla_g\varphi_l,\varphi_k\rangle\hat{u}_l(t_{n+1})+c\lambda_k\hat{u}_k(t_{n+1})
$$
where we have used the fact that the eigenfunctions $\{\varphi_k\}_{k=0}^\infty$ are orthonormal under $L^2(M)$.
\par Next, we need to replace $\varphi_k$ and $\lambda_k$ by the approximate eigen-pairs $(\hat{\varphi}_k,\hat{\lambda}_k)$ of the following eigen problem
$$
\mathbf{L}_0\hat{\varphi}_k=\hat{\lambda}_k\hat{\varphi}_k
$$
where $\mathbf{L}_0$ is the GPDM matrix constructed as in Section 2 to approximate $\Delta_g$, accounting for the Dirichlet boundary condition. Here, the GPDM matrix is constructed with the local kernel in \eqref{Eqn:Kexy} with $A(x)=0$.
We denote $\mathbf{L}_1$ as the GPDM matrix that approximates the operator $\nabla_g+\Delta_g$ in 1D case.
Here, the GPDM matrix is constructed using the local kernel in \eqref{Eqn:Kexy} with $A(x)=D\iota(x) \in \mathbb{R}^2$, where $D\iota:M\to \mathbb{R}^2$ denotes the Jacobian of the embedding function $\iota$. In the numerical example below, we assume that the embedding is given. In case $D\iota$ is not available, one can estimate it with the SVD local regression method (see Appendix~\ref{App:A}) up to a sign difference and then align these vectors in a consistent orientation. Defining $\mathbf{L}_2:=\mathbf{L}_\mathbf{1}-\mathbf{L}_0$, one has
$
\mathbf{L}_2\varphi_k\approx \nabla_g\varphi_k
$ for 1D case.
To conclude, we have the following scheme to approximate the solution of (\ref{burger}),
$$
\left\{\begin{aligned}
 \frac{\hat{u}_k(t_{n+1})-\hat{u}_k(t_n)}{\Delta t}&=\sum_{l=0}^K \langle u(x,t_n)\mathbf{L}_2\hat{\varphi}_l,\hat{\varphi}_k\rangle\hat{u}_l(t_{n+1})+c\hat{\lambda_k}\hat{u}_k(t_{n+1}),\quad 1\leq k\leq K\\
\sum_{k=0}^n\hat{u}_k(0)\hat{\varphi}_k(x_i)&=u_{0}(x_i),\quad \{x_i\}_{i=1}^N\in M
\end{aligned}\right..
$$

Numerically, we consider solving the nonlinear Burger's equation in \eqref{burger} on an one-dimensional curve $(\theta ,\sin \theta )$ which is embedded in $\mathbb{R}%
^{2}$\ with $0\leq \theta \leq 4\pi $. The Riemannian metric is given by $%
g=1+\cos ^{2}\theta $. The boundary conditions are $u=0$ at both boundaries $%
\theta =0$ and $\theta =4\pi $. We set the true solution to be $u(\theta
,t)=e^{-t}\sin (\theta )$, so that the manufactured forcing term can be
calculated as
\begin{equation*}
f:=u_{t}-\left( u\cdot \nabla _{g}u+\Delta _{g}u\right) =u_{t}-uu_{\theta }-%
\frac{\cos \theta \sin \theta }{(1+\cos ^{2}\theta )^{2}}u_{\theta }-\frac{1%
}{1+\cos ^{2}\theta }u_{\theta \theta }.
\end{equation*}%
We numerically check the accuracy of the GPDM and compare it with the error of the standard DM. In particular, the DM solution is obtained by the same pseudo-spectral method described above, except that the matrices $\mathbf{L}_0$ and $\mathbf{L}_2$ are constructed using the standard diffusion maps asymptotic (without ghost points correction). Effectively, the DM solution does not account for the Dirichlet boundary conditions.

In this numerical test, the grid points $\left\{ \theta
_{i}\right\} _{i=1,\ldots ,N}$ are uniformly distributed on $[0,4\pi ]$ using various $N$ for convergence examination. Here, we apply the spectral method using $K=160$ eigenvectors. In Fig. \ref{Fig6_burger1d}(a) and (b), we compare the absolute error of DM
solution with that of GPDM solution for $N=800$, integrated with $\Delta t = 0.001$. To apply DM and GPDM, we set $k\geq 100$ for each $N$ and employ
 the auto-tuned kernel bandwidth to specify the bandwidth parameter, $\epsilon$. Figure \ref%
{Fig6_burger1d}(c) shows the convergence of both solutions at time $t=0.005$ with time step $\Delta t=10^{-4}$. From these numerical results, one can clearly observe that the error of GPDM is smaller compared to the error of DM. This is not surprising as DM ignores the effect of the Dirichlet boundary conditions.

\begin{figure*}[tbp]
{\scriptsize \centering
\begin{tabular}{ccc}
{\normalsize (a) Error of DM} & {\normalsize (b) Error of GPDM} &
{\normalsize (c) $\ell^2$ norm error} \\
\includegraphics[width=0.3\linewidth]{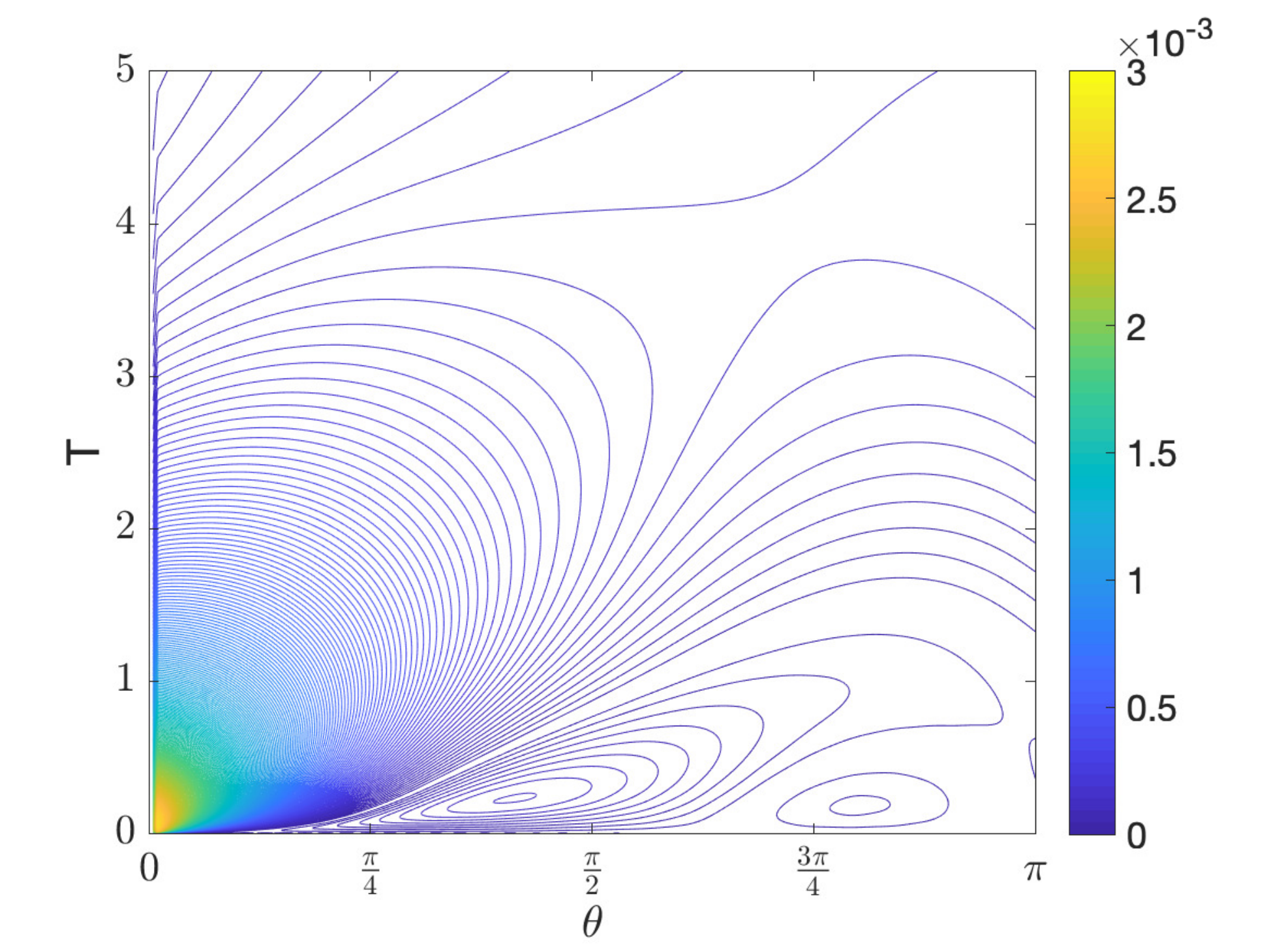} & %
\includegraphics[width=0.3\linewidth]{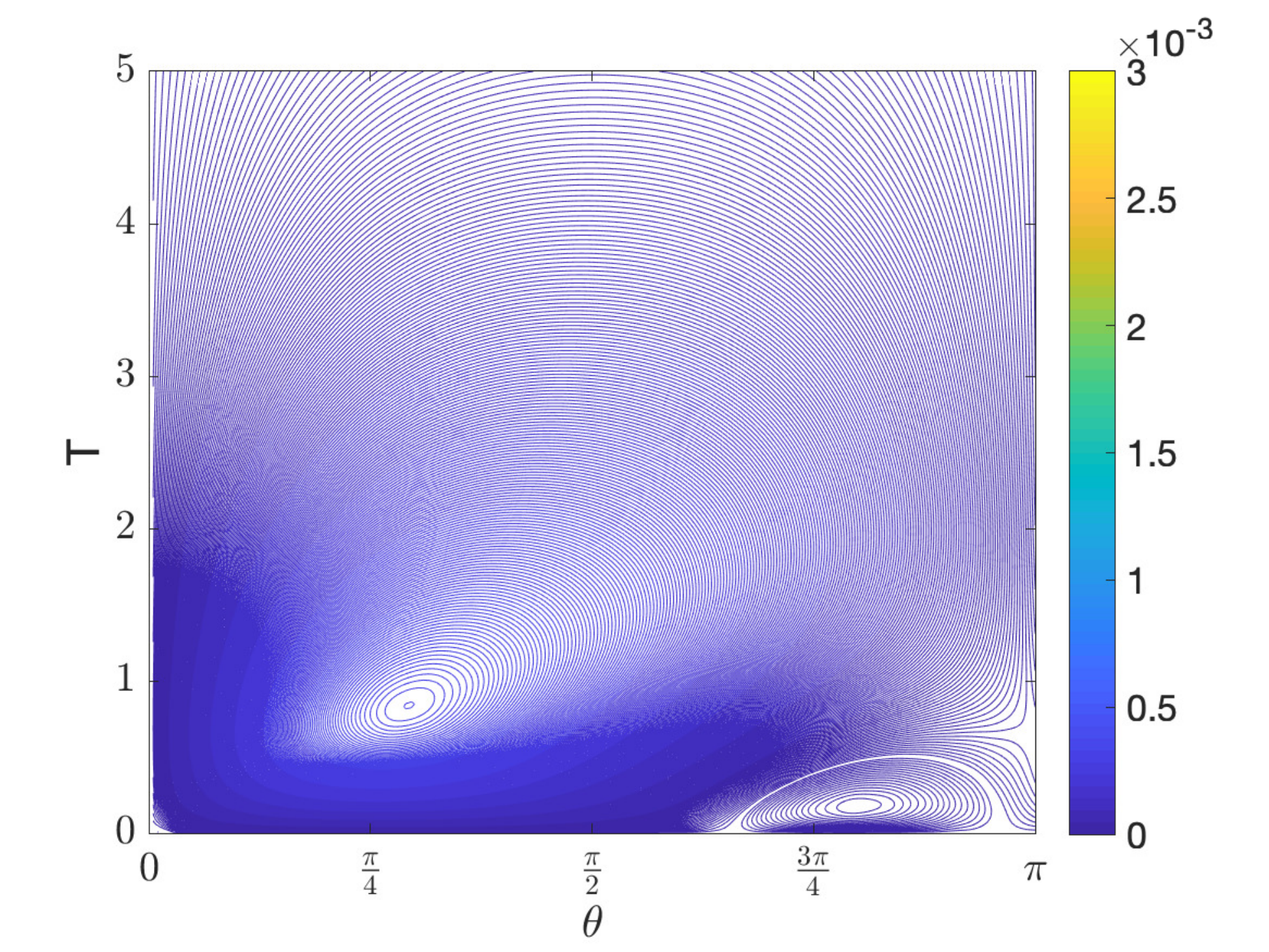} & %
\includegraphics[width=0.3\linewidth]{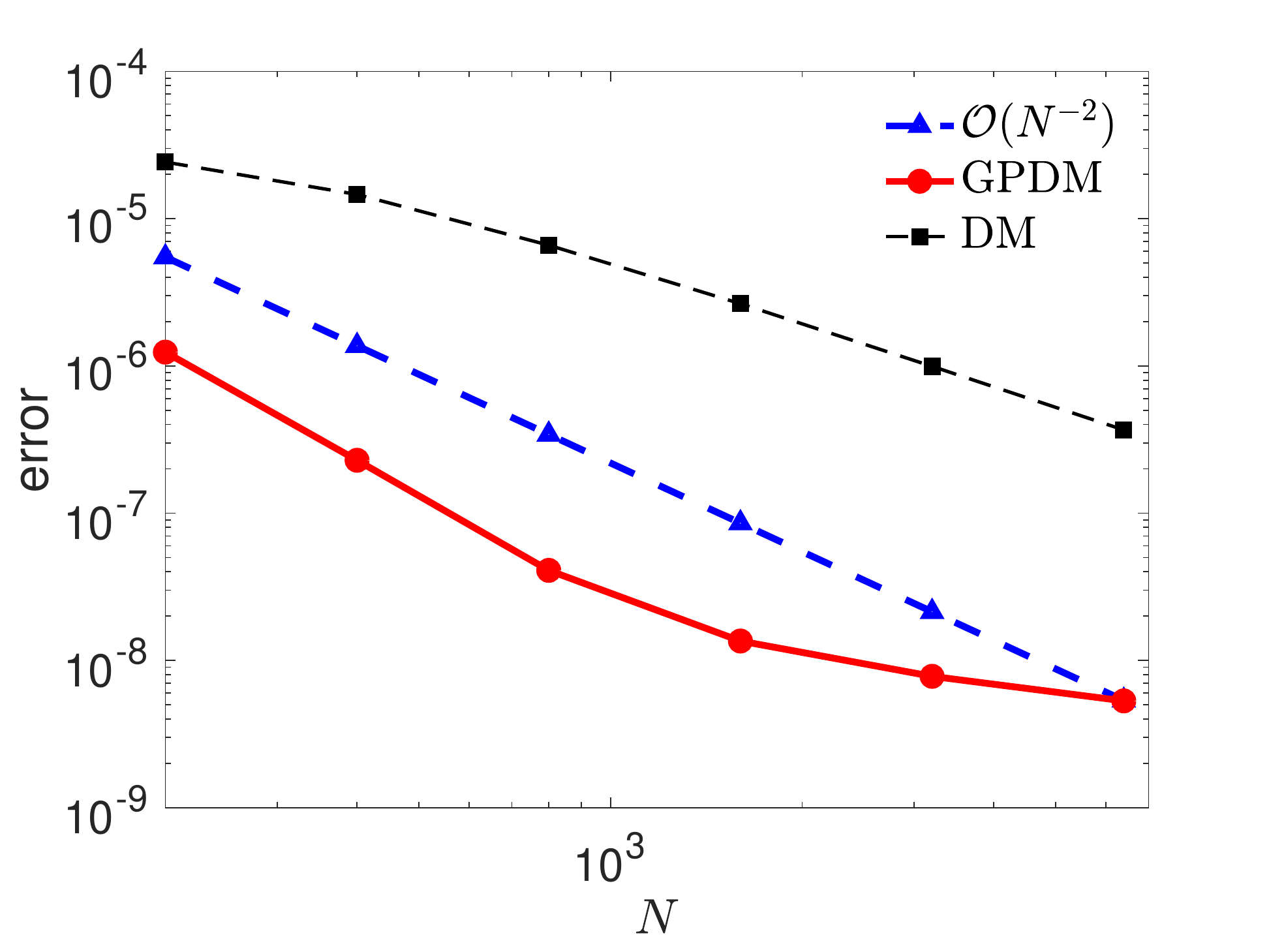}%
\end{tabular}
}
\caption{ Burger's equation: Absolute errors in space and time ($N=800,\Delta t=0.001,K=160$):
(a) DM for gradient operator and eigenvectors. (b) GPDM for gradient
operator and eigenvectors. (c) Comparison of the $\ell^2$-norm
errors, $\|\mathbf{U}^{50}_M -\mathbf{u}^{50}_M\|_{\ell^2}$, for $t=0.005,\Delta t=0.0001$ using $K=160$ eigenvectors. }
\label{Fig6_burger1d}
\end{figure*}

\section{Summary}\label{section6}
\par In this paper, we developed a mesh-free solver to solve the advection-diffusion parabolic-type PDE \eqref{Eqn:utbf} on smooth manifolds with Neumann, Dirichlet, and no boundary conditions. We analyzed the proposed method and showed that it is a convergent method under appropriate assumptions of the PDE problems. When the PDE is defined on a manifold without boundary, the uniform convergence rate of the proposed method is given in terms of the number of grid points and the kernel bandwidth parameter. When the manifold has boundaries, we proved the convergence in $\Vert\cdot\Vert_{\ell^2}$ for well-sampled data and in the uniform sense for randomly sampled data. While the convergence topology seems to be stronger for the randomly sampled data, this convergence is achieved by imposing a bandwidth parameter $\epsilon$ that decays at a much slower error rate as a function of $N$, compared to what we empirically found for well-sampled data. In particular, for 2-dimensional manifold problems, we empirically found that $\epsilon\sim N^{-1}$ for well-sampled data in $\ell^2$-sense and $\epsilon\sim N^{-1/5}$ in the uniform sense, which is slightly slower than the theory $N^{-1/4}$. We should also point out that the rate $N^{-1/4}$, deduced by balancing the two error terms in \eqref{Eqn:cbc}, is obtained for the Monte-Carlo error estimate. For well-sampled data, we suspect that we can also achieve uniform error with a slower convergence rate with a carefully chosen, $\epsilon$. Theoretically, however, since the secant line approximation is used to estimate the normal vectors along the boundary, it remains difficult to justify choosing $h$ to be other than $\epsilon^{1/2}$ when the interior points are well aligned in such a very special arrangement.

Numerically, we verified the effectiveness of the proposed method on some simple manifolds with well-sampled and randomly sampled grid points, and also on an unknown ``cow'' manifold. In these, examples, we showed the accuracy of the solvers and validate the convergence rates deduced in the theoretical study. We also demonstrated a numerical experiment of solving nonlinear viscous Burger's equation. While the result is promising, more works need to be done to extend this approach to solve vector-valued PDEs on manifolds. In particular, one possibly needs to extend diffusion maps to approximate Laplacian on 1-form, using various tools such as the Spectral Exterior Calculus  \cite{berry2020spectral}. More importantly, one needs to estimate the covariant derivative as a generalization of the nonlinear advection term on manifolds.

\section*{Acknowledgment}

The research of JH was partially supported under the NSF grant DMS-1854299. This research was supported in part by a Seed Grant award from the Institute for Computational and Data Sciences at the Pennsylvania State University. The authors also thank Faheem Gilani for providing an initial sample code of FELICITY FEM.

%
%

\appendix

\section{Kernel method for approximating normal vectors at the boundary}\label{App:A}

Assume that randomly sampled point clouds, $\{x_{i}\}_{i=1}^{N}$, lie on a $%
d$-dimensional manifold ${M\subseteq }\mathbb{R}^{m}${. Among these }${N}${%
\ data points, }${B}${\ boundary points, $\{x_{b}\}_{b=1}^{B}$, lie on the }$%
{(d-1)}${-dimensional boundary }$\partial {M}${. Our goal here is to
approximate normal vectors at these boundary points using} a kernel-based
weighted linear regression method as introduced in Corollary 3.2. of \cite%
{Berry2016IDM}.

\begin{algm}
\label{algm5_1} Kernel method for approximating normal vectors:

\begin{enumerate}
\item For a boundary point $x\in ${$\{x_{b}\}_{b=1}^{B}\subseteq \partial M$}%
, define $\mathbf{X}$ to be the $m\times k$ matrix with each column $\mathbf{%
X}_{j}=D(x)^{-1/2}\exp (-\Vert x-x_{j}\Vert ^{2}/4\epsilon )(x_{j}-x)$ where
$D(x)=\sum_{j=1}^{k}\exp \left( -\Vert x-x_{j}\Vert ^{2}/2\epsilon \right) $%
. Here, $x_{j}$ ($j=1,\ldots ,k$) are the $k>d$ nearest neighbors of the
boundary point $x$ chosen from the {$\{x_{i}\}_{i=1}^{N}\subseteq M$.}\ {The
bandwidth $\epsilon $ is specified using these }${k}${\ neighbors based on
the automated tuned method as in (\ref{scalingS}).}

\item Compute the left-singular vectors of $\mathbf{X}$\ using the singular
value decomposition (SVD) to obtain {$\boldsymbol{\tilde{t}}_{1},\boldsymbol{%
\tilde{t}}_{2},\ldots ,\boldsymbol{\tilde{t}}_{d}\in $}$\mathbb{R}^{m}${\
that span the tangent space of }${M}${\ at each boundary point.} Here, the
leading largest $d$ singular values of matrix $\mathbf{X}$ will be of order-$%
\sqrt{\epsilon }$ with the associated left-singular vectors {$\boldsymbol{%
\tilde{t}}_{1},\boldsymbol{\tilde{t}}_{2},\ldots ,\boldsymbol{\tilde{t}}_{d}$%
}\ parallel to the tangent space of $M$. Thus, {the error estimates of the
leading singular vectors $\boldsymbol{\tilde{t}}_{1},\boldsymbol{\tilde{t}}%
_{2},\ldots ,\boldsymbol{\tilde{t}}_{d}$ for approximating the tangent
vectors of }${M}${\ at boundary point }${x}${\ are also of order-$\sqrt{%
\epsilon }$ (see Appendix A of \cite{Berry2016IDM} for detailed discussion).
}The remaining $\min \{m,k\}-d$ smaller singular values will be of order-$%
\epsilon $ with the left-singular vectors orthogonal to the tangent space of
$M$.

\item Repeat Steps 1 and 2 to obtain{\ the }${(d-1)}${\ tangent vectors, $%
\boldsymbol{\tilde{s}}_{1},\ldots ,\boldsymbol{\tilde{s}}_{d-1}$, }%
associated with the leading largest $(d-1)$ singular values\ but using $k$
nearest neighbors of $x$ chosen from the {$\{x_{b}\}_{b=1}^{B}\subseteq
\partial M$ for each }$x\in ${$\partial M$}. These ${(d-1)}${\ tangent
vectors are approximated to span the }${(d-1)}${\ dimensional boundary $%
\partial M$. The bandwidth $\epsilon _{0}$ is also specified using these }${k%
}${\ nearest neighbors from $\partial M$.}

\item Calculate the normal direction $\boldsymbol{\tilde{\nu}}$ by
subtracting the orthonormal projection of $\boldsymbol{\tilde{t}}_{p}$\ onto
$Span\{\boldsymbol{{\boldsymbol{\tilde{s}}}}_{1}\boldsymbol{{,\ldots ,%
\boldsymbol{\tilde{s}}}}_{d-1}\}$\ from the tangent vector $\boldsymbol{%
\tilde{t}}_{p}$\ for some $p\in \{1,\ldots ,d\}$\ using the Gram--Schmidt
process or QR decomposition,
\begin{equation}
\boldsymbol{\tilde{\nu}}=\frac{\boldsymbol{\tilde{t}}_{p}-\sum_{i=1}^{d-1}%
\left\langle \boldsymbol{\tilde{t}}_{p},\boldsymbol{\tilde{s}}%
_{i}\right\rangle \boldsymbol{\tilde{s}}_{i}}{\Vert \boldsymbol{\tilde{t}}%
_{p}-\sum_{i=1}^{d-1}\left\langle \boldsymbol{\tilde{t}}_{p},\boldsymbol{%
\tilde{s}}_{i}\right\rangle \boldsymbol{\tilde{s}}_{i}\Vert },  \notag
\end{equation}%
where $\left\langle \boldsymbol{a},\boldsymbol{b}\right\rangle $ denotes the
inner product between vectors $\boldsymbol{a},\boldsymbol{b}\in \mathbb{R}%
^{m}$ and all tangent vectors $\{{\boldsymbol{\tilde{t}}_{1},\boldsymbol{%
\tilde{t}}_{2},\ldots ,\boldsymbol{\tilde{t}}_{d}}\}$ and $\{\boldsymbol{{%
\boldsymbol{\tilde{s}}}}_{1}\boldsymbol{{,\ldots ,\boldsymbol{\tilde{s}}}}%
_{d-1}\}$ are unit vectors generated from SVD algorithm. The error estimate
for the normal direction $\boldsymbol{\tilde{\nu}}$ is thereafter $\mathcal{O%
}(\sqrt{\epsilon },\sqrt{\epsilon _{0}})$, that is, $\Vert \boldsymbol{\nu }-%
\boldsymbol{\tilde{\nu}}\Vert =O(\sqrt{\epsilon },\sqrt{\epsilon
_{0}})$. {Here, basically }$\boldsymbol{\tilde{\nu}}${\ is a vector in the
tangent space of }${M}${, }${Span}\{{\boldsymbol{\tilde{t}}_{1},\boldsymbol{%
\tilde{t}}_{2},\ldots ,\boldsymbol{\tilde{t}}_{d}}\}$, {but also
perpendicular to the boundary }$Span\{\boldsymbol{{\boldsymbol{\tilde{s}}}}%
_{1}\boldsymbol{{,\ldots ,\boldsymbol{\tilde{s}}}}_{d-1}\}${\ as well. }

\item {Determine the sign of $\boldsymbol{\tilde{\nu}}$ from the orientation
of the manifold $M$ }by comparing {$\boldsymbol{\tilde{\nu}}$}\ with the
mean of vectors connecting $x$ and its $k$-nearest neighbors.
\end{enumerate}
\end{algm}

\section{Proof of Lemma~\ref{lemstabNeubou}}
\label{App:B}
{\ First, we prove that }$\Vert (\mathbf{I}-\Delta t\mathbf{N})^{-1}\Vert
_{\infty }\leq 1$. To obtain this result, we need the following two
properties of the matrix $\mathbf{N}$:
\begin{enumerate}
\item[1)] The diagonal entries are negative {$N_{ii}<0$ and non-diagonal entries are
non-negative $N_{ij}\geq 0$ for $j\neq i$; }
\item[2)] The row sum of $\mathbf{N}$ is zero , that is,
$\sum _{j=1}^{N-B}N_{ij}=0$ for $%
i=1,\ldots ,N-B$.
\end{enumerate}
{For the convenience of discussion, we define an }$(N-B)\times \bar{N}${\
sub-matrix of }$\mathbf{L}^{h}$ in {(\ref{eqnLtild}) for its }${N-B}${\ rows
corresponding to interior points,
\begin{equation}
(\mathbf{L}^{h})_{(N-B)\times \bar{N}}:=\left( \mathbf{L}^{I,I},\mathbf{L}%
^{I,B},\mathbf{L}^{I,G}\right) \in (\mathbb{R}^{(N-B)\times (N-B)},\mathbb{R}%
^{(N-B)\times B},\mathbb{R}^{(N-B)\times BK})=\mathbb{R}^{(N-B)\times \bar{N}%
},  \label{eqnLh2}
\end{equation}%
where each column of $\mathbf{L}^{I,I},\mathbf{L}^{I,B},$}$\mathbf{L}^{I,G}${%
\ corresponds to an interior point, a boundary point, and an exterior ghost
point, respectively. We also separate matrix }$\mathbf{G}\in \mathbb{R}^{{B}{%
K}\times N}$ in {(\ref{GPDMmatrix}) into two parts,}%
\begin{equation*}
\mathbf{G:=(G}^{G,I},\mathbf{G}^{G,B}\mathbf{)}\in (\mathbb{R}^{BK\times
(N-B)},\mathbb{R}^{BK\times B}),
\end{equation*}%
where {each column of }$\mathbf{G}^{G,I}$ and $\mathbf{G}^{G,B}${\
corresponds to an interior point and a boundary point, respectively. From the
extrapolation formula (\ref{Eqn:uvv_g2}), we have
\begin{equation*}
U_{b,k}=\left( k+1\right) u(x_{b})-ku(\tilde{x}_{b,0}).
\end{equation*}%
Then the submatrices }$\mathbf{G}^{G,I}$ and $\mathbf{G}^{G,B}$ can be
expressed as%
\begin{equation}
\mathbf{G}^{G,I}=\left(
\begin{array}{ccc}
& \vdots &  \\
\cdots & G_{\left\{ \left( b,k\right) ,(b,0)\right\} }=-k & \cdots \\
& \vdots &
\end{array}%
\right) _{BK\times (N-B)},\text{ \ \ }\mathbf{G}^{G,B}=\left(
\begin{array}{ccc}
& \vdots &  \\
\cdots & G_{\left\{ \left( b,k\right) ,b\right\} }=k+1 & \cdots \\
& \vdots &
\end{array}%
\right) _{BK\times B},  \label{EqnGGI}
\end{equation}%
where $G_{\left\{ \left( b,k\right) ,(b,0)\right\} }$ is the only nonzero
entry lies in the row of ghost $\tilde{x}_{b,k}$ and the column of interior $%
\tilde{x}_{b,0}$, $G_{\left\{ \left( b,k\right) ,b\right\} }$ is the only
nonzero entry lies in the row of ghost $\tilde{x}_{b,k}$ and the column of
boundary $x_{b}$. With these definitions, we can use {(\ref{GPDMmatrix}) }to
rewrite matrix ${\mathbf{N}}$,%
\begin{eqnarray}
{\mathbf{N}} &=&\mathbf{\tilde{L}}^{I,I}+\mathbf{\tilde{L}}^{I,B}\mathbf{E}%
^{B,I}=(\mathbf{L}^{I,I}+\mathbf{L}{^{I,G}\mathbf{G}^{G,I})}+(\mathbf{L}%
^{I,B}+\mathbf{L}^{I,G}\mathbf{G}^{G,B})\mathbf{E}^{B,I}  \notag \\
&=&\mathbf{L}^{I,I}+\mathbf{L}^{I,B}\mathbf{E}^{B,I}+\mathbf{L}{^{I,G}(%
\mathbf{G}^{G,I}+\mathbf{G}^{G,B}\mathbf{E}^{B,I}):=}\mathbf{L}^{I,I}+%
\mathbf{L}^{I,B}\mathbf{E}^{B,I}+\mathbf{L}{^{I,G}\mathbf{\tilde{G}}^{G,I},}
\label{defNN}
\end{eqnarray}%
where we have defined $\mathbf{\tilde{G}}^{G,I}:=
\mathbf{G}^{G,I}+\mathbf{G}^{G,B}\mathbf{E}^{B,I}$. Using the definitions of ${\mathbf{G}^{G,I}}$\ and ${\mathbf{G}^{G,B}}$\ in {%
(\ref{EqnGGI}) and }${\mathbf{E}^{B,I}}${\ in (\ref{NeumannD}), one can
calculate }${\mathbf{\tilde{G}}^{G,I}}${\ to obtain }%
\begin{equation*}
{\mathbf{\tilde{G}}^{G,I}=}\left(
\begin{array}{ccc}
& \vdots &  \\
\cdots & G_{\left\{ \left( b,k\right) ,(b,0)\right\} }=1 & \cdots \\
& \vdots &
\end{array}%
\right) _{BK\times (N-B)}.
\end{equation*}%
For ${\mathbf{\tilde{G}}^{G,I}}${, the only nonzero entry is }$G_{\left\{
\left( b,k\right) ,(b,0)\right\} }=1$ for $b=1,\ldots ,B$ and $k=1,\ldots ,K$%
. Thus, for $\mathbf{L}^{I,B}\mathbf{E}^{B,I}+\mathbf{L}{^{I,G}\mathbf{%
\tilde{G}}^{G,I}}$ in {(\ref{defNN}), all of its entris are non-negative, in
particular, all its diagonal entries are zero. Also notice that the diagonal
entries of }$\mathbf{L}^{I,I}$ are negative and all its non-diagonal entries
are non-negative. So far, {we have proved the first property, that is, for
all }${i=1,\ldots ,N-B}${, }$\mathbf{N}${\ has negative }diagonal {$N_{ii}<0$
and non-negative non-diagonal $N_{ij}\geq 0$ for $j\neq i$.}

We now verify the second property of $\mathbf{N}$\textbf{,} its every {row
sum is zero. Using the definition in (\ref{defNN}) and define all one vector
}$\mathbf{1}_{k}${\ with length }${k}${, one can show that }%
\begin{eqnarray}
\mathbf{N1}_{N-B} &=&(\mathbf{L}^{I,I}+\mathbf{L}^{I,B}\mathbf{E}^{B,I}+%
\mathbf{L}{^{I,G}\mathbf{\tilde{G}}^{G,I})}\mathbf{1}_{N-B} \notag\\
&=&\mathbf{L}^{I,I}\mathbf{1}_{N-B}+\mathbf{L}^{I,B}\mathbf{1}_{B}+\mathbf{L}%
{^{I,G}}\mathbf{1}_{BK}\mathbf{.} \notag \\
&=&\mathbf{0}, \label{eqnLrowsum}
\end{eqnarray}%
where we have used the fact that the row sum of $\mathbf{L}^{h}$\ in {(\ref{eqnLh2}) is
zero. Based on these two properties of }$\mathbf{N}${, one can immediately see that  $\mathbf{I}-\Delta t\mathbf{N}$ is strictly diagonally dominant (SDD) for any $\Delta t$.
Using the Ahlberg-Nilson-Varah bound for
SDD matrices, one obtains%
\begin{equation*}
\Vert (\mathbf{I}-\Delta t\mathbf{N})^{-1}\Vert _{\infty }\leq \frac{1}{%
\min_{i}(|1-\Delta tN_{ii}|-\Delta t\sum _{j\neq i}|N_{ij}|)}=\frac{1}{%
\min_{i}(1-\Delta t\sum _{j}N_{ij})}=1.
\end{equation*}%
}
This completes the first part of Lemma~\ref{lemstabNeubou}.

Next, we prove that $\Vert (\mathbf{I}-\Delta t\mathbf{N})^{-1}\Vert
_{2}\leq 1+C\Delta t$. To obtain this result, we need to show that $%
\mathbf{I}-\Delta t\mathbf{N}^{\top }$ is SDD for sufficiently large $N$ and
sufficiently small $\Delta t$. We can compute
\begin{eqnarray}
&&|1-\Delta tN_{ii}|-\Delta t\sum _{j\neq i}|N_{ji}|  \notag \\
&=&1-\Delta tN_{ii}-\Delta t\sum _{j\neq i}N_{ji}=1-\Delta t\sum
_{j}N_{ji}=1-\Delta t\sum _{j}N_{ij}+\Delta t(\sum _{j}N_{ij}-\sum
_{j}N_{ji})  \notag \\
&=&1+\Delta t[(\mathbf{L}^{I,I}+\mathbf{L}^{I,B}\mathbf{E}^{B,I}+\mathbf{L}{%
^{I,G}\mathbf{\tilde{G}}^{G,I})\mathbf{1}}_{N-B}-(\mathbf{L}^{I,I}+\mathbf{L}%
^{I,B}\mathbf{E}^{B,I}+\mathbf{L}{^{I,G}\mathbf{\tilde{G}}^{G,I})^{\top }%
\mathbf{1}}_{N-B}]_{i}  \notag \\
&=&1+\Delta t\frac{\epsilon ^{-d/2-1}}{m_{0}(N+BK)}[(\mathbf{K}^{I,I}+%
\mathbf{K}^{I,B}\mathbf{E}^{B,I}+\mathbf{K}^{I,G}{\mathbf{\tilde{G}}^{G,I}})%
\mathbf{1}-\left( \mathbf{K}^{I,I}+\mathbf{K}^{I,B}\mathbf{E}^{B,I}+\mathbf{K%
}^{I,G}{\mathbf{\tilde{G}}^{G,I}}\right) ^{\top }\mathbf{1]}_{i}  \notag \\
&=&1+\Delta t\frac{\epsilon ^{-d/2-1}}{m_{0}(N+BK)}\Big[\underbrace{\big(%
\mathbf{K}^{I,I}\mathbf{1-(K}^{I,I})^{\top }\mathbf{1}\big)}_{(I)}+%
\underbrace{\big(\mathbf{K}^{I,B}\mathbf{E}^{B,I}\mathbf{1-(K}^{I,B}\mathbf{E%
}^{B,I}\mathbf{)}^{\top }\mathbf{1}\big)}_{(II)}+\underbrace{\big(\mathbf{K}%
^{I,G}{\mathbf{\tilde{G}}^{G,I}}\mathbf{1}-(\mathbf{K}^{I,G}{\mathbf{\tilde{G%
}}^{G,I}})^{\top }\mathbf{1}\big)}_{(III)}\Big]_{i},  \label{transK}
\end{eqnarray}%
where $\mathbf{K}^{I,I},\mathbf{K}^{I,B},\mathbf{K}^{I,G}$ are the kernels
of $\mathbf{L}^{I,I},\mathbf{L}^{I,B},\mathbf{L}^{I,G}$ in {(\ref{eqnLh2})},
respectively, related to each other with respect to {(\ref{classicDMmatrix}%
). The second line follows from the two properties} 1) and 2) of $\mathbf{N%
}$, the third line follows from the calculation in {(\ref{defNN}), and the
fourth line follows from the construction of DM matrix in (\ref%
{classicDMmatrix}).}

\textbf{Bounding the term (I) in \eqref{transK}.} Let $H_{j}(x_{i}):=%
\epsilon ^{-d/2}K_{\epsilon }(x_{j},x_{i})$ such that,
\begin{equation*}
\mathbb{E}[H_{j}]=\epsilon ^{-d/2}\int_{M}K_{\epsilon
}(x_{j},y)dV_{y}=G_{\epsilon }1(x_{j})=m_0+\epsilon \omega (x_{j})+%
O(\epsilon ^{2}).
\end{equation*}%
It is easy to show that,
\begin{equation*}
\frac{1}{N}\sum_{i=1}^{N-B}H_{j}(x_{i})-\frac{1}{N-1}\sum_{i\neq
j}H_{j}(x_{i})=O((N-B)^{-1}\epsilon ^{-d/2},(N-B)^{-2}).
\end{equation*}%
Letting $\Big|\frac{1}{N-B-1}\sum_{i\neq j}H_{j}(x_{i})-\mathbb{E}[H_{j}]%
\Big|=O(\epsilon ^{2})$, we obtain,
\begin{equation*}
\Big|\frac{1}{N-B}\sum_{i=1}^{N-B}H_{j}(x_{i})-\mathbb{E}[H_{j}]\Big|=%
O(\epsilon ^{2},(N-B)^{-1}\epsilon ^{-d/2},(N-B)^{-2}).
\end{equation*}%
Balancing the first and second error bounds, the bias is given as $\epsilon =%
O((N-B)^{-\frac{2}{4+d}})$.

Let $Y_j(x_i) := H_j(x_i) - \mathbb{E}[H_j]$ and compute the variance,
\begin{eqnarray*}
\mbox{Var}\lbrack Y_j] = \mathbb{E}[H_j^2] - \mathbb{E}[H_j]^2 = \hat{m}_0
 \epsilon^{-d/2} - m_0^2 + O(\epsilon^{1-d/2}) = \hat{m}_0 \epsilon^{-d/2} + O(1),
\end{eqnarray*}
where $\hat{m}_0 = \int_{\mathbb{R}^d} \exp(-\frac{ |z-\sqrt{\epsilon} B(x)
|^2}{2})^2 dz$.

Using the Chernoff bound, we obtain
\begin{equation*}
P\Big(\Big|\frac{1}{N-B-1}\sum_{i\neq j}H_{j}(x_{i})-\mathbb{E}[H_{j}]\Big|%
>c\epsilon ^{2}\Big)=P\Big(\Big|\sum_{i\neq j}Y_{j}(x_{i})\Big|>c\epsilon
^{2}(N-B-1)\Big)\leq \exp \Big(-\frac{c^{2}\epsilon ^{4}(N-B-1)}{4\hat{m}_0%
\epsilon ^{-d/2}}\Big),
\end{equation*}%
where $c,\hat{m}_0=O(1)$. To have an order one exponent (this is
equivalent to balancing square of the bias and variance), we obtain $%
\epsilon =O(N^{-\frac{2}{8+d}})$ which is slower than the bias.

So far, we have shown that
\begin{equation*}
\frac{1}{N-B}\sum_{i=1}^{N-B}H_{j}(x_{i})=\mathbb{E}[H_{j}]+O%
(\epsilon ^{2})=m_0+\epsilon \omega (x_{j})+O(\epsilon
^{2})=m_0+\epsilon \omega (x_{j})+O((N-B)^{-\frac{4}{8+d}})
\end{equation*}%
as $N-B\rightarrow \infty $ in high probability. Repeating the same argument
with $H_{j}^{\ast }(x_{i}):=\epsilon ^{-d/2}K_{\epsilon }(x_{j},x_{i})^{\top
}$, one can conclude also that,
\begin{equation*}
\frac{1}{N-B}\sum_{i=1}^{N-B}H_{j}^{\ast }(x_{i})=m_0+\epsilon \omega
(x_{j})-\epsilon m_0 \mbox{div}b(x_{j})+O((N-B)^{-\frac{4}{8+d}})
\end{equation*}%
Based on these two expansions, we have,
\begin{equation*}
\frac{\epsilon ^{-d/2-1}}{2m_{0}(N-B)}\big((\mathbf{K}^{I,I}\mathbf{)}^{\top
}\mathbf{1}-\mathbf{K}^{I,I}\mathbf{1}\big)_{j}=\frac{\epsilon ^{-1}}{2m_0(N-B)}%
\sum_{i=1}^{N-B}\Big(H_{j}^{\ast }(x_{i})-H_{j}(x_{i})\Big)=-\frac{1}{2}%
\mbox{div}b(x_{j})+O((N-B)^{-\frac{2}{8+d}}).
\end{equation*}%
Thus, the first term can be bounded
by
\begin{equation*}
\frac{\epsilon ^{-d/2-1}}{m_{0}(N+BK)}\left\vert (\mathbf{K}^{I,I}\mathbf{%
1-(K}^{I,I})^{\top }\mathbf{1)}_{i}\right\vert =\frac{N-B}{N+BK}\left\vert %
\mbox{div}b(x_{i})\right\vert +O(N^{-\frac{2}{8+d}})=O(1),
\end{equation*}%
for sufficiently large ${N}$.

\textbf{Bounding the term (II) in \eqref{transK}.} For the second term, $%
\mathbf{K}^{I,B}\in \mathbb{R}^{(N-B)\times B}$ and $\mathbf{E}^{B,I}\in
\mathbb{R}^{B\times (N-B)}$ is a matrix whose $b$th row equals to one on the
column corresponding to the $\tilde{x}_{b,0}$ and zero everywhere else.
Then, we have
\begin{equation*}
(\mathbf{K}^{I,B}\mathbf{E1)}_{i}\mathbf{=(K}^{I,B}\mathbf{1)}_{i}\leq
B=\left\{
\begin{array}{cc}
O(1), & d=1 \\
O(\sqrt{N}), & d=2%
\end{array}%
\right. .
\end{equation*}%
and%
\begin{equation*}
(\mathbf{(K}^{I,B}\mathbf{E)}^{\top }\mathbf{1)}_{i}\mathbf{=((\mathbf{E}}%
^{B,I}\mathbf{\mathbf{)}^{\top }(K}^{I,B})^{\top }\mathbf{1)}_{i}\leq R((%
\mathbf{E}^{B,I}\mathbf{)^{\top }1)}_{i}\leq R=\left\{
\begin{array}{cc}
O(\sqrt{N}), & d=1 \\
O(\sqrt{N}), & d=2%
\end{array}%
\right. ,
\end{equation*}%
where we have used the assumption that each boundary point has at most $R=O(\sqrt{N})$
neighboring interior points for both $d=1$ and $d=2$. Thus, for second term,
\begin{equation*}
\frac{\epsilon ^{-d/2-1}}{m_{0}(N+BK)}\left\vert (\mathbf{K}^{I,B}\mathbf{E}%
^{B,I}\mathbf{1-(K}^{I,B}\mathbf{E}^{B,I}\mathbf{)}^{\top }\mathbf{1)}%
_{i}\right\vert \leq \frac{\epsilon ^{-d/2-1}\left( B+R\right) }{m_{0}(N+BK)}%
=\left\{
\begin{array}{cc}
O(N^{-1/14}), & d=1 \\
O(1), & d=2%
\end{array}%
\right. ,
\end{equation*}%
where we have used,
\begin{equation*}
\epsilon \sim N^{-\frac{2}{d+6}}=
\begin{cases}
N^{-2/7}, & d=1 \\
N^{-1/4}, & d=2
\end{cases},
\end{equation*}
obtained by balancing the two error terms in \eqref{Eqn:cbc}.

\textbf{Bounding the term (III) in \eqref{transK}.} For the third term, $%
\mathbf{K}^{I,G}\in \mathbb{R}^{(N-B)\times BK}$ and ${\mathbf{\tilde{G}}%
^{G,I}}\in \mathbb{R}^{BK\times (N-B)}$ is a matrix whose row corresponding
to $\tilde{x}_{b,k}$\ equals to one on the column corresponding to the $%
\tilde{x}_{b,0}$ and zero everywhere else. Following the same argument
above, we can show that,
\begin{eqnarray*}
(\mathbf{K}^{I,G}{\mathbf{\tilde{G}}^{G,I}}\mathbf{1)}_{i} &\mathbf{=}&%
\mathbf{(K}^{I,G}\mathbf{1)}_{i}\leq BK, \\
((\mathbf{K}^{I,G}{\mathbf{\tilde{G}}^{G,I}})^{\top }\mathbf{1)}_{i} &%
\mathbf{=}&\mathbf{(({\mathbf{\tilde{G}}}}^{G,I}\mathbf{{)}^{\top }(K}%
^{I,G})^{\top }\mathbf{1)}_{i}\leq R(\mathbf{({\mathbf{\tilde{G}}}}^{G,I}%
\mathbf{{)}^{\top }1)}_{i}\leq RK.
\end{eqnarray*}%
The third term can be bounded by,%
\begin{equation*}
\frac{\epsilon ^{-d/2-1}}{m_{0}(N+BK)}\left\vert (\mathbf{K}^{I,G}{\mathbf{%
\tilde{G}}^{G,I}}\mathbf{1}-(\mathbf{K}^{I,G}{\mathbf{\tilde{G}}^{G,I}}%
)^{\top }\mathbf{1)}_{i}\right\vert \leq \frac{\epsilon ^{-d/2-1}\left(
B+R\right) K}{m_{0}(N+BK)}=\left\{
\begin{array}{cc}
O(N^{-1/14}), & d=1 \\
O(1), & d=2%
\end{array}%
\right. .
\end{equation*}

To summarize, there exists a constant $C_{0}$ such that
\begin{equation}
\left\vert \frac{\epsilon ^{-d/2-1}}{m_{0}(N+BK)}[(\mathbf{K}^{I,I}+\mathbf{K%
}^{I,B}\mathbf{E}^{B,I}+\mathbf{K}^{I,G}{\mathbf{\tilde{G}}^{G,I}})\mathbf{1}%
-\left( \mathbf{K}^{I,I}+\mathbf{K}^{I,B}\mathbf{E}^{B,I}+\mathbf{K}^{I,G}{%
\mathbf{\tilde{G}}^{G,I}}\right) ^{\top }\mathbf{1]}_{i}\right\vert \leq
C_{0}, \label{eqnC0e}
\end{equation}%
for sufficiently large $N$. From {(\ref{transK}), we have }%
\begin{equation*}
|1-\Delta tN_{ii}|-\Delta t\sum _{j\neq i}|N_{ji}|\geq 1-C_{0}\Delta t>0,
\end{equation*}%
for sufficiently small $\Delta t$, which means that $\mathbf{I}-\Delta t%
\mathbf{N}^{\top }$ is a SDD matrix. {Using the Ahlberg-Nilson-Varah bound,
we have}
\begin{equation*}
\Vert (\mathbf{I}-\Delta t\mathbf{N})^{-1}\Vert _{1}=\Vert (\mathbf{I}%
-\Delta t\mathbf{N}^{\top })^{-1}\Vert _{\infty }\leq \frac{1}{1-C_{0}\Delta
t}.
\end{equation*}%
Therefore, there exists a constant $C$ such that the spectral norm can be
bounded by
\begin{equation}
\Vert (\mathbf{I}-\Delta t\mathbf{N})^{-1}\Vert _{2}\leq \left( \Vert (%
\mathbf{I}-\Delta t\mathbf{N})^{-1}\Vert _{\infty }\Vert (\mathbf{I}-\Delta t%
\mathbf{N})^{-1}\Vert _{1}\right) ^{1/2}\leq (\frac{1}{1-C_{0}\Delta t}%
)^{1/2}\leq 1+C\Delta t.  \label{eqn:stbN}
\end{equation}

\section{Proof of Lemma~\protect\ref{lemstabDir}}

\label{App:C}

For the stability of the Dirichlet problem \eqref{schemedirichlet},
basically one can follow similar steps as in the above Neumann case.
Following the calculation in \eqref{schemedirichlet}, we have
\begin{equation}
\mathbf{H=\tilde{L}}^{I,I}=\mathbf{L}^{I,I}+\mathbf{L}{^{I,G}\mathbf{G}%
^{G,I}.}  \label{eqnHLii}
\end{equation}%
Now, we show that $\mathbf{I}-\Delta t\mathbf{H}$ is SDD for sufficiently
large $N$ and sufficiently small $\Delta t$. Denote the entries of $\mathbf{H%
}\ $and $\mathbf{L}^{I,I}$\ with $\mathbf{H}=(H_{ij})_{i,j=1}^{N-B}$ and $%
\mathbf{L}^{I,I}=(L_{ij}^{I,I})_{i,j=1}^{N-B}$, where $H_{ii}=L_{ii}^{I,I}<0$
and $L_{ij}^{I,I}>0$ for $j\neq i$\ by noticing the structure of matrices in %
\eqref{eqnHLii}. We can compute
\begin{eqnarray}
&&|1-\Delta tH_{ii}|-\Delta t\sum_{j\neq i}|H_{ij}|=|1-\Delta
tL_{ii}^{I,I}|-\Delta t\sum_{j\neq i}|L_{ij}^{I,I}+(\mathbf{L}{^{I,G}\mathbf{%
G}^{G,I})}_{ij}|\geq |1-\Delta tL_{ii}^{I,I}|-\Delta t\sum_{j\neq
i}(|L_{ij}^{I,I}|+|(\mathbf{L}{^{I,G}\mathbf{G}^{G,I})}_{ij}|)  \notag \\
&=&1-\Delta tL_{ii}^{I,I}-\Delta t\sum_{j\neq i}L_{ij}^{I,I}-\Delta
t\sum_{j\neq i}|(\mathbf{L}{^{I,G}\mathbf{G}^{G,I})}_{ij}|=1-\Delta
t\sum_{j=1}^{N-B}L_{ij}^{I,I}-\Delta t\sum_{j\neq i}|(\mathbf{L}{^{I,G}%
\mathbf{G}^{G,I})}_{ij}|  \notag \\
&\geq &1-\Delta t\sum_{j\neq i}|(\mathbf{L}{^{I,G}\mathbf{G}^{G,I})}%
_{ij}|=1-\Delta t\frac{\epsilon ^{-d/2-1}}{m_{0}(N+BK)}\sum_{j\neq i}|(%
\mathbf{K}{^{I,G}\mathbf{G}^{G,I})}_{ij}|,  \label{eqnsddH}
\end{eqnarray}%
where in the last line, the equality follows from the definition of $\mathbf{%
L}$ in \eqref{classicDMmatrix} and the last inequality follows from the zero
row sum property of $\mathbf{L}$\ in \eqref{eqnLrowsum}, that is, $%
\sum_{j=1}^{N-B}L_{ij}^{I,I}=-\sum_{j=N-B+1}^{N+BK}L_{ij}^{I,I}<0$. By
noticing all entries of $\mathbf{K}{^{I,G}}$\ are nonnegative and all
entries of ${\mathbf{G}^{G,I}}$\ are nonpositive as in \eqref{EqnGGI}, we
can show that
\begin{equation*}
\sum_{j\neq i}|(\mathbf{K}{^{I,G}\mathbf{G}^{G,I})}_{ij}|=-(\mathbf{K}^{I,G}{%
\mathbf{G}^{G,I}}\mathbf{1)}_{i}\mathbf{\leq }K\mathbf{(K}^{I,G}\mathbf{1)}%
_{i}\leq BK^{2}.
\end{equation*}%
Following the same argument in \eqref{eqnC0e}, there exists a constant $C_{0}
$ such that
\begin{equation*}
\frac{\epsilon ^{-d/2-1}}{m_{0}(N+BK)}\sum_{j\neq i}|(\mathbf{K}{^{I,G}%
\mathbf{G}^{G,I})}_{ij}|\leq C_{0}.
\end{equation*}%
Thus, from \eqref{eqnsddH}, {we obtain that }$\mathbf{I}-\Delta t\mathbf{H}$
is SDD for sufficiently large $N$ and sufficiently small $\Delta t$,{\ }%
\begin{equation*}
|1-\Delta tH_{ii}|-\Delta t\sum_{j\neq i}|H_{ij}|\geq 1-C_{0}\Delta t>0.
\end{equation*}%
{Using the Ahlberg-Nilson-Varah bound, we have}
\begin{equation*}
\Vert (\mathbf{I}-\Delta t\mathbf{H})^{-1}\Vert _{\infty }\leq \frac{1}{%
\min_{i}(|1-\Delta tH_{ii}|-\Delta t\sum_{j\neq i}|H_{ij}|)}\leq \frac{1}{%
1-C_{0}\Delta t}\leq 1+C_{1}\Delta t.
\end{equation*}%
Following almost the same steps as in the above Neumann case, we can show
that \ $\Vert (\mathbf{I}-\Delta t\mathbf{H})^{-1}\Vert _{2}\leq
1+C_{2}\Delta t$.


\end{document}